\documentclass{amsart}
\usepackage{url}
\usepackage{amsmath,amsfonts,euscript,amsthm,amssymb,amscd}
\usepackage[T1]{fontenc}
\usepackage{lmodern}
\usepackage{epigraph}
\usepackage[utf8]{inputenc}
\theoremstyle{definition}
\newtheorem{Def}{Definition}[subsection]
\newtheorem{Def-Prop}[Def]{Definition-Proposition}

\newtheorem{remark}[Def]{Remark}
\newtheorem{Prop}[Def]{Proposition}

\newtheorem{Conj}{Conjecture}

\DeclareMathOperator*{\summm}{\oplus}

\newcommand{\hdot}{{\:\raisebox{3pt}{\text{\circle*{1.5}}}}}

\expandafter\ifx\csname ProvidesPackage\endcsname\relax\toks0=\expandafter{%
\fi\ProvidesPackage{diagrams}[2014/12/31 v3.94 Paul Taylor's commutative
diagrams]
\toks0=\bgroup}
\ifx\diagram\isundefined\else\expandafter\endinput
\fi

\edef\cdrestoreat{
\noexpand\catcode`\noexpand\@=\the\catcode`\@
\noexpand\catcode`\noexpand\#=\the\catcode`\#
\noexpand\catcode`\noexpand\$=\the\catcode`\$
\noexpand\catcode`\noexpand\<=\the\catcode`\<
\noexpand\catcode`\noexpand\>=\the\catcode`\>
\noexpand\catcode`\noexpand\:=\the\catcode`\:
\noexpand\catcode`\noexpand\;=\the\catcode`\;
\noexpand\catcode`\noexpand\!=\the\catcode`\!
\noexpand\catcode`\noexpand\?=\the\catcode`\?
\noexpand\catcode`\noexpand\+=\the\catcode'53
}\catcode`\@=11 \catcode`\#=6 \catcode`\<=12 \catcode`\>=12 \catcode'53=12
\catcode`\:=12 \catcode`\;=12 \catcode`\!=12 \catcode`\?=12

\ifx\diagram@help@messages\CD@qK\let\diagram@help@messages y\fi

\def\cdps@Rokicki#1{\special{ps:#1}}\let\cdps@dvips\cdps@Rokicki\let
\cdps@RadicalEye\cdps@Rokicki\let\CD@HB\cdps@Rokicki\let\CD@IK\cdps@Rokicki
\let\CD@HB\cdps@Rokicki
\def\cdps@Bechtolsheim#1{\special{dvitps: Literal "#1"}}%
\let\cdps@dvitps\cdps@Bechtolsheim\let\cdps@IntegratedComputerSystems
\cdps@Bechtolsheim
\def\cdps@Clark#1{\special{dvitops: inline #1}}
\let\cdps@dvitops\cdps@Clark
\let\cdps@OzTeX\empty\let\cdps@oztex\empty\let\cdps@Trevorrow\empty
\def\cdps@Coombes#1{\special{ps-string #1}}

\count@=\year\multiply\count@12 \advance\count@\month
\ifnum\count@>24240 
\ifnum\count@>24243 
\errhelp{You may press RETURN and carry on for the time being.}\errmessage{! remain
compatible with your files. Please get it NOW.}\fi\fi

\def\CD@DE{\global\let}\def\CD@RH{\outer\def}

{\escapechar\m@ne\xdef\CD@o{\string\{}\xdef\CD@yC{\string\}}
\catcode`\&=4 \CD@DE\CD@Q=&\xdef\CD@S{\string\&}
\catcode`\$=3 \CD@DE\CD@k=$\CD@DE\CD@ND=$
\xdef\CD@nC{\string\$}\gdef\CD@LG{$$}
\catcode`\_=8 \CD@DE\CD@lJ=_
\obeylines\catcode`\^=7 \CD@DE\@super=^
\ifnum\newlinechar=10 \gdef\CD@uG{^^J}
\else\ifnum\newlinechar=13 \gdef\CD@uG{^^M}
\else\ifnum\newlinechar=-1 \gdef\CD@uG{^^J}
\else\CD@DE\CD@uG\space\expandafter\message{! input error: \noexpand
\newlinechar\space is ASCII \the\newlinechar, not LF=10 or CR=13.}
\fi\fi\fi}

\mathchardef\lessthan='30474 \mathchardef\greaterthan='30476

\ifx\tenln\CD@qK
\font\tenln=line10\relax
\fi\ifx\tenlnw\CD@qK\ifx\tenln\nullfont\let\tenlnw\nullfont\else
\font\tenlnw=linew10\relax
\fi\fi

\ifx\inputlineno\CD@qK\csname newcount\endcsname\inputlineno\inputlineno\m@ne
\fi

\def\cd@shouldnt#1{\CD@KB{* THIS (#1) SHOULD NEVER HAPPEN! *}}

\def\get@round@pair#1(#2,#3){#1{#2}{#3}}
\def\get@square@arg#1[#2]{#1{#2}}
\def\CD@AE#1{\CD@PK\let\CD@DH\CD@@E\CD@@E#1,],}
\def\CD@m{[}\def\CD@RD{]}\def\commdiag#1{{\let\enddiagram\relax\diagram[]#1%
\enddiagram}}

\def\CD@BF{{\ifx\CD@EH[\aftergroup\get@square@arg\aftergroup\CD@YH\else
\aftergroup\CD@JH\fi}}
\def\CD@CF#1#2{\def\CD@YH{#1}\def\CD@JH{#2}\futurelet\CD@EH\CD@BF}

\def\CD@KK{|}

\def\CD@PB{
\tokcase\CD@DD:\CD@y\break@args;\catcase\@super:\upper@label;\catcase\CD@lJ:%
\lower@label;\tokcase{~}:\middle@label;
\tokcase<:\CD@iF;
\tokcase>:\CD@iI;
\tokcase(:\CD@BC;
\tokcase[:\optional@;
\tokcase.:\CD@JJ;
\catcase\space:\eat@space;\catcase\bgroup:\positional@;\default:\CD@@A
\break@args;\endswitch}

\def\switch@arg{
\catcase\@super:\upper@label;\catcase\CD@lJ:\lower@label;\tokcase[:\optional@
;
\tokcase.:\CD@JJ;
\catcase\space:\eat@space;\catcase\bgroup:\positional@;\tokcase{~}:%
\middle@label;
\default:\CD@y\break@args;\endswitch}

\let\CD@tJ\relax\ifx\protect\CD@qK\let\protect\relax\fi\ifx\AtEndDocument
\CD@qK\def\CD@PG{\CD@gB}\def\CD@GF#1#2{}\else\def\CD@PG#1{\edef\CD@CH{#1}%
\expandafter\CD@oC\CD@CH\CD@OD}\def\CD@oC#1\CD@OD{\AtEndDocument{\typeout{%
\CD@tA: #1}}}\def\CD@GF#1#2{\gdef#1{#2}\AtEndDocument{#1}}\fi\def\CD@ZA#1#2{%
\def#1{\CD@PG{#2\CD@mD\CD@W}\CD@DE#1\relax}}\def\CD@uF#1\repeat{\def\CD@p{#1}%
\CD@OF}\def\CD@OF{\CD@p\relax\expandafter\CD@OF\fi}\def\CD@sF#1\repeat{\def
\CD@q{#1}\CD@PF}\def\CD@PF{\CD@q\relax\expandafter\CD@PF\fi}\def\CD@tF#1%
\repeat{\def\CD@r{#1}\CD@QF}\def\CD@QF{\CD@r\relax\expandafter\CD@QF\fi}\def
\CD@tG#1#2#3{\def#2{\let#1\iftrue}\def#3{\let#1\iffalse}#3}\if y%
\diagram@help@messages\def\CD@rG#1#2{\csname newtoks\endcsname#1#1=%
\expandafter{\csname#2\endcsname}}\else\csname newtoks\endcsname\no@cd@help
\no@cd@help{See the manual}\def\CD@rG#1#2{\let#1\no@cd@help}\fi\chardef\CD@lF
=1 \chardef\CD@lI=2 \chardef\CD@MH=5 \chardef\CD@tH=6 \chardef\CD@sH=7
\chardef\CD@PC=9 \dimendef\CD@hI=2 \dimendef\CD@hF=3 \dimendef\CD@mF=4
\dimendef\CD@mI=5 \dimendef\CD@wJ=6 \dimendef\CD@tI=8 \dimendef\CD@sI=9
\skipdef\CD@uB=1 \skipdef\CD@NF=2 \skipdef\CD@tB=3 \skipdef\CD@ZE=4 \skipdef
\countdef\CD@JC=9 \countdef\CD@gD=8 \countdef\CD@A=7 \def\sdef#1#2{\def#1{#2}%
}\def\CD@L#1{\expandafter\aftergroup\csname#1\endcsname}\def\CD@RC#1{%
\expandafter\def\csname#1\endcsname}\def\CD@sD#1{\expandafter\gdef\csname#1%
\endcsname}\def\CD@vC#1{\expandafter\edef\csname#1\endcsname}\def\CD@nF#1#2{%
\expandafter\let\csname#1\expandafter\endcsname\csname#2\endcsname}\def\CD@EE
#1#2{\expandafter\CD@DE\csname#1\expandafter\endcsname\csname#2\endcsname}%
\def\CD@AK#1{\csname#1\endcsname}\def\CD@XJ#1{\expandafter\show\csname#1%
\endcsname}\def\CD@ZJ#1{\expandafter\showthe\csname#1\endcsname}\def\CD@WJ#1{%
\expandafter\showbox\csname#1\endcsname}\def\CD@tA{Commutative Diagram}\edef
\CD@kH{\string\par}\edef\CD@dC{\string\diagram}\edef\CD@HD{\string\enddiagram
}\edef\CD@EC{\string\\}\def\CD@eF{LaTeX}\ifx\@ignoretrue\CD@qK\expandafter
\CD@tG\csname if@ignore\endcsname\ignore@true\ignore@false\def\@ignoretrue{%
\global\ignore@true}\def\@ignorefalse{\global\ignore@false}\fi

\def\CD@g{{\ifnum0=`}\fi}\def\CD@wC{\ifnum0=`{\fi}}\def\catcase#1:{\ifcat
\noexpand\CD@EH#1\CD@tJ\expandafter\CD@kC\else\expandafter\CD@dJ\fi}\def
\tokcase#1:{\ifx\CD@EH#1\CD@tJ\expandafter\CD@kC\else\expandafter\CD@dJ\fi}%
\def\CD@kC#1;#2\endswitch{#1}\def\CD@dJ#1;{}\let\endswitch\relax\def\default:%
#1;#2\endswitch{#1}\ifx\at@\CD@qK\def\at@{@}\fi\edef\CD@P{\CD@o pt\CD@yC}%
\CD@RC{\CD@P>}#1>#2>{\CD@z\rTo\sp{#1}\sb{#2}\CD@z}\CD@RC{\CD@P<}#1<#2<{\CD@z
\lTo\sp{#1}\sb{#2}\CD@z}\CD@RC{\CD@P)}#1)#2){\CD@z\rTo\sp{#1}\sb{#2}\CD@z}%
\CD@RC{\CD@P(}#1(#2({\CD@z\lTo\sp{#1}\sb{#2}\CD@z}
\def\CD@O{\def\endCD{\enddiagram}\CD@RC{\CD@P A}##1A##2A{\uTo<{##1}>{##2}%
\CD@z\CD@z}\CD@RC{\CD@P V}##1V##2V{\dTo<{##1}>{##2}\CD@z\CD@z}\CD@RC{\CD@P=}{%
\CD@z\hEq\CD@z}\CD@RC{\CD@P\CD@KK}{\vEq\CD@z\CD@z}\CD@RC{\CD@P\string\vert}{%
\vEq\CD@z\CD@z}\CD@RC{\CD@P.}{\CD@z\CD@z}\let\CD@z\CD@Q}\def\CD@IE{\let\tmp
\CD@JE\ifcat A\noexpand\CD@CH\else\ifcat=\noexpand\CD@CH\else\ifcat\relax
\noexpand\CD@CH\else\let\tmp\at@\fi\fi\fi\tmp}\def\CD@JE#1{\CD@nF{tmp}{\CD@P
\string#1}\ifx\tmp\relax\def\tmp{\at@#1}\fi\tmp}\def\CD@z{}\begingroup
\aftergroup\def\aftergroup\CD@T\aftergroup{\aftergroup\def\catcode`\@\active
\aftergroup @\endgroup{\futurelet\CD@CH\CD@IE}}\newcount\CD@uA\newcount\CD@vA
\newcount\CD@wA\newcount\CD@xA\newdimen\CD@OA\newdimen\CD@PA\CD@tG\CD@gE
\CD@@A\CD@y\CD@tG\CD@hE\CD@EA\CD@BA\newdimen\CD@RA\newdimen\CD@SA\newcount
\CD@yA\newcount\CD@zA\newdimen\CD@QA\newbox\CD@DA\CD@tG\CD@lE\CD@dA\CD@bA
\newcount\CD@LH\newcount\CD@TC\def\CD@V#1#2{\ifdim#1<#2\relax#1=#2\relax\fi}%
\def\CD@X#1#2{\ifdim#1>#2\relax#1=#2\relax\fi}\newdimen\CD@XH\CD@XH=1sp
\newdimen\CD@zC\CD@zC\z@\def\CD@cJ{\ifdim\CD@zC=1em\else\CD@nJ\fi}\def\CD@nJ{%
\CD@zC1em\def\CD@NC{\fontdimen8\textfont3 }\CD@@J\CD@NJ\setbox0=\vbox{\CD@t
\noindent\CD@k\null\penalty-9993\null\CD@ND\null\endgraf\setbox0=\lastbox
\unskip\unpenalty\setbox1=\lastbox\global\setbox\CD@IG=\hbox{\unhbox0\unskip
\unskip\unpenalty\setbox0=\lastbox}\global\setbox\CD@KG=\hbox{\unhbox1\unskip
\unpenalty\setbox1=\lastbox}}}\newdimen\CD@@I\CD@@I=1true in \divide\CD@@I300
\def\CD@zH#1{\multiply#1\tw@\advance#1\ifnum#1<\z@-\else+\fi\CD@@I\divide#1%
\tw@\divide#1\CD@@I\multiply#1\CD@@I}\def\MapBreadth{\afterassignment\CD@gI
\CD@LF}\newdimen\CD@LF\newdimen\CD@oI\def\CD@gI{\CD@oI\CD@LF\CD@V\CD@@I{4%
\CD@XH}\CD@X\CD@@I\p@\CD@zH\CD@oI\ifdim\CD@LF>\z@\CD@V\CD@oI\CD@@I\fi\CD@cJ}%
\def\CD@RJ#1{\CD@zD\count@\CD@@I#1\ifnum\count@>\z@\divide\CD@@I\count@\fi
\CD@gI\CD@NJ}\def\CD@NJ{\dimen@\CD@QC\count@\dimen@\divide\count@5\divide
\count@\CD@@I\edef\CD@OC{\the\count@}}\def\CD@AJ{\CD@QJ\z@}\def\CD@QJ#1{%
\CD@tI\axisheight\advance\CD@tI#1\relax\advance\CD@tI-.5\CD@oI\CD@zH\CD@tI
\CD@sI-\CD@tI\advance\CD@tI\CD@LF}\newdimen\CD@DC\CD@DC\z@\newdimen\CD@eJ
\CD@eJ\z@\def\CD@CJ#1{\CD@sI#1\relax\CD@tI\CD@sI\advance\CD@tI\CD@LF\relax}%
\def\horizhtdp{height\CD@tI depth\CD@sI}\def\axisheight{\fontdimen22\the
\textfont\tw@}\def\script@axisheight{\fontdimen22\the\scriptfont\tw@}\def
\ss@axisheight{\fontdimen22\the\scriptscriptfont\tw@}\def\CD@NC{0.4pt}\def
\CD@VK{\fontdimen3\textfont\z@}\def\CD@UK{\fontdimen3\textfont\z@}\newdimen
\PileSpacing\newdimen\CD@nA\CD@nA\z@\def\CD@RG{\ifincommdiag1.3em\else2em\fi}%
\newdimen\CD@YB\def\CellSize{\afterassignment\CD@kB\DiagramCellHeight}%
\newdimen\DiagramCellHeight\DiagramCellHeight-\maxdimen\newdimen
\DiagramCellWidth\DiagramCellWidth-\maxdimen\def\CD@kB{\DiagramCellWidth
\DiagramCellHeight}\def\CD@QC{3em}\newdimen\MapShortFall\def\MapsAbut{%
\MapShortFall\z@\objectheight\z@\objectwidth\z@}\newdimen\CD@iA\CD@iA\z@
\CD@tG\CD@vE\CD@aB\CD@ZB\expandafter\ifx\expandafter\iftrue\csname
ifUglyObsoleteDiagrams\endcsname\CD@ZB\else\CD@aB\fi\CD@nF{%
ifUglyObsoleteDiagrams}{relax}\newif\ifUglyObsoleteDiagrams\def\CD@nK{\CD@aB
\UglyObsoleteDiagramsfalse}\def\CD@oK{\CD@ZB\UglyObsoleteDiagramstrue}\CD@vE
\CD@nK\else\CD@oK\fi\CD@tG\CD@hK\CD@dK\CD@cK\CD@cK\def\CD@sK{\ifx\pdfoutput
\CD@qK\else\ifx\pdfoutput\relax\else\ifnum\pdfoutput>\z@\CD@pK\fi\fi\fi} \def
\CD@pK{\global\CD@dK\global\CD@aB\global\UglyObsoleteDiagramsfalse\global\let
\CD@n\empty\global\let\CD@oK\relax\global\let\CD@pK\relax\global\let\CD@sK
\relax}\def\CD@tK#1{}\ifx\pdfliteral\CD@qK\else\ifx
\pdfliteral\relax\else\let\CD@tK\pdfliteral\fi\fi\ifx\XeTeXrevision\CD@qK
\else\ifx\XeTeXrevision\relax\else\ifdim\XeTeXrevision pt<.996pt \expandafter
\message{! XeTeX version \XeTeXrevision\space does not support PDF literals,
so diagonals will not work!}\else\expandafter\message{RUNNING UNDER XETEX
\XeTeXrevision}\CD@pK\fi\fi\fi\CD@sK\def\newarrowhead{\CD@mG h\CD@BG\CD@GG>}%
\def\newarrowtail{\CD@mG t\CD@BG\CD@GG>}\def\newarrowmiddle{\CD@mG m\CD@BG
\hbox@maths\empty}\def\newarrowfiller{\CD@mG f\CD@bE\CD@MK-}\def\CD@mG#1#2#3#%
4#5#6#7#8#9{\CD@RC{r#1:#5}{#2{#6}}\CD@RC{l#1:#5}{#2{#7}}\CD@RC{d#1:#5}{#3{#8}%
}\CD@RC{u#1:#5}{#3{#9}}\CD@vC{-#1:#5}{\expandafter\noexpand\csname-#1:#4%
\endcsname\noexpand\CD@MC}\CD@vC{+#1:#5}{\expandafter\noexpand\csname+#1:#4%
\endcsname\noexpand\CD@MC}}\CD@ZA\CD@MC{\CD@eF\space diagonals are used unless
PostScript is set}\def\defaultarrowhead#1{\edef\CD@sJ{#1}\CD@@J}\def\CD@@J{%
\CD@IJ\CD@sJ<>ht\CD@IJ\CD@sJ<>th}\def\CD@IJ#1#2#3#4#5{\CD@HJ{r#4}{#3}{l#5}{#2%
}{r#4:#1}\CD@HJ{r#5}{#2}{l#4}{#3}{l#4:#1}\CD@HJ{d#4}{#3}{u#5}{#2}{d#4:#1}%
\CD@HJ{d#5}{#2}{u#4}{#3}{u#4:#1}}\def\CD@HJ#1#2#3#4#5{\begingroup\aftergroup
\CD@GJ\CD@L{#1+:#2}\CD@L{#1:#2}\CD@L{#3:#4}\CD@L{#5}\endgroup}\def\CD@GJ#1#2#%
3#4{\csname newbox\endcsname#1\def#2{\copy#1}\def#3{\copy#1}\setbox#1=\box
\voidb@x}\def\CD@sJ{}\CD@@J\def\CD@GJ#1#2#3#4{\setbox#1=#4}\ifx\tenln
\nullfont\def\CD@sJ{vee}\else\let\CD@sJ\CD@eF\fi\def\CD@xF#1#2#3{\begingroup
\aftergroup\CD@wF\CD@L{#1#2:#3#3}\CD@L{#1#2:#3}\aftergroup\CD@yF\CD@L{#1#2:#3%
-#3}\CD@L{#1#2:#3}\endgroup}\def\CD@wF#1#2{\def#1{\hbox{\rlap{#2}\kern.4%
\CD@zC#2}}}\def\CD@yF#1#2{\def#1{\hbox{\rlap{#2}\kern.4\CD@zC#2\kern-.4\CD@zC
}}}\CD@xF lh>\CD@xF rt>\CD@xF rh<\CD@xF rt<\def\CD@yF#1#2{\def#1{\hbox{\kern-%
.4\CD@zC\rlap{#2}\kern.4\CD@zC#2}}}\CD@xF rh>\CD@xF lh<\CD@xF lt>\CD@xF lt<%
\def\CD@wF#1#2{\def#1{\vbox{\vbox to\z@{#2\vss}\nointerlineskip\kern.4\CD@zC#%
2}}}\def\CD@yF#1#2{\def#1{\vbox{\vbox to\z@{#2\vss}\nointerlineskip\kern.4%
\CD@zC#2\kern-.4\CD@zC}}}\CD@xF uh>\CD@xF dt>\CD@xF dh<\CD@xF dt<\def\CD@yF#1%
#2{\def#1{\vbox{\kern-.4\CD@zC\vbox to\z@{#2\vss}\nointerlineskip\kern.4%
\CD@zC#2}}}\CD@xF dh>\CD@xF ut>\CD@xF uh<\CD@xF ut<\def\CD@BG#1{\hbox{%
\mathsurround\z@\offinterlineskip\CD@k\mkern-1.5mu{#1}\mkern-1.5mu\CD@ND}}%
\def\hbox@maths#1{\hbox{\CD@k#1\CD@ND}}\def\CD@GG#1{\hbox to\CD@LF{\setbox0=%
\hbox{\offinterlineskip\mathsurround\z@\CD@k{#1}\CD@ND}\dimen0.5\wd0\advance
\dimen0-.5\CD@oI\CD@zH{\dimen0}\kern-\dimen0\unhbox0\hss}}\def\CD@sB#1{\hbox
to2\CD@LF{\hss\offinterlineskip\mathsurround\z@\CD@k{#1}\CD@ND\hss}}\def
\CD@vF#1{\hbox{\mathsurround\z@\CD@k{#1}\CD@ND}}\def\CD@bE#1{\hbox{\kern-.15%
\CD@zC\CD@k{#1}\CD@ND\kern-.15\CD@zC}}\def\CD@MK#1{\vbox{\offinterlineskip
\kern-.2ex\CD@GG{#1}\kern-.2ex}}\def\@fillh{\xleaders\vrule\horizhtdp}\def
\@fillv{\xleaders\hrule width\CD@LF}\CD@nF{rf:-}{@fillh}\CD@nF{lf:-}{@fillh}%
\CD@nF{df:-}{@fillv}\CD@nF{uf:-}{@fillv}\CD@nF{rh:}{null}\CD@nF{rm:}{null}%
\CD@nF{rt:}{null}\CD@nF{lh:}{null}\CD@nF{lm:}{null}\CD@nF{lt:}{null}\CD@nF{dh%
:}{null}\CD@nF{dm:}{null}\CD@nF{dt:}{null}\CD@nF{uh:}{null}\CD@nF{um:}{null}%
\CD@nF{ut:}{null}\CD@nF{+h:}{null}\CD@nF{+m:}{null}\CD@nF{+t:}{null}\CD@nF{-h%
:}{null}\CD@nF{-m:}{null}\CD@nF{-t:}{null}\def\CD@@D{\hbox{\vrule height 1pt
depth-1pt width 1pt}}\CD@RC{rf:}{\CD@@D}\CD@nF{lf:}{rf:}\CD@nF{+f:}{rf:}%
\CD@RC{df:}{\CD@@D}\CD@nF{uf:}{df:}\CD@nF{-f:}{df:}\def\CD@BD{\CD@U\null
\CD@@D\null\CD@@D\null}\edef\CD@lG{\string\newarrow}\def\newarrow#1#2#3#4#5#6%
{\begingroup\edef\@name{#1}\edef\CD@oJ{#2}\edef\CD@iD{#3}\edef\CD@QG{#4}\edef
\CD@jD{#5}\edef\CD@LE{#6}\let\CD@HE\CD@sG\let\CD@FK\CD@BH\let\@x\CD@AH\ifx
\CD@oJ\CD@iD\let\CD@oJ\empty\fi\ifx\CD@LE\CD@jD\let\CD@LE\empty\fi\def\CD@LI{%
r}\def\CD@SF{l}\def\CD@IC{d}\def\CD@yJ{u}\def\CD@gH{+}\def\@m{-}\ifx\CD@iD
\CD@jD\ifx\CD@QG\CD@iD\let\CD@QG\empty\fi\ifx\CD@LE\empty\ifx\CD@iD\CD@aE\let
\@x\CD@yG\else\let\@x\CD@zG\fi\fi\else\edef\CD@a{\CD@iD\CD@oJ}\ifx\CD@a\empty
\ifx\CD@QG\CD@jD\let\CD@QG\empty\fi\fi\fi\ifmmode\aftergroup\CD@kG\else\CD@@A
\CD@oB rh{head\space\space}\CD@LE\CD@oB rf{filler}\CD@iD\CD@oB rm{middle}%
\CD@QG\ifx\CD@jD\CD@iD\else\CD@oB rf{filler}\CD@jD\fi\CD@oB rt{tail\space
\space}\CD@oJ\CD@gE\CD@HE\CD@FK\@x\CD@nG l-2+2{lu}{nw}\NorthWest\CD@nG r+2+2{%
ru}{ne}\NorthEast\CD@nG l-2-2{ld}{sw}\SouthWest\CD@nG r+2-2{rd}{se}\SouthEast
\else\aftergroup\CD@b\CD@L{r\@name}\fi\fi\endgroup}\def\CD@sG{\CD@vG\CD@LI
\CD@SF rl\Horizontal@Map}\def\CD@BH{\CD@vG\CD@IC\CD@yJ du\Vertical@Map}\def
\CD@AH{\CD@vG\CD@gH\@m+-\Vector@Map}\def\CD@yG{\CD@vG\CD@gH\@m+-\Slant@Map}%
\def\CD@zG{\CD@vG\CD@gH\@m+-\Slope@Map}\catcode`\/=\active\def\CD@vG#1#2#3#4#%
5{\CD@jG#1#3#5t:\CD@oJ/f:\CD@iD/m:\CD@QG/f:\CD@jD/h:\CD@LE//\CD@jG#2#4#5h:%
\CD@LE/f:\CD@jD/m:\CD@QG/f:\CD@iD/t:\CD@oJ//}\def\CD@jG#1#2#3#4//{\edef\CD@fG
{#2}\aftergroup\sdef\CD@L{#1\@name}\aftergroup{\aftergroup#3\CD@M#4//%
\aftergroup}}\def\CD@M#1/{\edef\CD@EH{#1}\ifx\CD@EH\empty\else\CD@L{\CD@fG#1}%
\expandafter\CD@M\fi}\catcode`\/=12 \def\CD@nG#1#2#3#4#5#6#7#8{\aftergroup
\sdef\CD@L{#6\@name}\aftergroup{\CD@L{#2\@name}\if#2#4\aftergroup\CD@CI\else
\aftergroup\CD@BI\fi\CD@L{#1\@name}%
\aftergroup(\aftergroup#3\aftergroup,\aftergroup#5\aftergroup)\aftergroup}}%
\def\CD@oB#1#2#3#4{\expandafter\ifx\csname#1#2:#4\endcsname\relax\CD@y\CD@gB{%
arrow#3 "#4" undefined}\fi}\CD@rG\CD@VE{All five components must be defined
before an arrow.}\CD@rG\CD@SE{\CD@lG, unlike \string\HorizontalMap, is a
declaration.}\def\CD@b#1{\CD@YA{Arrows \string#1 etc could not be defined}%
\CD@VE}\def\CD@kG{\CD@YA{misplaced \CD@lG}\CD@SE}\def\newdiagramgrid#1#2#3{%
\CD@RC{cdgh@#1}{#2,],}
\CD@RC{cdgv@#1}{#3,],}}
\CD@tG\ifincommdiag\incommdiagtrue\incommdiagfalse\CD@tG\CD@@F\CD@IF\CD@HF
\newcount\CD@VA\CD@VA=0 \def\CD@yH{\CD@VA6 }\def\CD@OB{\CD@VA1 \global\CD@yA1
\CD@DE\CD@YF\empty}\def\CD@YF{}\def\CD@nB#1{\relax\CD@MD\edef\CD@vJ{#1}%
\begingroup\CD@rE\else\ifcase\CD@VA\ifmmode\else\CD@YG\CD@E0\fi\or\CD@cE5\or
\CD@YG\CD@F5\or\CD@YG\CD@B5\or\CD@YG\CD@B5\or\CD@YG\CD@C5\or\CD@cE7\or\CD@YG
\CD@D7\fi\fi\endgroup\xdef\CD@YF{#1}}\def\CD@pB#1#2#3#4#5{\relax\CD@MD\xdef
\CD@vJ{#4}\begingroup\ifnum\CD@VA<#1 \expandafter\CD@cE\ifcase\CD@VA0\or#2\or
#3\else#2\fi\else\ifnum\CD@VA<6 \CD@tJ\CD@YG\CD@B#2\else\CD@YG\CD@G#2\fi\fi
\endgroup\CD@DE\CD@YF\CD@vJ\ifincommdiag\let\CD@ZD#5\else\let\CD@ZD\CD@LK\fi}%
\def\CD@yI{\global\CD@yA=\ifnum\CD@VA<5 1\else2\fi\relax}\def\CD@OI{\CD@VA
\CD@yA}\def\CD@cE#1{\aftergroup\CD@VA\aftergroup#1\aftergroup\relax}\def
\CD@HH{\def\CD@nB##1{\relax}\let\CD@pB\CD@FH\let\CD@yH\relax\let\CD@OB\relax
\let\CD@yI\relax\let\CD@OI\relax}\def\CD@FH#1#2#3#4#5{\ifincommdiag\let\CD@ZD
#5\else\xdef\CD@vJ{#4}\let\CD@ZD\CD@LK\fi}\def\CD@YG#1{\aftergroup#1%
\aftergroup\relax\CD@cE}\def\CD@B{\CD@YE\CD@S\CD@ME\CD@Q}\def\CD@G{\CD@YE{%
\CD@yC\CD@S}\CD@XE\CD@QD\CD@Q}\def\CD@F{\CD@YE{*\CD@S}\CD@RE\clubsuit\CD@Q}%
\def\CD@C{\CD@YE{\CD@S*\CD@S}\CD@RE\CD@Q\clubsuit\CD@Q}\def\CD@D{\CD@YE\CD@EC
\CD@TE\\}\def\CD@E{\CD@YE\CD@nC\CD@QE\CD@k}\def\CD@LK{\CD@YA{\CD@vJ\space
ignored \CD@dH}\CD@WE}\def\CD@FE{}\def\CD@d{\CD@YA{maps must never be enclosed
in braces}\CD@OE}\def\CD@dH{outside diagram}\def\CD@FC{\string\HonV, \string
\VonH\space and \string\HmeetV}\CD@rG\CD@ME{The way that horizontal and
vertical arrows are terminated implicitly means\CD@uG that they cannot be
mixed with each other or with \CD@FC.}\CD@rG\CD@XE{\string\pile\space is for
parallel horizontal arrows; verticals can just be put together in\CD@uG a cell%
. \CD@FC\space are not meaningful in a \string\pile.}\CD@rG\CD@RE{The
horizontal maps must point to an object, not each other (I've put in\CD@uG one
which you're unlikely to want). Use \string\pile\space if you want them
parallel.}\CD@rG\CD@TE{Parallel horizontal arrows must be in separate layers
of a \string\pile.}\CD@rG\CD@QE{Horizontal arrows may be used \CD@dH s, but
must still be in maths.}\CD@rG\CD@WE{Vertical arrows, \CD@FC\space\CD@dH s don%
't know where\CD@uG where to terminate.}\CD@rG\CD@OE{This prevents them from
stretching correctly.}\def\CD@YE#1{\CD@YA{"#1" inserted \ifx\CD@YF\empty
before \CD@vJ\else between \CD@YF\ifx\CD@YF\CD@vJ s\else\space and \CD@vJ\fi
\fi}}\count@=\year\multiply\count@12 \advance\count@\month\ifnum\count@>24247
\expandafter\endinput\fi\def
\Horizontal@Map{\CD@nB{horizontal map}\CD@LC\CD@TJ\CD@qD}\def\CD@TJ{\CD@GB-%
9999 \let\CD@ZD\CD@XD\ifincommdiag\else\CD@cJ\ifinpile\else\skip2\z@ plus 1.5%
\CD@VK minus .5\CD@UK\skip4\skip2 \fi\fi\let\CD@kD\@fillh\CD@nF{fill@dot}{rf:%
.}}\def\Vector@Map{\CD@HK4}\def\Slant@Map{\CD@HK{\CD@EF255\else6\fi}}\def
\Slope@Map{\CD@HK\CD@OC}\def\CD@HK#1#2#3#4#5#6{\CD@LC\def\CD@WK{2}\def\CD@aK{%
2}\def\CD@ZK{1}\def\CD@bK{1}\let\Horizontal@Map\CD@nI\def\CD@OG{#1}\def\CD@NI
{\CD@U#2#3#4#5#6}}\def\CD@nI{\CD@TJ\CD@JB\let\CD@ZD\CD@TD\CD@qD}\CD@tG\CD@pE
\CD@rA\CD@qA\CD@rA\def\cds@missives{\CD@rA}\def\CD@TD{\CD@vE\let\CD@OG\CD@OC
\CD@x\CD@zE\CD@WF\fi\setbox0\hbox{\incommdiagfalse\CD@HI}\CD@pE\CD@aD\else
\global\CD@YC\CD@bD\fi\ifvoid6 \ifvoid7 \CD@eE\fi\fi\CD@zE\else\CD@BD\global
\CD@YC\let\CD@CG\CD@IH\CD@YD\fi\else\CD@NI\CD@MI\global\CD@YC\CD@YD\fi}\def
\CD@LC{\begingroup\dimen1=\MapShortFall\dimen2=\CD@RG\dimen5=\MapShortFall
\setbox3=\box\voidb@x\setbox6=\box\voidb@x\setbox7=\box\voidb@x\CD@pD
\mathsurround\z@\skip2\z@ plus1fill minus 1000pt\skip4\skip2 \CD@TB}\CD@tG
\CD@tE\CD@UB\CD@TB\def\CD@U#1#2#3#4#5{\let\CD@oJ#1\let\CD@iD#2\let\CD@QG#3%
\let\CD@jD#4\let\CD@LE#5\CD@TB\ifx\CD@iD\CD@jD\CD@UB\fi}\def\CD@qD#1#2#3#4#5{%
\CD@U#1#2#3#4#5\CD@tD}\def\Vertical@Map{\CD@pB433{vertical map}\CD@cD\CD@LC
\CD@GB-9995 \let\CD@kD\@fillv\CD@nF{fill@dot}{df:.}\CD@qD}\def\break@args{%
\def\CD@tD{\CD@ZD}\CD@ZD\endgroup\aftergroup\CD@FE}\def\CD@MJ{\setbox1=\CD@oJ
\setbox5=\CD@LE\ifvoid3 \ifx\CD@QG\null\else\setbox3=\CD@QG\fi\fi\CD@@G2%
\CD@iD\CD@@G4\CD@jD}\def\CD@pF#1{\ifvoid1\else\CD@oF1#1\fi\ifvoid2\else\CD@oF
2#1\fi\ifvoid3\else\CD@oF3#1\fi\ifvoid4\else\CD@oF4#1\fi\ifvoid5\else\CD@oF5#%
1\fi} \def\CD@oF#1#2{\setbox#1\vbox{\offinterlineskip\box#1\dimen@\prevdepth
\advance\dimen@-#2\relax\setbox0\null\dp0\dimen@\ht0-\dimen@\box0}}\def\CD@@G
#1#2{\ifx#2\CD@kD\setbox#1=\box\voidb@x\else\setbox#1=#2\def#2{\xleaders\box#%
1}\fi}\CD@ZA\CD@BK{\string\HorizontalMap, \string\VerticalMap\space and
\string\DiagonalMap\CD@uG are obsolete - use \CD@lG\space to pre-define maps}%
\def\HorizontalMap#1#2#3#4#5{\CD@BK\CD@nB{old horizontal map}\CD@LC\CD@TJ\def
\CD@oJ{\CD@UH{#1}}\CD@SH\CD@iD{#2}\def\CD@QG{\CD@UH{#3}}\CD@SH\CD@jD{#4}\def
\CD@LE{\CD@UH{#5}}\CD@tD}\def\VerticalMap#1#2#3#4#5{\CD@BK\CD@pB433{vertical
map}\CD@cD\CD@LC\CD@GB-9995 \let\CD@kD\@fillv\def\CD@oJ{\CD@GG{#1}}\CD@VH
\CD@iD{#2}\def\CD@QG{\CD@GG{#3}}\CD@VH\CD@jD{#4}\def\CD@LE{\CD@GG{#5}}\CD@tD}%
\def\DiagonalMap#1#2#3#4#5{\CD@BK\CD@LC\def\CD@OG{4}\let\CD@kD\CD@qK\let
\CD@ZD\CD@YD\def\CD@WK{2}\def\CD@aK{2}\def\CD@ZK{1}\def\CD@bK{1}\def\CD@QG{%
\CD@vF{#3}}\ifPositiveGradient\let\mv\raise\def\CD@oJ{\CD@vF{#5}}\def\CD@iD{%
\CD@vF{#4}}\def\CD@jD{\CD@vF{#2}}\def\CD@LE{\CD@vF{#1}}\else\let\mv\lower\def
\CD@oJ{\CD@vF{#1}}\def\CD@iD{\CD@vF{#2}}\def\CD@jD{\CD@vF{#4}}\def\CD@LE{%
\CD@vF{#5}}\fi\CD@tD}\def\CD@aE{-}\def\CD@AD{\empty}\def\CD@SH{\CD@EG\CD@bE
\CD@aE\@fillh}\def\CD@VH{\CD@EG\CD@MK\CD@KK\@fillv}\def\CD@EG#1#2#3#4#5{\def
\CD@CH{#5}\ifx\CD@CH#2\let#4#3\else\let#4\null\ifx\CD@CH\empty\else\ifx\CD@CH
\CD@AD\else\let#4\CD@CH\fi\fi\fi}\def\CD@UH#1{\hbox{\mathsurround\z@
\offinterlineskip\def\CD@CH{#1}\ifx\CD@CH\empty\else\ifx\CD@CH\CD@AD\else
\CD@k\mkern-1.5mu{\CD@CH}\mkern-1.5mu\CD@ND\fi\fi}}\def\CD@yD#1#2{\setbox#1=%
\hbox\bgroup\setbox0=\hbox{\CD@k\labelstyle()\CD@ND}
\setbox1=\null\ht1\ht0\dp1\dp0\box1 \kern.1\CD@zC\CD@k\bgroup\labelstyle
\aftergroup\CD@LD\CD@xD}\def\CD@LD{\CD@ND\kern.1\CD@zC\egroup\CD@tD}\def
\CD@xD{\futurelet\CD@EH\CD@mJ}\def\CD@mJ{
\catcase\bgroup:\CD@v;\catcase\egroup:\missing@label;\catcase\space:\CD@TF;%
\tokcase[:\CD@XF;
\default:\CD@zJ;\endswitch}\def\CD@v{\let\CD@MD\CD@c\let\CD@CH}\def\CD@zJ#1{%
\let\CD@UF\egroup{\let\actually@braces@missing@around@macro@in@label\CD@ZH
\let\CD@MD\CD@xC\let\CD@UF\CD@VF#1%
\actually@braces@missing@around@macro@in@label}\CD@UF}\def
\actually@braces@missing@around@macro@in@label{\let\CD@CH=}\def\missing@label
{\egroup\CD@YA{missing label}\CD@PE}\def\CD@xC{\egroup\missing@label}\outer
\def\CD@ZH{}\def\CD@UF{}\def\CD@VF{\CD@wC\CD@UF}\def\CD@MD{}\def\CD@XF{\let
\CD@N\CD@xD\get@square@arg\CD@AE}\CD@rG\CD@PE{The text which has just been
read is not allowed within map labels.}\def\CD@c{\egroup\CD@YA{missing \CD@yC
\space inserted after label}\CD@PE}\def\upper@label{\CD@oD\CD@yD6}\def
\lower@label{\def\positional@{\CD@@A\break@args}\CD@yD7}\def\middle@label{%
\CD@yD3}\CD@tG\CD@yE\CD@pD\CD@oD\def\CD@iF{\ifPositiveGradient\CD@tJ
\expandafter\upper@label\else\expandafter\lower@label\fi}\def\CD@iI{%
\ifPositiveGradient\CD@tJ\expandafter\lower@label\else\expandafter
\upper@label\fi}\def\positional@{\CD@gB{labels as positional arguments are
obsolete}\CD@yE\CD@tJ\expandafter\upper@label\else\expandafter\lower@label\fi
-}\def\CD@tD{\futurelet\CD@EH\switch@arg}\def\eat@space{\afterassignment
\CD@tD\let\CD@EH= }\def\CD@TF{\afterassignment\CD@xD\let\CD@EH= }\def\CD@BC{%
\get@round@pair\CD@uD}\def\CD@uD#1#2{\def\CD@WK{#1}\def\CD@aK{#2}\CD@tD}\def
\optional@{\let\CD@N\CD@tD\get@square@arg\CD@AE}\def\CD@JJ.{\CD@sC\CD@tD}\def
\CD@sC{\let\CD@iD\fill@dot\let\CD@jD\fill@dot\def\CD@MI{\let\CD@iD\dfdot\let
\CD@jD\dfdot}}\def\CD@MI{}\def\CD@@E#1,{\CD@nH#1,\begingroup\ifx\@name\CD@RD
\CD@FF\aftergroup\CD@e\fi\aftergroup\CD@jC\else\expandafter\def\expandafter
\CD@RF\expandafter{\csname\@name\endcsname}\expandafter\CD@vD\CD@RF\CD@KD\ifx
\CD@RF\empty\aftergroup\CD@pC\expandafter\aftergroup\csname\CD@FB\@name
\endcsname\expandafter\aftergroup\csname\CD@FB @\@name\endcsname\else\gdef
\CD@GE{#1}\CD@gB{\string\relax\space inserted before `[\CD@GE'}\message{(I was
trying to read this as a \CD@tA\ option.)}\aftergroup\CD@H\fi\fi\endgroup}%
\def\CD@vD#1#2\CD@KD{\def\CD@RF{#2}}\def\CD@jC{\let\CD@CH\CD@N\let\CD@N\relax
\CD@CH}\def\CD@H#1],{
\CD@jC\relax\def\CD@RF{#1}\ifx\CD@RF\empty\def\CD@RF{[\CD@GE]}%
\else\def\CD@RF{[\CD@GE,#1]}
\fi\CD@RF}\def\CD@pC#1#2{\ifx#2\CD@qK\ifx#1\CD@qK\CD@gB{option `\@name'
undefined}\else#1\fi\else\CD@FF\expandafter#2\CD@GK\CD@PK\else\CD@QK\fi\fi
\CD@DH}\CD@tG\CD@FF\CD@QK\CD@PK\def\CD@nH#1,{\CD@FF\ifx\CD@GK\CD@qK\CD@e\else
\expandafter\CD@oH\CD@GK,#1,(,),(,)[]%
\fi\fi\CD@FF\else\CD@mH#1==,\fi}\def\CD@e{\CD@gB{option `\@name' needs (x,y)
value}\CD@PK\let\@name\empty}\def\CD@mH#1=#2=#3,{\def\@name{#1}\def\CD@GK{#2}%
\def\CD@RF{#3}\ifx\CD@RF\empty\let\CD@GK\CD@qK\fi}%
\def\CD@oH#1(#2,#3)#4,(#5,#6)#7[]{\def\CD@GK{{#2}{#3}}\def\CD@RF{#1#4#5#6}%
\ifx\CD@RF\empty\def\CD@RF{#7}\ifx\CD@RF\empty\CD@e\fi\else\CD@e\fi}\def
\CD@FB{cds@}\let\CD@N\relax\def\CD@zD#1{\ifx\CD@GK\CD@qK\CD@gB{option `\@name
' needs a value}\else#1\CD@GK\relax\fi}\def\CD@BE#1#2{\ifx\CD@GK\CD@qK#1#2%
\relax\else#1\CD@GK\relax\fi}\def\cds@@showpair#1#2{\message{x=#1,y=#2}}\def
\cds@@diagonalbase#1#2{\edef\CD@ZK{#1}\edef\CD@bK{#2}}\def\CD@DI#1{\def\CD@CH
{#1}\CD@nF{@x}{cdps@#1}\ifx\CD@CH\empty\CD@f\CD@CH{cannot be used}\else\ifx
\CD@CH\relax\CD@f\CD@CH{unknown}\else\let\CD@IK\@x\fi\fi}\def\CD@f#1#2{\CD@gB
{PostScript translator `#1' #2}}\def\CD@PH{}\def\CD@PJ{\CD@fA\edef\CD@PH{%
\noexpand\CD@KB{\@name\space ignored within maths}}}\def\diagramstyle{\CD@cJ
\let\CD@N\relax\CD@CF\CD@AE\CD@AE}\CD@tG\CD@sE
\CD@SB\CD@RB\CD@tG\CD@qE\CD@EB\CD@DB\CD@tG\CD@oE\CD@pA\CD@oA\CD@tG\CD@iE
\CD@HA\CD@GA\CD@HA\CD@tG\CD@jE\CD@JA\CD@IA\CD@tG\CD@kE\CD@LA\CD@KA\CD@tG
\CD@EF\CD@DK\CD@CK\CD@tG\CD@rE\CD@JB\CD@IB\CD@tG\CD@mE\CD@gA\CD@fA\CD@tG
\CD@nE\CD@kA\CD@jA\CD@tG\CD@AF\CD@iG\CD@hG\CD@RC{cds@ }{}\CD@RC{cds@}{}\CD@RC
{cds@1em}{\CellSize1\CD@zC}\CD@RC{cds@1.5em}{\CellSize1.5\CD@zC}\CD@RC{cds@2%
em}{\CellSize2\CD@zC}\CD@RC{cds@2.5em}{\CellSize2.5\CD@zC}\CD@RC{cds@3em}{%
\CellSize3\CD@zC}\CD@RC{cds@3.5em}{\CellSize3.5\CD@zC}\CD@RC{cds@4em}{%
\CellSize4\CD@zC}\CD@RC{cds@4.5em}{\CellSize4.5\CD@zC}\CD@RC{cds@5em}{%
\CellSize5\CD@zC}\CD@RC{cds@6em}{\CellSize6\CD@zC}\CD@RC{cds@7em}{\CellSize7%
\CD@zC}\CD@RC{cds@8em}{\CellSize8\CD@zC}\def\cds@abut{\MapsAbut\dimen1\z@
\dimen5\z@}\def\cds@alignlabels{\CD@IA\CD@KA}\def\cds@amstex{\ifincommdiag
\CD@O\else\def\CD{\diagram[amstex]}
\fi\CD@T\catcode`\@\active}\def\cds@b{\let\CD@dB\CD@bB}\def\cds@balance{\let
\CD@hA\CD@AA}\let\cds@bottom\cds@b\def\cds@center{\cds@vcentre\cds@nobalance}%
\let\cds@centre\cds@center\def\cds@centerdisplay{\CD@HA\CD@PJ\cds@balance}%
\let\cds@centredisplay\cds@centerdisplay\def\cds@crab{\CD@BE\CD@DC{.5%
\PileSpacing}}\CD@RC{cds@crab-}{\CD@DC-.5\PileSpacing}\CD@RC{cds@crab+}{%
\CD@DC.5\PileSpacing}\CD@RC{cds@crab++}{\CD@DC1.5\PileSpacing}\CD@RC{cds@crab%
--}{\CD@DC-1.5\PileSpacing}\def\cds@defaultsize{\CD@BE{\let\CD@QC}{3em}\CD@NJ
}\def\cds@displayoneliner{\CD@DB}\let\cds@dotted\CD@sC\def\cds@dpi{\CD@RJ{1%
truein}}\def\cds@dpm{\CD@RJ{100truecm}}\let\CD@XA\CD@qK\def\cds@eqno{\let
\CD@XA\CD@GK\let\CD@EJ\empty}\def\cds@fixed{\CD@qA}\CD@tG\CD@fE\CD@J\CD@I\def
\cds@flushleft{\CD@I\CD@GA\CD@PJ\cds@nobalance\CD@BE\CD@nA\CD@nA}\def\cds@gap
{\CD@AJ\setbox3=\null\ht3=\CD@tI\dp3=\CD@sI\CD@BE{\wd3=}\MapShortFall} \def
\cds@grid{\ifx\CD@GK\CD@qK\let\h@grid\relax\let\v@grid\relax\else\CD@nF{%
h@grid}{cdgh@\CD@GK}\CD@nF{v@grid}{cdgv@\CD@GK}\ifx\h@grid\relax\CD@gB{%
unknown grid `\CD@GK'}\else\CD@WB\fi\fi}\let\h@grid\relax\let\v@grid\relax
\def\cds@gridx{\ifx\CD@GK\CD@qK\else\cds@grid\fi\let\CD@CH\h@grid\let\h@grid
\v@grid\let\v@grid\CD@CH}\def\cds@h{\CD@zD\DiagramCellHeight}\def\cds@hcenter
{\let\CD@hA\CD@aA}\let\cds@hcentre\cds@hcenter\def\cds@heads{\CD@BE{\let
\CD@sJ}\CD@sJ\CD@@J\CD@vE\else\ifx\CD@sJ\CD@eF\else\CD@MC\fi\fi}\let
\cds@height\cds@h\let\cds@hmiddle\cds@balance\def\cds@htriangleheight{\CD@BE
\DiagramCellHeight\DiagramCellHeight\DiagramCellWidth1.73205%
\DiagramCellHeight}\def\cds@htrianglewidth{\CD@BE\DiagramCellWidth
\DiagramCellWidth\DiagramCellHeight.57735\DiagramCellWidth}\CD@tG\CD@zE\CD@eE
\CD@dE\CD@eE\def\cds@hug{\CD@eE} \def\cds@inline{\CD@gA\let\CD@PH\empty}\def
\cds@inlineoneliner{\CD@EB}\CD@RC{cds@l>}{\CD@zD{\let\CD@RG}\dimen2=\CD@RG}%
\def\cds@labelstyle{\CD@zD{\let\labelstyle}}\def\cds@landscape{\CD@kA}\def
\cds@large{\CellSize5\CD@zC}\let\CD@EJ\empty\def\CD@FJ{\refstepcounter{%
equation}\def\CD@XA{\hbox{\@eqnnum}}}\def\cds@LaTeXeqno{\let\CD@EJ\CD@FJ}\def
\cds@lefteqno{\CD@pA}\def\cds@leftflush{\cds@flushleft\CD@J}\def
\cds@leftshortfall{\CD@zD{\dimen1 }}\def\cds@lowershortfall{%
\ifPositiveGradient\cds@leftshortfall\else\cds@rightshortfall\fi}\def
\cds@loose{\CD@VB}\def\cds@midhshaft{\CD@JA}\def\cds@midshaft{\CD@JA}\def
\cds@midvshaft{\CD@LA}\def\cds@moreoptions{\CD@@A}\let\cds@nobalance
\cds@hcenter\def\cds@nohcheck{\CD@HH}\def\cds@nohug{\CD@dE} \def
\cds@nooptions{\def\CD@aC{\CD@WD}}\let\cds@noorigin\cds@nobalance\def
\cds@nopixel{\CD@@I4\CD@XH\CD@cJ}\def\cds@UO{\CD@oK\global\let\CD@n\empty}%
\def\cds@UglyObsolete{\cds@UO\let\cds@PS\empty}\def\CD@rK#1{\CD@gB{option `#1%
' renamed as `UglyObsolete'}}\def\cds@noPostScript{\CD@rK{noPostScript}}\def
\cds@noPS{\CD@rK{noPostScript}}\def\cds@notextflow{\CD@RB}\def\cds@noTPIC{%
\CD@CK}\def\cds@objectstyle{\CD@zD{\let\objectstyle}}\def\cds@origin{\let
\CD@hA\CD@iB}\def\cds@p{\CD@zD\PileSpacing}\let\cds@pilespacing\cds@p\def
\cds@pixelsize{\CD@zD\CD@@I\CD@gI}\def\cds@portrait{\CD@jA}\def
\cds@PostScript{\CD@nK\global\let\CD@n\empty\CD@BE\CD@DI\empty}\def\cds@PS{%
\CD@nK\global\let\CD@n\empty}\CD@GF\CD@n{\typeout{\CD@tA: try the PostScript
option for better results}}\def\cds@repositionpullbacks{\let\make@pbk\CD@fH
\let\CD@qH\CD@pH}\def\cds@righteqno{\CD@oA}\def\cds@rightshortfall{\CD@zD{%
\dimen5 }}\def\cds@ruleaxis{\CD@zD{\let\axisheight}}\def\cds@cmex{\let\CD@GG
\CD@sB\let\CD@QJ\CD@CJ}\def\cds@s{\cds@height\DiagramCellWidth
\DiagramCellHeight}\def\cds@scriptlabels{\let\labelstyle\scriptstyle}\def
\cds@shortfall{\CD@zD\MapShortFall\dimen1\MapShortFall\dimen5\MapShortFall}%
\def\cds@showfirstpass{\CD@BE{\let\CD@nD}\z@}\def\cds@silent{\def\CD@KB##1{}%
\def\CD@gB##1{}}\let\cds@size\cds@s\def\cds@small{\CellSize2\CD@zC}\def
\cds@snake{\CD@BE\CD@eJ\z@}\def\cds@t{\let\CD@dB\CD@fB}\def\cds@textflow{%
\CD@SB\CD@PJ}\def\cds@thick{\let\CD@rF\tenlnw\CD@LF\CD@NC\CD@BE\MapBreadth{2%
\CD@LF}\CD@@J}\def\cds@thin{\let\CD@rF\tenln\CD@BE\MapBreadth{\CD@NC}\CD@@J}%
\def\cds@tight{\CD@WB}\let\cds@top\cds@t\def\cds@TPIC{\CD@DK}\def
\cds@uppershortfall{\ifPositiveGradient\cds@rightshortfall\else
\cds@leftshortfall\fi}\def\cds@vcenter{\let\CD@dB\CD@cB}\let\cds@vcentre
\cds@vcenter\def\cds@vtriangleheight{\CD@BE\DiagramCellHeight
\DiagramCellHeight\DiagramCellWidth.577035\DiagramCellHeight}\def
\cds@vtrianglewidth{\CD@BE\DiagramCellWidth\DiagramCellWidth
\DiagramCellHeight1.73205\DiagramCellWidth}\def\cds@vmiddle{\let\CD@dB\CD@eB}%
\def\cds@w{\CD@zD\DiagramCellWidth}\let\cds@width\cds@w\def\diagram{\relax
\protect\CD@bC}\def\enddiagram{\protect\CD@SG}\def\CD@bC{\CD@g\CD@uI
\incommdiagtrue\edef\CD@wI{\the\CD@NB}\global\CD@NB\z@\boxmaxdepth\maxdimen
\everycr{}\CD@sK\everymath{}\everyhbox{}\ifx\pdfsyncstop\CD@qK\else
\pdfsyncstop\fi\CD@aC}\def\CD@aC{\CD@y\let\CD@N\CD@ZC\CD@CF\CD@AE\CD@WD}\def
\CD@ZC{\CD@gE\expandafter\CD@aC\else\expandafter\CD@WD\fi}\def\CD@WD{\let
\CD@EH\relax\CD@nE\CD@vE\else\CD@hK\else\CD@KB{landscape ignored without
PostScript}\CD@jA\fi\fi\fi\CD@EJ\setbox2=\vbox\bgroup\CD@JF\CD@VD}\def\CD@cH{%
\CD@nE\CD@fB\else\CD@dB\fi\CD@hA\nointerlineskip\setbox0=\null\ht0-\CD@pI\dp0%
\CD@pI\wd0\CD@kI\box0 \global\CD@QA\CD@kF\global\CD@yA\CD@XB\ifx\CD@NK\CD@qK
\global\CD@RA\CD@kF\else\global\CD@RA\CD@NK\fi\egroup\CD@zF\CD@nE\setbox2=%
\hbox to\dp2{\vrule height\wd2 depth\CD@QA width\z@\global\CD@QA\ht2\ht2\z@
\dp2\z@\wd2\z@\CD@hK\CD@tK{q 0 1 -1 0 0 0 cm}\else\global\CD@iG\CD@IK{0 1
bturn}\fi\box2\CD@gK\hss}\CD@DB\fi\ifnum\CD@yA=1 \else\CD@DB\fi\global
\@ignorefalse\CD@mE\leavevmode\fi\ifvmode\CD@TA\else\ifmmode\CD@PH\CD@GI\else
\CD@qE\CD@gA\fi\ifinner\CD@gA\fi\CD@mE\CD@GI\else\CD@sE\CD@QB\else\CD@TA\fi
\fi\fi\fi\CD@dD}\def\CD@dD{\global\CD@NB\CD@wI\relax\CD@xE\global\CD@ID\else
\aftergroup\CD@mC\fi\if@ignore\aftergroup\ignorespaces\fi\CD@wC\ignorespaces}%
\def\CD@fB{\advance\CD@pI\dimen1\relax}\def\CD@eB{\advance\CD@pI.5\dimen1%
\relax}\def\CD@bB{}\def\CD@cB{\CD@fB\advance\CD@pI\CD@YB\divide\CD@pI2
\advance\CD@pI-\axisheight\relax}\def\CD@aA{}\def\CD@iB{\CD@kF\z@}\def\CD@AA{%
\ifdim\dimen2>\CD@kF\CD@kF\dimen2 \else\dimen2\CD@kF\CD@kI\dimen0 \advance
\CD@kI\dimen2 \fi}\def\CD@QB{\skip0\z@\relax\loop\skip1\lastskip\ifdim\skip1>%
\z@\unskip\advance\skip0\skip1 \repeat\vadjust{\prevdepth\dp\strutbox\penalty
\predisplaypenalty\vskip\abovedisplayskip\CD@UA\penalty\postdisplaypenalty
\vskip\belowdisplayskip}\ifdim\skip0=\z@\else\hskip\skip0 \global\@ignoretrue
\fi}\def\CD@TA{\CD@LG\kern-\displayindent\CD@UA\CD@LG\global\@ignoretrue}\def
\CD@UA{\hbox to\hsize{\CD@fE\ifdim\CD@RA=\z@\else\advance\CD@QA-\CD@RA\setbox
2=\hbox{\kern\CD@RA\box2}\fi\fi\setbox1=\hbox{\ifx\CD@XA\CD@qK\else\CD@k
\CD@XA\CD@ND\fi}\CD@oE\CD@iE\else\advance\CD@QA\wd1 \fi\wd1\z@\box1 \fi\dimen
0\wd2 \advance\dimen0\wd1 \advance\dimen0-\hsize\ifdim\dimen0>-\CD@nA\CD@HA
\fi\advance\dimen0\CD@QA\ifdim\dimen0>\z@\CD@KB{wider than the page by \the
\dimen0 }\CD@HA\fi\CD@iE\hss\else\CD@V\CD@QA\CD@nA\fi\CD@GI\hss\kern-\wd1\box
1 }}\def\CD@GI{\CD@AF\CD@@F\else\CD@SC\global\CD@hG\fi\fi\kern\CD@QA\box2 }%
\CD@tG\CD@wE\CD@YC\CD@XC\def\CD@JF{\CD@cJ\ifdim\DiagramCellHeight=-\maxdimen
\DiagramCellHeight\CD@QC\fi\ifdim\DiagramCellWidth=-\maxdimen
\DiagramCellWidth\CD@QC\fi\global\CD@XC\CD@IF\let\CD@FE\empty\let\CD@z\CD@Q
\let\overprint\CD@eH\let\CD@s\CD@rJ\let\enddiagram\CD@ED\let\\\CD@cC\let\par
\CD@jH\let\CD@MD\empty\let\switch@arg\CD@PB\let\shift\CD@iA\baselineskip
\DiagramCellHeight\lineskip\z@\lineskiplimit\z@\mathsurround\z@\tabskip\z@
\CD@OB}\def\CD@VD{\penalty-123 \begingroup\CD@jA\aftergroup\CD@K\halign
\bgroup\global\advance\CD@NB1 \vadjust{\penalty1}\global\CD@FA\z@\CD@OB\CD@j#%
#\CD@DD\CD@Q\CD@Q\CD@OI\CD@j##\CD@DD\cr}\def\CD@ED{\CD@MD\CD@GD\crcr\egroup
\global\CD@JD\endgroup}\def\CD@j{\global\advance\CD@FA1 \futurelet\CD@EH\CD@i
}\def\CD@i{\ifx\CD@EH\CD@DD\CD@tJ\hskip1sp plus 1fil \relax\let\CD@DD\relax
\CD@vI\else\hfil\CD@k\objectstyle\let\CD@FE\CD@d\fi}\def\CD@DD{\CD@MD\relax
\CD@yI\CD@vI\global\CD@QA\CD@iA\penalty-9993 \CD@ND\hfil\null\kern-2\CD@QA
\null}\def\CD@cC{\cr}\def\across#1{\span\omit\mscount=#1 \global\advance
\CD@FA\mscount\global\advance\CD@FA\m@ne\CD@sF\ifnum\mscount>2 \CD@fJ\repeat
\ignorespaces}\def\CD@fJ{\relax\span\omit\advance\mscount\m@ne}\def\CD@qJ{%
\ifincommdiag\ifx\CD@iD\@fillh\ifx\CD@jD\@fillh\ifdim\dimen3>\z@\else\ifdim
\dimen2>93\CD@@I\ifdim\dimen2>18\p@\ifdim\CD@LF>\z@\count@\CD@bJ\advance
\count@\m@ne\ifnum\count@<\z@\count@20\let\CD@aJ\CD@uJ\fi\xdef\CD@bJ{\the
\count@}\fi\fi\fi\fi\fi\fi\fi}\def\CD@cG#1{\vrule\horizhtdp width#1\dimen@
\kern2\dimen@}\def\CD@uJ{\rlap{\dimen@\CD@@I\CD@V\dimen@{.182\p@}\CD@zH
\dimen@\advance\CD@tI\dimen@\CD@cG0\CD@cG0\CD@cG2\CD@cG6\CD@cG6\CD@cG2\CD@cG0%
\CD@cG0\CD@cG2\CD@cG6\CD@cG0\CD@cG0\CD@cG2\CD@cG2\CD@cG6\CD@cG0\CD@cG0\CD@cG2%
\CD@cG6\CD@cG2\CD@cG2\CD@cG0\CD@cG0}}\def\CD@bJ{10}\def\CD@aJ{}\def\CD@XD{%
\CD@gE\CD@TB\fi\CD@x\CD@WF\CD@HI}\def\CD@x{\CD@QJ\CD@DC\CD@MJ\ifdim\CD@DC=\z@
\else\CD@pF\CD@DC\fi\ifvoid3 \setbox3=\null\ht3\CD@tI\dp3\CD@sI\else\CD@V{\ht
3}\CD@tI\CD@V{\dp3}\CD@sI\fi\dimen3=.5\wd3 \ifdim\dimen3=\z@\CD@tE\else\dimen
3-\CD@XH\fi\else\CD@TB\fi\CD@V{\dimen2}{\wd7}\CD@V{\dimen2}{\wd6}\CD@qJ
\advance\dimen2-2\dimen3 \dimen4.5\dimen2 \dimen2\dimen4 \advance\dimen2%
\CD@eJ\advance\dimen4-\CD@eJ\advance\dimen2-\wd1 \advance\dimen4-\wd5 \ifvoid
2 \else\CD@V{\ht3}{\ht2}\CD@V{\dp3}{\dp2}\CD@V{\dimen2}{\wd2}\fi\ifvoid4 \else
\CD@V{\ht3}{\ht4}\CD@V{\dp3}{\dp4}\CD@V{\dimen4}{\wd4}\fi\advance\skip2\dimen
2 \advance\skip4\dimen4 \CD@tE\advance\skip2\skip4 \dimen0\dimen5 \advance
\dimen0\wd5 \skip3-\skip4 \advance\skip3-\dimen0 \let\CD@jD\empty\else\skip3%
\z@\relax\dimen0\z@\fi}\def\CD@WF{\offinterlineskip\lineskip.2\CD@zC\ifvoid6
\else\setbox3=\vbox{\hbox to2\dimen3{\hss\box6\hss}\box3}\fi\ifvoid7 \else
\setbox3=\vtop{\box3 \hbox to2\dimen3{\hss\box7\hss}}\fi}\def\CD@HI{\kern
\dimen1 \box1 \CD@aJ\CD@iD\hskip\skip2 \kern\dimen0 \ifincommdiag\CD@jE
\penalty1\fi\kern\dimen3 \penalty\CD@GB\hskip\skip3 \null\kern-\dimen3 \else
\hskip\skip3 \fi\box3 \CD@jD\hskip\skip4 \box5 \kern\dimen5}\def\CD@MF{\ifnum
\CD@LH>\CD@TC\CD@V{\dimen1}\objectheight\CD@V{\dimen5}\objectheight\else\CD@V
{\dimen1}\objectwidth\CD@V{\dimen5}\objectwidth\fi}\def\CD@Y{\begingroup
\ifdim\dimen7=\z@\kern\dimen8 \else\ifdim\dimen6=\z@\kern\dimen9 \else\dimen5%
\dimen6 \dimen6\dimen9 \CD@KJ\dimen4\dimen2 \CD@dG{\dimen4}\dimen6\dimen5
\dimen7\dimen8 \CD@KJ\CD@iC{\dimen2}\ifdim\dimen2<\dimen4 \kern\dimen2 \else
\kern\dimen4 \fi\fi\fi\endgroup}\def\CD@jJ{\CD@JI\setbox\z@\hbox{\lower
\axisheight\hbox to\dimen2{\CD@DF\ifPositiveGradient\dimen8\ht\CD@MH\dimen9%
\CD@mI\else\dimen8\dp3 \dimen9\dimen1 \fi\else\dimen8 \ifPositiveGradient
\objectheight\else\z@\fi\dimen9\objectwidth\fi\advance\dimen8
\ifPositiveGradient-\fi\axisheight\CD@Y\unhbox\z@\CD@DF\ifPositiveGradient
\dimen8\dp3 \dimen9\dimen0 \else\dimen8\ht\CD@MH\dimen9\CD@mF\fi\else\dimen8
\ifPositiveGradient\z@\else\objectheight\fi\dimen9\objectwidth\fi\advance
\dimen8 \ifPositiveGradient\else-\fi\axisheight\CD@Y}}}\def\CD@bD{\dimen6
\CD@aK\DiagramCellHeight\dimen7 \CD@WK\DiagramCellWidth\CD@jJ
\ifPositiveGradient\advance\dimen7-\CD@ZK\DiagramCellWidth\else\dimen7 \CD@ZK
\DiagramCellWidth\dimen6\z@\fi\advance\dimen6-\CD@bK\DiagramCellHeight\CD@mK
\setbox0=\rlap{\kern-\dimen7 \lower\dimen6\box\z@}\ht0\z@\dp0\z@\raise
\axisheight\box0 }\def\CD@mK{\setbox0\hbox{\ht\z@\z@\dp\z@\z@\wd\z@\z@\CD@hK
\expandafter\CD@tK{q \CD@eK\space\CD@lK\space\CD@kK\space\CD@eK\space0 0 cm}%
\else\global\CD@iG\CD@eD{\the\CD@TC\space\ifPositiveGradient\else-\fi\the
\CD@LH\space bturn}\fi\box\z@\CD@gK}}\def\CD@vB{\advance\CD@hF-\CD@mI\CD@wJ
\CD@hF\advance\CD@wJ\CD@hI\ifvoid\CD@sH\ifdim\CD@wJ<.1em\ifnum\CD@gD=\@m\else
\CD@aG h\CD@wJ<.1em:objects overprint:\CD@FA\CD@gD\fi\fi\else\ifhbox\CD@sH
\CD@SK\else\CD@TK\fi\advance\CD@wJ\CD@mI\CD@bH{-\CD@mI}{\box\CD@sH}{\CD@wJ}%
\z@\fi\CD@hF-\CD@mF\CD@gD\CD@FA\CD@hI\z@}\def\CD@SK{\setbox\CD@sH=\hbox{%
\unhbox\CD@sH\unskip\unpenalty}\setbox\CD@tH=\hbox{\unhbox\CD@tH\unskip
\unpenalty}\setbox\CD@sH=\hbox to\CD@wJ{\CD@OA\wd\CD@sH\unhbox\CD@sH\CD@PA
\lastkern\unkern\ifdim\CD@PA=\z@\CD@UB\advance\CD@OA-\wd\CD@tH\else\CD@TB\fi
\ifnum\lastpenalty=\z@\else\CD@JA\unpenalty\fi\kern\CD@PA\ifdim\CD@hF<\CD@OA
\CD@JA\fi\ifdim\CD@hI<\wd\CD@tH\CD@JA\fi\CD@jE\CD@hI\CD@wJ\advance\CD@hI-%
\CD@OA\advance\CD@hI\wd\CD@tH\ifdim\CD@hI<2\wd\CD@tH\CD@aG h\CD@hI<2\wd\CD@tH
:arrow too short:\CD@FA\CD@gD\fi\divide\CD@hI\tw@\CD@hF\CD@wJ\advance\CD@hF-%
\CD@hI\fi\CD@tE\kern-\CD@hI\fi\hbox to\CD@hI{\unhbox\CD@tH}\CD@HG}}\CD@tG
\ifinpile\inpiletrue\inpilefalse\inpilefalse\def\pile{\protect\CD@UJ\protect
\CD@uH}\def\CD@uH#1{\CD@l#1\CD@QD}\def\CD@UJ{\CD@nB{pile}\setbox0=\vtop
\bgroup\aftergroup\CD@lD\inpiletrue\let\CD@FE\empty\let\pile\CD@KF\let\CD@QD
\CD@PD\let\CD@GD\CD@FD\CD@yH\baselineskip.5\PileSpacing\lineskip.1\CD@zC
\relax\lineskiplimit\lineskip\mathsurround\z@\tabskip\z@\let\\\CD@wH}\def
\CD@l{\CD@DE\CD@YF\empty\halign\bgroup\hfil\CD@k\let\CD@FE\CD@d\let\\\CD@vH##%
\CD@MD\CD@ND\hfil\CD@Q\CD@R##\cr}\CD@rG\CD@NE{pile only allows one column.}%
\CD@rG\CD@UE{you left it out!}\def\CD@R{\CD@QD\CD@Q\relax\CD@YA{missing \CD@yC
\space inserted after \string\pile}\CD@NE}\def\CD@PD{\CD@MD\crcr\egroup
\egroup}\def\CD@GD{\CD@MD}\def\CD@FD{\CD@MD\relax\CD@QD\CD@YA{missing \CD@yC
\space inserted between \string\pile\space and \CD@HD}\CD@UE}\def\CD@QD{%
\CD@MD}\def\CD@lD{\vbox{\dimen1\dp0 \unvbox0 \setbox0=\lastbox\advance\dimen1%
\dp0 \nointerlineskip\box0 \nointerlineskip\setbox0=\null\dp0.5\dimen1\ht0-%
\dp0 \box0}\ifincommdiag\CD@tJ\penalty-9998 \fi\xdef\CD@YF{pile}}\def\CD@vH{%
\cr}\def\CD@wH{\noalign{\skip@\prevdepth\advance\skip@-\baselineskip
\prevdepth\skip@}}\def\CD@KF#1{#1}\def\CD@TK{\setbox\CD@sH=\vbox{\unvbox
\CD@sH\setbox1=\lastbox\setbox0=\box\voidb@x\CD@tF\setbox\CD@sH=\lastbox
\ifhbox\CD@sH\CD@rC\repeat\unvbox0 \global\CD@QA\CD@ZE}\CD@ZE\CD@QA}\def
\CD@rC{\CD@jE\setbox\CD@sH=\hbox{\unhbox\CD@sH\unskip\setbox\CD@sH=\lastbox
\unskip\unhbox\CD@sH}\ifdim\CD@wJ<\wd\CD@sH\CD@aG h\CD@wJ<\wd\CD@sH:arrow in
pile too short:\CD@FA\CD@gD\else\setbox\CD@sH=\hbox to\CD@wJ{\unhbox\CD@sH}%
\fi\else\CD@gJ\fi\setbox0=\vbox{\box\CD@sH\nointerlineskip\ifvoid0 \CD@tJ\box
1 \else\vskip\skip0 \unvbox0 \fi}\skip0=\lastskip\unskip}\def\CD@gJ{\penalty7
\noindent\unhbox\CD@sH\unskip\setbox\CD@sH=\lastbox\unskip\unhbox\CD@sH
\endgraf\setbox\CD@tH=\lastbox\unskip\setbox\CD@tH=\hbox{\CD@JG\unhbox\CD@tH
\unskip\unskip\unpenalty}\ifcase\prevgraf\cd@shouldnt P\or\ifdim\CD@wJ<\wd
\CD@tH\CD@aG h\CD@wJ<\wd\CD@sH:object in pile too wide:\CD@FA\CD@gD\setbox
\CD@sH=\hbox to\CD@wJ{\hss\unhbox\CD@tH\hss}\else\setbox\CD@sH=\hbox to\CD@wJ
{\hss\kern\CD@hF\unhbox\CD@tH\kern\CD@hI\hss}\fi\or\setbox\CD@sH=\lastbox
\unskip\CD@SK\else\cd@shouldnt Q\fi\unskip\unpenalty}\def\CD@cD{\CD@MJ\ifvoid
3 \setbox3=\null\ht3\axisheight\dp3-\ht3 \dimen3.5\CD@LF\else\dimen4\dp3
\dimen3.5\wd3 \setbox3=\CD@GG{\box3}\dp3\dimen4 \ifdim\ht3=-\dp3 \else\CD@TB
\fi\fi\dimen0\dimen3 \advance\dimen0-.5\CD@LF\setbox0\null\ht0\ht3\dp0\dp3\wd
0\wd3 \ifvoid6\else\setbox6\hbox{\unhbox6\kern\dimen0\kern2pt}\dimen0\wd6 \fi
\ifvoid7\else\setbox7\hbox{\kern2pt\kern\dimen3\unhbox7}\dimen3\wd7 \fi
\setbox3\hbox{\ifvoid6\else\kern-\dimen0\unhbox6\fi\unhbox3 \ifvoid7\else
\unhbox7\kern-\dimen3\fi}\ht3\ht0\dp3\dp0\wd3\wd0 \CD@tE\dimen4=\ht\CD@MH
\advance\dimen4\dp5 \advance\dimen4\dimen1 \let\CD@jD\empty\else\dimen4\ht3
\fi\setbox0\null\ht0\dimen4 \offinterlineskip\setbox8=\vbox spread2ex{\kern
\dimen5 \box1 \CD@iD\vfill\CD@tE\else\kern\CD@eJ\fi\box0}\ht8=\z@\setbox9=%
\vtop spread2ex{\kern-\ht3 \kern-\CD@eJ\box3 \CD@jD\vfill\box5 \kern\dimen1}%
\dp9=\z@\hskip\dimen0plus.0001fil \box9 \kern-\CD@LF\box8 \CD@kE\penalty2 \fi
\CD@tE\penalty1 \fi\kern\PileSpacing\kern-\PileSpacing\kern-.5\CD@LF\penalty
\CD@GB\null\kern\dimen3}\def\CD@cI{\ifhbox\CD@VA\CD@KB{clashing verticals}\ht
\CD@MH.5\dp\CD@VA\dp\CD@MH-\ht5 \CD@yB\ht\CD@MH\z@\dp\CD@MH\z@\fi\dimen1\dp
\CD@VA\CD@xA\prevgraf\unvbox\CD@VA\CD@wA\lastpenalty\unpenalty\setbox\CD@VA=%
\null\setbox\CD@lI=\hbox{\CD@JG\unhbox\CD@lI\unskip\unpenalty\dimen0\lastkern
\unkern\unkern\unkern\kern\dimen0 \CD@HG}\setbox\CD@lF=\hbox{\unhbox\CD@lF
\dimen0\lastkern\unkern\unkern\global\CD@QA\lastkern\unkern\kern\dimen0 }%
\CD@tF\ifnum\CD@xA>4 \CD@zI\repeat\unskip\unskip\advance\CD@mF.5\wd\CD@VA
\advance\CD@mF\wd\CD@lF\advance\CD@mI.5\wd\CD@VA\advance\CD@mI\wd\CD@lI\ifnum
\CD@FA=\CD@lA\CD@OA.5\wd\CD@VA\edef\CD@NK{\the\CD@OA}\fi\setbox\CD@VA=\hbox{%
\kern-\CD@mF\box\CD@lF\unhbox\CD@VA\box\CD@lI\kern-\CD@mI\penalty\CD@wA
\penalty\CD@NB}\ht\CD@VA\dimen1 \dp\CD@VA\z@\wd\CD@VA\CD@tB\CD@vB}\def\CD@zI{%
\ifdim\wd\CD@lF<\CD@QA\setbox\CD@lF=\hbox to\CD@QA{\CD@JG\unhbox\CD@lF}\fi
\advance\CD@xA\m@ne\setbox\CD@VA=\hbox{\box\CD@lF\unhbox\CD@VA}\unskip\setbox
\CD@lF=\lastbox\setbox\CD@lF=\hbox{\unhbox\CD@lF\unskip\unpenalty\dimen0%
\lastkern\unkern\unkern\global\CD@QA\lastkern\unkern\kern\dimen0 }}\def\CD@yB
{\dimen1\dp\CD@VA\ifhbox\CD@VA\CD@xB\else\CD@zB\fi\setbox\CD@VA=\vbox{%
\penalty\CD@NB}\dp\CD@VA-\dp\CD@MH\wd\CD@VA\CD@tB}\def\CD@zB{\unvbox\CD@VA
\CD@wA\lastpenalty\unpenalty\ifdim\dimen1<\ht\CD@MH\CD@aG v\dimen1<\ht\CD@MH:%
rows overprint:\CD@NB\CD@wA\fi}\def\CD@xB{\dimen0=\ht\CD@VA\setbox\CD@VA=%
\hbox\bgroup\advance\dimen1-\ht\CD@MH\unhbox\CD@VA\CD@xA\lastpenalty
\unpenalty\CD@wA\lastpenalty\unpenalty\global\CD@RA-\lastkern\unkern\setbox0=%
\lastbox\CD@tF\setbox\CD@VA=\hbox{\box0\unhbox\CD@VA}\setbox0=\lastbox\ifhbox
0 \CD@kJ\repeat\global\CD@SA-\lastkern\unkern\global\CD@QA\CD@JK\unhbox\CD@VA
\egroup\CD@JK\CD@QA\CD@bH{\CD@SA}{\box\CD@VA}{\CD@RA}{\dimen1}}\def\CD@kJ{%
\setbox0=\hbox to\wd0\bgroup\unhbox0 \unskip\unpenalty\dimen7\lastkern\unkern
\ifnum\lastpenalty=1 \unpenalty\CD@UB\else\CD@TB\fi\ifnum\lastpenalty=2
\unpenalty\dimen2.5\dimen0\advance\dimen2-.5\dimen1\advance\dimen2-%
\axisheight\else\dimen2\z@\fi\setbox0=\lastbox\dimen6\lastkern\unkern\setbox1%
=\lastbox\setbox0=\vbox{\unvbox0 \CD@tE\kern-\dimen1 \else\ifdim\dimen2=\z@
\else\kern\dimen2 \fi\fi}\ifdim\dimen0<\ht0 \CD@aG v\dimen0<\ht0:upper part of
vertical too short:{\CD@tE\CD@NB\else\CD@wA\fi}\CD@xA\else\setbox0=\vbox to%
\dimen0{\unvbox0}\fi\setbox1=\vtop{\unvbox1}\ifdim\dimen1<\dp1 \CD@aG v\dimen
1<\dp1:lower part of vertical too short:\CD@NB\CD@wA\else\setbox1=\vtop to%
\dimen1{\ifdim\dimen2=\z@\else\kern-\dimen2 \fi\unvbox1 }\fi\box1 \kern\dimen
6 \box0 \kern\dimen7 \CD@HG\global\CD@QA\CD@JK\egroup\CD@JK\CD@QA\relax}%
\countdef\CD@u=14 \newcount\CD@CA\newcount\CD@XB\newcount\CD@NB\let\CD@LB
\insc@unt\newcount\CD@FA\newcount\CD@lA\let\CD@mA\CD@XB\newcount\CD@MB\CD@tG
\CD@DF\CD@bI\CD@aI\CD@aI\def\CD@nD{-1}\def\CD@K{\ifnum\CD@nD<\z@\else
\begingroup\scrollmode\showboxdepth\CD@nD\showboxbreadth\maxdimen\showlists
\endgroup\fi\CD@bI\CD@zF\CD@CA=\CD@u\advance\CD@CA1 \CD@XB=\CD@CA\ifnum\CD@NB
=1 \CD@JA\fi\advance\CD@XB\CD@NB\dimen1\z@\skip0\z@\count@=\insc@unt\advance
\count@\CD@u\divide\count@2 \ifnum\CD@XB>\count@\CD@KB{The diagram has too
many rows! It can't be reformatted.}\else\CD@NG\CD@WI\fi\CD@cH}\def\CD@NG{%
\CD@NB\CD@CA\CD@uF\ifnum\CD@NB<\CD@XB\setbox\CD@NB\box\voidb@x\advance\CD@NB1%
\relax\repeat\CD@NB\CD@CA\skip\z@\z@\CD@uF\CD@GB\lastpenalty\unpenalty\ifnum
\CD@GB>\z@\CD@KE\repeat\ifnum\CD@GB=-123 \CD@tJ\unpenalty\else\cd@shouldnt D%
\fi\ifx\v@grid\relax\else\CD@NB\CD@XB\advance\CD@NB\m@ne\expandafter\CD@VJ
\v@grid\fi\CD@MB\CD@mA\CD@tB\z@\CD@XG\ifx\h@grid\relax\else\expandafter\CD@LJ
\h@grid\fi\count@\CD@XB\advance\count@\m@ne\CD@YB\ht\count@}\def\CD@KE{%
\ifcase\CD@GB\or\CD@MG\else\CD@uA-\lastpenalty\unpenalty\CD@vA\lastpenalty
\unpenalty\setbox0=\lastbox\CD@WG\fi\CD@wD}\def\CD@wD{\skip1\lastskip\unskip
\advance\skip0\skip1 \ifdim\skip1=\z@\else\expandafter\CD@wD\fi}\def\CD@MG{%
\setbox0=\lastbox\CD@pI\dp0 \advance\CD@pI\skip\z@\skip\z@\z@\advance\CD@NF
\CD@pI\CD@uE\ifnum\CD@NB>\CD@CA\CD@NF\DiagramCellHeight\CD@pI\CD@NF\advance
\CD@pI-\CD@qI\fi\fi\CD@qI\ht0 \CD@NF\CD@qI\setbox\CD@NB\hbox{\unhbox\CD@NB
\unhbox0}\dp\CD@NB\CD@pI\ht\CD@NB\CD@qI\advance\CD@NB1 }\def\CD@WG{\ifnum
\CD@uA<\z@\advance\CD@uA\CD@XB\ifnum\CD@uA<\CD@CA\CD@UG\else\CD@OA\dp\CD@uA
\CD@PA\ht\CD@uA\setbox\CD@uA\hbox{\box\z@\penalty\CD@vA\penalty\CD@GB\unhbox
\CD@uA}\dp\CD@uA\CD@OA\ht\CD@uA\CD@PA\fi\else\CD@UG\fi}\def\CD@UG{\CD@KB{%
diagonal goes outside diagram (lost)}}\def\CD@fI{\advance\CD@uA\CD@XB\ifnum
\CD@uA<\CD@CA\CD@UG\else\ifnum\CD@uA=\CD@NB\CD@VG\else\ifnum\CD@uA>\CD@NB
\cd@shouldnt M\else\CD@OA\dp\CD@uA\CD@PA\ht\CD@uA\setbox\CD@uA\hbox{\box\z@
\penalty\CD@vA\penalty\CD@GB\unhbox\CD@uA}\dp\CD@uA\CD@OA\ht\CD@uA\CD@PA\fi
\fi\fi}\def\CD@WI{\CD@t\CD@AJ\setbox\CD@PC=\hbox{\CD@k A\@super f\CD@lJ f%
\CD@ND}\CD@ZE\z@\CD@JK\z@\CD@kI\z@\CD@kF\z@\CD@NB=\CD@XB\CD@NF\z@\CD@uB\z@
\CD@uF\ifnum\CD@NB>\CD@CA\advance\CD@NB\m@ne\CD@qI\ht\CD@NB\CD@pI\dp\CD@NB
\advance\CD@NF\CD@qI\CD@rI\advance\CD@uB\CD@NF\CD@KC\CD@ZI\CD@w\ht\CD@NB
\CD@qI\dp\CD@NB\CD@pI\nointerlineskip\box\CD@NB\CD@NF\CD@pI\setbox\CD@NB\null
\ht\CD@NB\CD@uB\repeat\CD@wB\nointerlineskip\box\CD@NB\CD@gG\CD@ZE
\DiagramCellWidth{width}\CD@gG\CD@JK\DiagramCellHeight{height}\CD@VA\CD@LB
\advance\CD@VA-\CD@lA\advance\CD@VA\m@ne\advance\CD@VA\CD@mA\dimen0\wd\CD@VA
\CD@tI\axisheight\dimen1\CD@uB\advance\dimen1-\CD@YB\dimen2\CD@kI\advance
\dimen2-\dimen0 \advance\CD@XB-\CD@CA\advance\CD@LB-\CD@lA}\count@\year
\multiply\count@12 \advance\count@\month\ifnum\count@>24254 \loop\iftrue
\repeat\fi\def\CD@wB{\CD@qI-\CD@NF\CD@pI\CD@NF
\setbox\CD@MH=\null\dp\CD@MH\CD@NF\ht\CD@MH-\CD@NF\CD@mF\z@\CD@mI\z@\CD@lA
\CD@LB\advance\CD@lA-\CD@MB\advance\CD@lA\CD@mA\CD@FA\CD@LB\CD@VA\CD@MB\CD@sF
\ifnum\CD@FA>\CD@lA\advance\CD@FA\m@ne\advance\CD@VA\m@ne\CD@tB\wd\CD@VA
\setbox\CD@FA=\box\voidb@x\CD@yB\repeat\CD@w\ht\CD@NB\CD@qI\dp\CD@NB\CD@pI}%
\def\CD@gG#1#2#3{\ifdim#1>.01\CD@zC\CD@PA#2\relax\advance\CD@PA#1\relax
\advance\CD@PA.99\CD@zC\count@\CD@PA\divide\count@\CD@zC\CD@KB{increase cell #%
3 to \the\count@ em}\fi}\def\CD@rI{\CD@FA=\CD@LB\penalty4 \noindent\unhbox
\CD@NB\CD@sF\unskip\setbox0=\lastbox\ifhbox0 \advance\CD@FA\m@ne\setbox\CD@FA
\hbox to\wd0{\null\penalty-9990\null\unhbox0}\repeat\CD@lA\CD@FA\advance
\CD@FA\CD@MB\advance\CD@FA-\CD@mA\ifnum\CD@FA<\CD@LB\count@\CD@FA\advance
\count@\m@ne\dimen0=\wd\count@\count@\CD@MB\advance\count@\m@ne\CD@tB\wd
\count@\CD@sF\ifnum\CD@FA<\CD@LB\CD@DJ\CD@XG\dimen0\wd\CD@FA\advance\CD@FA1
\repeat\fi\CD@sF\CD@GB\lastpenalty\unpenalty\ifnum\CD@GB>\z@\CD@vA
\lastpenalty\unpenalty\CD@VG\repeat\endgraf\unskip\ifnum\lastpenalty=4
\unpenalty\else\cd@shouldnt S\fi}\def\CD@VG{\advance\CD@vA\CD@lA\advance
\CD@vA\m@ne\setbox0=\lastbox\ifnum\CD@vA<\CD@LB\setbox\CD@vA\hbox{\box0%
\penalty\CD@GB\unhbox\CD@vA}\else\CD@UG\fi}\def\CD@bG{}\CD@tG\CD@uE\CD@WB
\CD@VB\def\CD@DJ{\advance\dimen0\wd\CD@FA\divide\dimen0\tw@\CD@uE\dimen0%
\DiagramCellWidth\else\CD@V{\dimen0}\DiagramCellWidth\CD@pJ\fi\advance\CD@tB
\dimen0 }\def\CD@XG{\setbox\CD@MB=\vbox{}\dp\CD@MB=\CD@uB\wd\CD@MB\CD@tB
\advance\CD@MB1 }\def\CD@LJ#1,{\def\CD@GK{#1}\ifx\CD@GK\CD@RD\else\advance
\CD@tB\CD@GK\DiagramCellWidth\CD@XG\expandafter\CD@LJ\fi}\def\CD@VJ#1,{\def
\CD@GK{#1}\ifx\CD@GK\CD@RD\else\ifnum\CD@NB>\CD@CA\CD@NF\CD@GK
\DiagramCellHeight\advance\CD@NF-\dp\CD@NB\advance\CD@NB\m@ne\ht\CD@NB\CD@NF
\fi\expandafter\CD@VJ\fi}\def\CD@pJ{\CD@wE\CD@OA\dimen0 \advance\CD@OA-%
\DiagramCellWidth\ifdim\CD@OA>2\MapShortFall\CD@KB{badly drawn diagonals (see
manual)}\let\CD@pJ\empty\fi\else\let\CD@pJ\empty\fi}\def\CD@KC{\CD@VA\CD@mA
\CD@sF\ifnum\CD@VA<\CD@MB\dimen0\dp\CD@VA\advance\dimen0\CD@NF\dp\CD@VA\dimen
0 \advance\CD@VA1 \repeat}\def\CD@bH#1#2#3#4{\ifnum\CD@FA<\CD@LB\CD@OA=#1%
\relax\setbox\CD@FA=\hbox{\setbox0=#2\dimen7=#4\relax\dimen8=#3\relax\ifhbox
\CD@FA\unhbox\CD@FA\advance\CD@OA-\lastkern\unkern\fi\ifdim\CD@OA=\z@\else
\kern-\CD@OA\fi\raise\dimen7\box0 \kern-\dimen8 }\ifnum\CD@FA=\CD@lA\CD@V
\CD@kF\CD@OA\fi\else\cd@shouldnt O\fi}\def\CD@w{\setbox\CD@NB=\hbox{\CD@FA
\CD@lA\CD@VA\CD@mA\CD@PA\z@\relax\CD@sF\ifnum\CD@FA<\CD@LB\CD@tB\wd\CD@VA
\relax\CD@eI\advance\CD@FA1 \advance\CD@VA1 \repeat}\CD@V\CD@kI{\wd\CD@NB}\wd
\CD@NB\z@}\def\CD@eI{\ifhbox\CD@FA\CD@OA\CD@tB\relax\advance\CD@OA-\CD@PA
\relax\ifdim\CD@OA=\z@\else\kern\CD@OA\fi\CD@PA\CD@tB\advance\CD@PA\wd\CD@FA
\relax\unhbox\CD@FA\advance\CD@PA-\lastkern\unkern\fi}\def\CD@ZI{\setbox
\CD@sH=\box\voidb@x\CD@VA=\CD@MB\CD@FA\CD@LB\CD@VA\CD@mA\advance\CD@VA\CD@FA
\advance\CD@VA-\CD@lA\advance\CD@VA\m@ne\CD@tB\wd\CD@VA\count@\CD@LB\advance
\count@\m@ne\CD@hF.5\wd\count@\advance\CD@hF\CD@tB\CD@A\m@ne\CD@gD\@m\CD@sF
\ifnum\CD@FA>\CD@lA\advance\CD@FA\m@ne\advance\CD@hF-\CD@tB\CD@PI\wd\CD@VA
\CD@tB\advance\CD@hF\CD@tB\advance\CD@VA\m@ne\CD@tB\wd\CD@VA\repeat\CD@mF
\CD@kF\CD@mI-\CD@mF\CD@vB}\newcount\CD@GB\def\CD@s{}\def\CD@t{\mathsurround
\z@\hsize\z@\rightskip\z@ plus1fil minus\maxdimen\parfillskip\z@\linepenalty
9000 \looseness0 \hfuzz\maxdimen\hbadness10000 \clubpenalty0 \widowpenalty0
\displaywidowpenalty0 \interlinepenalty0 \predisplaypenalty0
\postdisplaypenalty0 \interdisplaylinepenalty0 \interfootnotelinepenalty0
\floatingpenalty0 \brokenpenalty0 \everypar{}\leftskip\z@\parskip\z@
\parindent\z@\pretolerance10000 \tolerance10000 \hyphenpenalty10000
\exhyphenpenalty10000 \binoppenalty10000 \relpenalty10000 \adjdemerits0
\doublehyphendemerits0 \finalhyphendemerits0 \baselineskip\z@\CD@IA\prevdepth
\z@}\newbox\CD@KG\newbox\CD@IG\def\CD@JG{\unhcopy\CD@KG}\def\CD@HG{\unhcopy
\CD@IG}\def\CD@iJ{\hbox{}\penalty1\nointerlineskip}\def\CD@PI{\penalty5
\noindent\setbox\CD@MH=\null\CD@mF\z@\CD@mI\z@\ifnum\CD@FA<\CD@LB\ht\CD@MH\ht
\CD@FA\dp\CD@MH\dp\CD@FA\unhbox\CD@FA\skip0=\lastskip\unskip\else\CD@OK\skip0%
=\z@\fi\endgraf\ifcase\prevgraf\cd@shouldnt Y \or\cd@shouldnt Z \or\CD@RI\or
\CD@XI\else\CD@QI\fi\unskip\setbox0=\lastbox\unskip\unskip\unpenalty\noindent
\unhbox0\setbox0\lastbox\unpenalty\unskip\unskip\unpenalty\setbox0\lastbox
\CD@tF\CD@GB\lastpenalty\unpenalty\ifnum\CD@GB>\z@\setbox\z@\lastbox\CD@lB
\repeat\endgraf\unskip\unskip\unpenalty}\def\CD@YJ{\CD@uA\CD@XB\advance\CD@uA
-\CD@NB\CD@vA\CD@FA\advance\CD@vA-\CD@lA\advance\CD@vA1 \expandafter\message{%
prevgraf=\the\prevgraf at (\the\CD@uA,\the\CD@vA)}}\def\CD@XI{\CD@CE\setbox
\CD@lI=\lastbox\setbox\CD@lI=\hbox{\unhbox\CD@lI\unskip\unpenalty}\unskip
\ifdim\ht\CD@lI>\ht\CD@PC\setbox\CD@MH=\copy\CD@lI\else\ifdim\dp\CD@lI>\dp
\CD@PC\setbox\CD@MH=\copy\CD@lI\else\CD@FG\CD@lI\fi\fi\advance\CD@mF.5\wd
\CD@lI\advance\CD@mI.5\wd\CD@lI\setbox\CD@lI=\hbox{\unhbox\CD@lI\CD@HG}\CD@bH
\CD@mF{\box\CD@lI}\CD@mI\z@\CD@yB\CD@vB}\def\CD@CE{\ifnum\CD@A>0 \advance
\dimen0-\CD@tB\CD@iA-.5\dimen0 \CD@A-\CD@A\else\CD@A0 \CD@iA\z@\fi\setbox
\CD@MH=\lastbox\setbox\CD@MH=\hbox{\unhbox\CD@MH\unskip\unskip\unpenalty
\setbox0=\lastbox\global\CD@QA\lastkern\unkern}\advance\CD@iA-.5\CD@QA\unskip
\setbox\CD@MH=\null\CD@mI\CD@iA\CD@mF-\CD@iA}\def\CD@Z{\ht\CD@MH\CD@tI\dp
\CD@MH\CD@sI}\def\CD@FG#1{\setbox\CD@MH=\hbox{\CD@V{\ht\CD@MH}{\ht#1}\CD@V{%
\dp\CD@MH}{\dp#1}\CD@V{\wd\CD@MH}{\wd#1}\vrule height\ht\CD@MH depth\dp\CD@MH
width\wd\CD@MH}}\def\CD@QI{\CD@CE\CD@Z\setbox\CD@lI=\lastbox\unskip\setbox
\CD@lF=\lastbox\unskip\setbox\CD@lF=\hbox{\unhbox\CD@lF\unskip\global\CD@yA
\lastpenalty\unpenalty}\advance\CD@yA9999 \ifcase\CD@yA\CD@VI\or\CD@YI\or
\CD@TI\or\CD@dI\or\CD@cI\or\CD@SI\else\cd@shouldnt9\fi}\def\CD@VI{\CD@FG
\CD@lI\CD@UI\setbox\CD@sH=\box\CD@lF\setbox\CD@tH=\box\CD@lI}\def\CD@YI{%
\CD@FG\CD@lF\setbox\CD@lI\hbox{\penalty8 \unhbox\CD@lI\unskip\unpenalty\ifnum
\lastpenalty=8 \else\CD@xH\fi}\CD@UI\setbox\CD@lF=\hbox{\unhbox\CD@lF\unskip
\unpenalty\global\setbox\CD@DA=\lastbox}\ifdim\wd\CD@lF=\z@\else\CD@xH\fi
\setbox\CD@sH=\box\CD@DA}\def\CD@xH{\CD@KB{extra material in \string\pile
\space cell (lost)}}\def\CD@UI{\CD@yB\ifvoid\CD@sH\else\CD@KB{Clashing
horizontal arrows}\CD@mI.5\CD@hF\CD@mF-\CD@mI\CD@vB\CD@mI\z@\CD@mF\z@\fi
\CD@hI\CD@hF\advance\CD@hI-\CD@mI\CD@hF-\CD@mF\CD@JC\CD@FA}\def\CD@RI{\setbox
0\lastbox\unskip\CD@iA\z@\CD@Z\ifdim\skip0>\z@\CD@tJ\CD@A0 \else\ifnum\CD@A<1
\CD@A0 \dimen0\CD@tB\fi\advance\CD@A1 \fi}\def\VonH{\CD@MA46\VonH{.5\CD@LF}}%
\def\HonV{\CD@MA57\HonV{.5\CD@LF}}\def\HmeetV{\CD@MA44\HmeetV{-\MapShortFall}%
}\def\CD@MA#1#2#3#4{\CD@pB34#1{\string#3}\CD@SD\CD@GB-999#2 \dimen0=#4\CD@tI
\dimen0\advance\CD@tI\axisheight\CD@sI\dimen0\advance\CD@sI-\axisheight\CD@CF
\CD@HC\CD@ZD}\def\CD@HC#1{\setbox0=\hbox{\CD@k#1\CD@ND}\dimen0.5\wd0 \CD@tI
\ht0 \CD@sI\dp0 \CD@ZD}\def\CD@SD{\setbox0=\null\ht0=\CD@tI\dp0=\CD@sI\wd0=%
\dimen0 \copy0\penalty\CD@GB\box0 }\def\CD@TI{\CD@GC\CD@yB}\def\CD@dI{\CD@GC
\CD@vB}\def\CD@SI{\CD@GC\CD@yB\CD@vB}\def\CD@GC{\setbox\CD@lI=\hbox{\unhbox
\CD@lI}\setbox\CD@lF=\hbox{\unhbox\CD@lF\global\setbox\CD@DA=\lastbox}\ht
\CD@MH\ht\CD@DA\dp\CD@MH\dp\CD@DA\advance\CD@mF\wd\CD@DA\advance\CD@mI\wd
\CD@lI}\CD@tG\ifPositiveGradient\CD@CI\CD@BI\CD@CI\CD@tG\ifClimbing\CD@rB
\CD@qB\CD@rB\newcount\DiagonalChoice\DiagonalChoice\m@ne\ifx\tenln\nullfont
\CD@tJ\def\CD@qF{\CD@KH\ifPositiveGradient/\else\CD@k\backslash\CD@ND\fi}%
\else\def\CD@qF{\CD@rF\char\count@}\fi\let\CD@rF\tenln\def\Use@line@char#1{%
\hbox{#1\CD@rF\ifPositiveGradient\else\advance\count@64 \fi\char\count@}}\def
\CD@cF{\Use@line@char{\count@\CD@TC\multiply\count@8\advance\count@-9\advance
\count@\CD@LH}}\def\CD@ZF{\Use@line@char{\ifcase\DiagonalChoice\CD@gF\or
\CD@fF\or\CD@fF\else\CD@gF\fi}}\def\CD@gF{\ifnum\CD@TC=\z@\count@'33 \else
\count@\CD@TC\multiply\count@\sixt@@n\advance\count@-9\advance\count@\CD@LH
\advance\count@\CD@LH\fi}\def\CD@fF{\count@'\ifcase\CD@LH55\or\ifcase\CD@TC66%
\or22\or52\or61\or72\fi\or\ifcase\CD@TC66\or25\or22\or63\or52\fi\or\ifcase
\CD@TC66\or16\or36\or22\or76\fi\or\ifcase\CD@TC66\or27\or25\or67\or22\fi\fi
\relax}\def\CD@uC#1{\hbox{#1\setbox0=\Use@line@char{#1}\ifPositiveGradient
\else\raise.3\ht0\fi\copy0 \kern-.7\wd0 \ifPositiveGradient\raise.3\ht0\fi
\box0}}\def\CD@jF#1{\hbox{\setbox0=#1\kern-.75\wd0 \vbox to.25\ht0{%
\ifPositiveGradient\else\vss\fi\box0 \ifPositiveGradient\vss\fi}}}\def\CD@jI#%
1{\hbox{\setbox0=#1\dimen0=\wd0 \vbox to.25\ht0{\ifPositiveGradient\vss\fi
\box0 \ifPositiveGradient\else\vss\fi}\kern-.75\dimen0 }}\CD@RC{+h:>}{%
\Use@line@char\CD@fF}\CD@RC{-h:>}{\Use@line@char\CD@gF}\CD@nF{+t:<}{-h:>}%
\CD@nF{-t:<}{+h:>}\CD@RC{+t:>}{\CD@jF{\Use@line@char\CD@fF}}\CD@RC{-t:>}{%
\CD@jI{\Use@line@char\CD@gF}}\CD@nF{+h:<}{-t:>}\CD@nF{-h:<}{+t:>}\CD@RC{+h:>>%
}{\CD@uC\CD@fF}\CD@RC{-h:>>}{\CD@uC\CD@gF}\CD@nF{+t:<<}{-h:>>}\CD@nF{-t:<<}{+%
h:>>}\CD@nF{+h:>->}{+h:>>}\CD@nF{-h:>->}{-h:>>}\CD@nF{+t:<-<}{-h:>>}\CD@nF{-t%
:<-<}{+h:>>}\CD@RC{+t:>>}{\CD@jF{\CD@uC\CD@fF}}\CD@RC{-t:>>}{\CD@jI{\CD@uC
\CD@gF}}\CD@nF{+h:<<}{-t:>>}\CD@nF{-h:<<}{+t:>>}\CD@nF{+t:>->}{+t:>>}\CD@nF{-%
t:>->}{-t:>>}\CD@nF{+h:<-<}{-t:>>}\CD@nF{-h:<-<}{+t:>>}\CD@RC{+f:-}{\CD@EF
\null\else\CD@cF\fi}\CD@nF{-f:-}{+f:-}\def\CD@tC#1#2{\vbox to#1{\vss\hbox to#%
2{\hss.\hss}\vss}}\def\hfdot{\CD@tC{2\axisheight}{.5em}}%
\def\vfdot{\CD@tC{1ex}\z@}
\def\CD@bF{\hbox{\dimen0=.3\CD@zC\dimen1\dimen0 \ifnum\CD@LH>\CD@TC\CD@iC{%
\dimen1}\else\CD@dG{\dimen0}\fi\CD@tC{\dimen0}{\dimen1}}}\newarrowfiller{.}%
\hfdot\hfdot\vfdot\vfdot\def\dfdot{\CD@bF\CD@CK}\CD@RC{+f:.}{\dfdot}\CD@RC{-f%
:.}{\dfdot}\def\CD@@K#1{\hbox\bgroup\def\CD@CH{#1\egroup}\afterassignment
\CD@CH
\count@='}\def\lnchar{\CD@@K\CD@qF}\def\CD@dF#1{\setbox#1=\hbox{\dimen5\dimen
#1 \setbox8=\box#1 \dimen1\wd8 \count@\dimen5 \divide\count@\dimen1 \ifnum
\count@=0 \box8 \ifdim\dimen5<.95\dimen1 \CD@gB{diagonal line too short}\fi
\else\dimen3=\dimen5 \advance\dimen3-\dimen1 \divide\dimen3\count@\dimen4%
\dimen3 \CD@dG{\dimen4}\ifPositiveGradient\multiply\dimen4\m@ne\fi\dimen6%
\dimen1 \advance\dimen6-\dimen3 \loop\raise\count@\dimen4\copy8 \ifnum\count@
>0 \kern-\dimen6 \advance\count@\m@ne\repeat\fi}}\def\CD@CG#1{\CD@EF\CD@xJ{#1%
}\else\CD@dF{#1}\fi}\def\CD@IH#1{}\newdimen\objectheight\objectheight1.8ex
\newdimen\objectwidth\objectwidth1em \def\CD@YD{\dimen6=\CD@aK
\DiagramCellHeight\dimen7=\CD@WK\DiagramCellWidth\CD@KJ\ifnum\CD@LH>0 \ifnum
\CD@TC>0 \CD@aF\else\aftergroup\CD@VC\fi\else\aftergroup\CD@UC\fi}\def\CD@VC{%
\CD@YA{diagonal map is nearly vertical}\CD@NA}\def\CD@UC{\CD@YA{diagonal map
is nearly horizontal}\CD@NA}\CD@rG\CD@NA{Use an orthogonal map instead}\def
\CD@aF{\CD@MJ\dimen3\dimen7\dimen7\dimen6\CD@iC{\dimen7}\advance\dimen3-%
\dimen7 \CD@MF\ifnum\CD@LH>\CD@TC\advance\dimen6-\dimen1\advance\dimen6-%
\dimen5 \CD@iC{\dimen1}\CD@iC{\dimen5}\else\dimen0\dimen1\advance\dimen0%
\dimen5\CD@dG{\dimen0}\advance\dimen6-\dimen0 \fi\dimen2.5\dimen7\advance
\dimen2-\dimen1 \dimen4.5\dimen7\advance\dimen4-\dimen5 \ifPositiveGradient
\dimen0\dimen5 \advance\dimen1-\CD@WK\DiagramCellWidth\advance\dimen1 \CD@ZK
\DiagramCellWidth\setbox6=\llap{\unhbox6\kern.1\ht2}\setbox7=\rlap{\kern.1\ht
2\unhbox7}\else\dimen0\dimen1 \advance\dimen1-\CD@ZK\DiagramCellWidth\setbox7%
=\llap{\unhbox7\kern.1\ht2}\setbox6=\rlap{\kern.1\ht2\unhbox6}\fi\setbox6=%
\vbox{\box6\kern.1\wd2}\setbox7=\vtop{\kern.1\wd2\box7}\CD@dG{\dimen0}%
\advance\dimen0-\axisheight\advance\dimen0-\CD@bK\DiagramCellHeight\dimen5-%
\dimen0 \advance\dimen0\dimen6 \advance\dimen1.5\dimen3 \ifdim\wd3>\z@\ifdim
\ht3>-\dp3\CD@TB\fi\fi\dimen3\dimen2 \dimen7\dimen2\advance\dimen7\dimen4
\ifvoid3 \else\CD@tE\else\dimen8\ht3\advance\dimen8-\axisheight\CD@iC{\dimen8%
}\CD@X{\dimen8}{.5\wd3}\dimen9\dp3\advance\dimen9\axisheight\CD@iC{\dimen9}%
\CD@X{\dimen9}{.5\wd3}\ifPositiveGradient\advance\dimen2-\dimen9\advance
\dimen4-\dimen8 \else\advance\dimen4-\dimen9\advance\dimen2-\dimen8 \fi\fi
\advance\dimen3-.5\wd3 \fi\dimen9=\CD@aK\DiagramCellHeight\advance\dimen9-2%
\DiagramCellHeight\CD@tE\advance\dimen2\dimen4 \CD@CG{2}\dimen2-\dimen0%
\advance\dimen2\dp2 \else\CD@CG{2}\CD@CG{4}\ifPositiveGradient\dimen2-\dimen0%
\advance\dimen2\dp2 \dimen4\dimen5\advance\dimen4-\ht4 \else\dimen4-\dimen0%
\advance\dimen4\dp4 \dimen2\dimen5\advance\dimen2-\ht2 \fi\fi\setbox0=\hbox to%
\z@{\kern\dimen1 \ifvoid1 \else\ifPositiveGradient\advance\dimen0-\dp1 \lower
\dimen0 \else\advance\dimen5-\ht1 \raise\dimen5 \fi\rlap{\unhbox1}\fi\raise
\dimen2\rlap{\unhbox2}\ifvoid3 \else\lower.5\dimen9\rlap{\kern\dimen3\unhbox3%
}\fi\kern.5\dimen7 \lower.5\dimen9\box6 \lower.5\dimen9\box7 \kern.5\dimen7
\CD@tE\else\raise\dimen4\llap{\unhbox4}\fi\ifvoid5 \else\ifPositiveGradient
\advance\dimen5-\ht5 \raise\dimen5 \else\advance\dimen0-\dp5 \lower\dimen0 \fi
\llap{\unhbox5}\fi\hss}\ht0=\axisheight\dp0=-\ht0\box0 }\def\NorthWest{\CD@BI
\CD@rB\DiagonalChoice0 }\def\NorthEast{\CD@CI\CD@rB\DiagonalChoice1 }\def
\SouthWest{\CD@CI\CD@qB\DiagonalChoice3 }\def\SouthEast{\CD@BI\CD@qB
\DiagonalChoice2 }\def\CD@aD{\vadjust{\CD@uA\CD@FA\advance\CD@uA
\ifPositiveGradient\else-\fi\CD@ZK\relax\CD@vA\CD@NB\advance\CD@vA-\CD@bK
\relax\hbox{\advance\CD@uA\ifPositiveGradient-\fi\CD@WK\advance\CD@vA\CD@aK
\hbox{\box6 \kern\CD@DC\kern\CD@eJ\penalty1 \box7 \box\z@}\penalty\CD@uA
\penalty\CD@vA}\penalty\CD@uA\penalty\CD@vA\penalty104}}\def\CD@eH#1{\relax
\vadjust{\hbox@maths{#1}\penalty\CD@FA\penalty\CD@NB\penalty\tw@}}\def\CD@lB{%
\ifcase\CD@GB\or\or\CD@bH{.5\wd0}{\box0}{.5\wd0}\z@\or\unhbox\z@\setbox\z@
\lastbox\CD@bH{.5\wd0}{\box0}{.5\wd0}\z@\unpenalty\unpenalty\setbox\z@
\lastbox\or\CD@TG\else\advance\CD@GB-100 \ifnum\CD@GB<\z@\cd@shouldnt B\fi
\setbox\z@\hbox{\kern\CD@mF\copy\CD@MH\kern\CD@mI\CD@uA\CD@XB\advance\CD@uA-%
\CD@NB\penalty\CD@uA\CD@uA\CD@FA\advance\CD@uA-\CD@lA\penalty\CD@uA\unhbox\z@
\global\CD@yA\lastpenalty\unpenalty\global\CD@zA\lastpenalty\unpenalty}\CD@uA
-\CD@yA\CD@vA\CD@zA\CD@fI\fi}\def\CD@TG{\unhbox\z@\setbox\z@\lastbox\CD@uA
\lastpenalty\unpenalty\advance\CD@uA\CD@mA\CD@vA\CD@XB\advance\CD@vA-%
\lastpenalty\unpenalty\dimen1\lastkern\unkern\setbox3\lastbox\dimen0\lastkern
\unkern\setbox0=\hbox to\z@{\unhbox0\setbox0\lastbox\setbox7\lastbox
\unpenalty\CD@eJ\lastkern\unkern\CD@DC\lastkern\unkern\setbox6\lastbox\dimen7%
\CD@tB\advance\dimen7-\wd\CD@uA\ifdim\dimen7<\z@\CD@CI\multiply\dimen7\m@ne
\let\mv\empty\else\CD@BI\def\mv{\raise\ht1}\kern-\dimen7 \fi\ifnum\CD@vA>%
\CD@NB\dimen6\CD@uB\advance\dimen6-\ht\CD@vA\else\dimen6\z@\fi\CD@jJ\CD@mK
\setbox1\null\ht1\dimen6\wd1\dimen7 \dimen7\dimen2 \dimen6\wd1 \CD@KJ\CD@uA
\CD@LH\CD@vA\CD@TC\dimen6\ht1 \CD@KJ\setbox2\null\divide\dimen2\tw@\advance
\dimen2\CD@eJ\CD@eG{\dimen2}\wd2\dimen2 \dimen0.5\dimen7 \advance\dimen0%
\ifPositiveGradient\else-\fi\CD@eJ\CD@dG{\dimen0}\advance\dimen0-\axisheight
\ht2\dimen0 \dimen0\CD@DC\CD@eG{\dimen0}\advance\dimen0\ht2\ht2\dimen0 \dimen
0\ifPositiveGradient-\fi\CD@DC\CD@dG{\dimen0}\advance\dimen0\wd2\wd2\dimen0
\setbox4\null\dimen0 .6\CD@zC\CD@eG{\dimen0}\ht4\dimen0 \dimen0 .2\CD@zC
\CD@dG{\dimen0}\wd4\dimen0 \dimen0\wd2 \ifvoid6\else\dimen1\ht4 \advance
\dimen1\ht2 \CD@CC6+-\raise\dimen1\rlap{\ifPositiveGradient\advance\dimen0-%
\wd6\advance\dimen0-\wd4 \else\advance\dimen0\wd4 \fi\kern\dimen0\box6}\fi
\dimen0\wd2 \ifvoid7\else\dimen1\ht4 \advance\dimen1-\ht2 \CD@CC7-+\lower
\dimen1\rlap{\ifPositiveGradient\advance\dimen0\wd4 \else\advance\dimen0-\wd7%
\advance\dimen0-\wd4 \fi\kern\dimen0\box7}\fi\mv\box0\hss}\ht0\z@\dp0\z@
\CD@bH{\z@}{\box\z@}{\z@}{\axisheight}}\def\CD@CC#1#2#3{\dimen4.5\wd#1 \ifdim
\dimen4>.25\dimen7\dimen4=.25\dimen7\fi\ifdim\dimen4>\CD@zC\dimen4.4\dimen4
\advance\dimen4.6\CD@zC\fi\CD@eG{\dimen4}\dimen5\axisheight\CD@dG{\dimen5}%
\advance\dimen4-\dimen5 \dimen5\dimen4\CD@eG{\dimen5}\advance\dimen0%
\ifPositiveGradient#2\else#3\fi\dimen5 \CD@dG{\dimen4}\advance\dimen1\dimen4 }
\def\CD@eD#1{\expandafter\CD@IK{#1}}\CD@ZA\CD@EK{output is PostScript
dependent}\def\CD@SC{\CD@IK{/bturn {gsave currentpoint currentpoint translate
4 2 roll neg exch atan rotate neg exch neg exch translate } def /eturn {%
currentpoint grestore moveto} def}}\def\CD@gK{\relax\CD@hK\CD@tK{Q}\else
\CD@IK{eturn}\fi} \def\CD@OJ#1{\count@#1\relax\multiply\count@7\advance
\count@16577\divide\count@33154 }\def\CD@fD#1{\expandafter\special{#1}} \def
\CD@xJ#1{\setbox#1=\hbox{\dimen0\dimen#1\CD@dG{\dimen0}\CD@OJ{\dimen0}\setbox
0=\null\ifPositiveGradient\count@-\count@\ht0\dimen0 \else\dp0\dimen0 \fi\box
0 \CD@uA\count@\CD@OJ\CD@LF\CD@fD{pn \the\count@}\CD@fD{pa 0 0}\CD@OJ{\dimen#%
1}\CD@fD{pa \the\count@\space\the\CD@uA}\CD@fD{fp}\kern\dimen#1}}\def\CD@JI{%
\CD@KJ\begingroup\ifdim\dimen7<\dimen6 \dimen2=\dimen6 \dimen6=\dimen7 \dimen
7=\dimen2 \count@\CD@LH\CD@LH\CD@TC\CD@TC\count@\else\dimen2=\dimen7 \fi
\ifdim\dimen6>.01\p@\CD@KI\global\CD@QA\dimen0 \else\global\CD@QA\dimen7 \fi
\endgroup\dimen2\CD@QA\CD@iK\CD@lK{\ifPositiveGradient\else-\fi\dimen6}\CD@iK
\CD@kK{\ifPositiveGradient-\fi\dimen6}\CD@iK\CD@eK{\dimen7}}\def\CD@KI{\CD@hJ
\ifdim\dimen7>1.73\dimen6 \divide\dimen2 4 \multiply\CD@TC2 \else\dimen2=0.%
353553\dimen2 \advance\CD@LH-\CD@TC\multiply\CD@TC4 \fi\dimen0=4\dimen2 \CD@ZG
4\CD@ZG{-2}\CD@ZG2\CD@ZG{-2.5}}\def\CD@AI{\begingroup\count@\dimen0 \dimen2 45%
pt \divide\count@\dimen2 \ifdim\dimen0<\z@\advance\count@\m@ne\fi\ifodd
\count@\advance\count@1\CD@@A\else\CD@y\fi\advance\dimen0-\count@\dimen2
\CD@gE\multiply\dimen0\m@ne\fi\ifnum\count@<0 \multiply\count@-7 \fi\dimen3%
\dimen1 \dimen6\dimen0 \dimen7 3754936sp \ifdim\dimen0<6\p@\def\CD@OG{4000}%
\fi\CD@KJ\dimen2\dimen3\CD@dG{\dimen2}\CD@hJ\multiply\CD@TC-6 \dimen0\dimen2
\CD@ZG1\CD@ZG{0.3}\dimen1\dimen0 \dimen2\dimen3 \dimen0\dimen3 \CD@ZG3\CD@ZG{%
1.5}\CD@ZG{0.3}\divide\count@2 \CD@gE\multiply\dimen1\m@ne\fi\ifodd\count@
\dimen2\dimen1\dimen1\dimen0\dimen0-\dimen2 \fi\divide\count@2 \ifodd\count@
\multiply\dimen0\m@ne\multiply\dimen1\m@ne\fi\global\CD@QA\dimen0\global
\CD@RA\dimen1\endgroup\dimen6\CD@QA\dimen7\CD@RA}\def\CD@OC{255}\let\CD@OG
\CD@OC\def\CD@KJ{\begingroup\ifdim\dimen7<\dimen6 \dimen9\dimen7\dimen7\dimen
6\dimen6\dimen9\CD@@A\else\CD@y\fi\dimen2\z@\dimen3\CD@XH\dimen4\CD@XH\dimen0%
\z@\dimen8=\CD@OG\CD@XH\CD@lC\global\CD@yA\dimen\CD@gE0\else3\fi\global\CD@zA
\dimen\CD@gE3\else0\fi\endgroup\CD@LH\CD@yA\CD@TC\CD@zA}\def\CD@lC{\count@
\dimen6 \divide\count@\dimen7 \advance\dimen6-\count@\dimen7 \dimen9\dimen4
\advance\dimen9\count@\dimen0 \ifdim\dimen9>\dimen8 \CD@@C\else\CD@AC\ifdim
\dimen6>\z@\dimen9\dimen6 \dimen6\dimen7 \dimen7\dimen9 \expandafter
\expandafter\expandafter\CD@lC\fi\fi}\def\CD@@C{\ifdim\dimen0=\z@\ifdim\dimen
9<2\dimen8 \dimen0\dimen8 \fi\else\advance\dimen8-\dimen4 \divide\dimen8%
\dimen0 \ifdim\count@\CD@XH<2\dimen8 \count@\dimen8 \dimen9\dimen4 \advance
\dimen9\count@\dimen0 \CD@AC\fi\fi}\def\CD@AC{\dimen4\dimen0 \dimen0\dimen9
\advance\dimen2\count@\dimen3 \dimen9\dimen2 \dimen2\dimen3 \dimen3\dimen9 }%
\def\CD@ZG#1{\CD@dG{\dimen2}\advance\dimen0 #1\dimen2 }\def\CD@dG#1{\divide#1%
\CD@TC\multiply#1\CD@LH}\def\CD@eG#1{\divide#1\CD@vA\multiply#1\CD@uA}\def
\CD@iC#1{\divide#1\CD@LH\multiply#1\CD@TC}\def\CD@hJ{\dimen6\CD@LH\CD@XH
\multiply\dimen6\CD@LH\dimen7\CD@TC\CD@XH\multiply\dimen7\CD@TC\CD@KJ}\def
\CD@iK#1#2{\begingroup\dimen@#2\relax\loop\ifdim\dimen2<.4\maxdimen\multiply
\dimen2\tw@\multiply\dimen@\tw@\repeat\divide\dimen2\@cclvi\divide\dimen@
\dimen2\relax\multiply\dimen@\@cclvi\expandafter\CD@jK\the\dimen@\endgroup
\let#1\CD@fK}{\catcode`p=12 \catcode`0=12 \catcode`.=12 \catcode`t=12 \gdef
\CD@jK#1pt{\gdef\CD@fK{#1}}}\ifx\errorcontextlines\CD@qK\CD@tJ\let\CD@GH
\relax\else\def\CD@GH{\errorcontextlines\m@ne}\fi\ifnum\inputlineno<0 \let
\CD@CD\empty\let\CD@W\empty\let\CD@mD\relax\let\CD@uI\relax\let\CD@vI\relax
\let\CD@zF\relax\else\def\CD@W{ at line \number\inputlineno}\def\CD@mD{ - first occurred%
}\def\CD@uI{\edef\CD@h{\the\inputlineno}\global\let\CD@jB\CD@h}\def\CD@h{9999%
}\def\CD@vI{\xdef\CD@jB{\the\inputlineno}}\def\CD@jB{\CD@h}\def\CD@zF{\ifnum
\CD@h<\inputlineno\edef\CD@CD{\space at lines \CD@h--\the\inputlineno}\else
\edef\CD@CD{\CD@W}\fi}\fi\let\CD@CD\empty\def\CD@YA#1#2{\CD@GH\errhelp=#2%
\expandafter\errmessage{\CD@tA: #1}}\def\CD@KB#1{\begingroup\expandafter
\message{! \CD@tA: #1\CD@CD}\ifnum\CD@XB>\CD@NB\ifnum\CD@CA>\CD@NB\else\ifnum
\CD@lA>\CD@FA\else\ifnum\CD@LB>\CD@FA\advance\CD@XB-\CD@NB\advance\CD@FA-%
\CD@lA\advance\CD@FA1\relax\expandafter\message{! (error detected at row \the
\CD@XB, column \the\CD@FA, but probably caused elsewhere)}\fi\fi\fi\fi
\endgroup}\def\CD@gB#1{{\expandafter\message{\CD@tA\space Warning: #1\CD@W}}}%
\def\CD@CB#1#2{\CD@gB{#1 \string#2 is obsolete\CD@mD}}\def\CD@AB#1{\CD@CB{%
Dimension}{#1}\CD@DE#1\CD@BB\CD@BB}\def\CD@BB{\CD@OA=}\def\CD@@B#1{\CD@CB{%
Count}{#1}\CD@DE#1\CD@OH\CD@OH}\def\CD@OH{\count@=}\def\HorizontalMapLength{%
\CD@AB\HorizontalMapLength}\def\VerticalMapHeight{\CD@AB\VerticalMapHeight}%
\def\VerticalMapDepth{\CD@AB\VerticalMapDepth}\def\VerticalMapExtraHeight{%
\CD@AB\VerticalMapExtraHeight}\def\VerticalMapExtraDepth{\CD@AB
\VerticalMapExtraDepth}\def\DiagonalLineSegments{\CD@@B\DiagonalLineSegments}%
\ifx\tenln\nullfont\CD@ZA\CD@KH{\CD@eF\space diagonal line and arrow font not
available}\else\let\CD@KH\relax\fi\def\CD@aG#1#2<#3:#4:#5#6{\begingroup\CD@PA
#3\relax\advance\CD@PA-#2\relax\ifdim.1em<\CD@PA\CD@uA#5\relax\CD@vA#6\relax
\ifnum\CD@uA<\CD@vA\count@\CD@vA\advance\count@-\CD@uA\CD@KB{#4 by \the\CD@PA
}\if#1v\let\CD@CH\CD@JK\edef\tmp{\the\CD@uA--\the\CD@vA,\the\CD@FA}\else
\advance\count@\count@\if#1l\advance\count@-\CD@A\else\if#1r\advance\count@
\CD@A\fi\fi\advance\CD@PA\CD@PA\let\CD@CH\CD@ZE\edef\tmp{\the\CD@NB,\the
\CD@uA--\the\CD@vA}\fi\divide\CD@PA\count@\ifdim\CD@CH<\CD@PA\global\CD@CH
\CD@PA\fi\fi\fi\endgroup}\CD@tG\CD@xE\CD@JD\CD@ID\CD@rG\CD@xI{See the message
above.}\CD@rG\CD@lH{Perhaps you've forgotten to end the diagram before
resuming the text, in\CD@uG which case some garbage may be added to the
diagram, but we should be ok now.\CD@uG Alternatively you've left a blank line
in the middle - TeX will now complain\CD@uG that the remaining \CD@S s are
misplaced - so please use comments for layout.}\CD@rG\CD@hD{You have already
closed too many brace pairs or environments; an \CD@HD\CD@uG command was (%
over)due.}\CD@rG\CD@hH{\CD@dC\space and \CD@HD\space commands must match.}%
\def\CD@jH{\ifnum\inputlineno=0 \else\expandafter\CD@iH\fi}\def\CD@iH{\CD@MD
\CD@GD\crcr\CD@YA{missing \CD@HD\space inserted before \CD@kH- type "h"}%
\CD@lH\enddiagram\CD@AG\CD@kH\par}\def\CD@AG#1{\edef\enddiagram{\noexpand
\CD@rD{#1\CD@W}}}\def\CD@rD#1{\CD@YA{\CD@HD\space(anticipated by #1) ignored}%
\CD@xI\let\enddiagram\CD@SG}\def\CD@SG{\CD@YA{misplaced \CD@HD\space ignored}%
\CD@hH}\def\CD@mC{\CD@YA{missing \CD@HD\space inserted.}\CD@hD\CD@AG{closing
group}}\ifx\ifx
\fi\fi

\catcode`\$=3 
\def\vboxtoz{\vbox to\z@}

\def\scriptaxis#1{\@scriptaxis{$\scriptstyle#1$}}
\def\ssaxis#1{\@ssaxis{$\scriptscriptstyle#1$}}
\def\@scriptaxis#1{\dimen0\axisheight\advance\dimen0-\ss@axisheight\raise
\dimen0\hbox{#1}}\def\@ssaxis#1{\dimen0\axisheight\advance\dimen0-%
\ss@axisheight\raise\dimen0\hbox{#1}}

\ifx\boldmath\CD@qK
\let\boldscriptaxis\scriptaxis
\def\boldscript#1{\hbox{$\scriptstyle#1$}}
\else\def\boldscriptaxis#1{\@scriptaxis{\boldmath$\scriptstyle#1$}}
\def\boldscript#1{\hbox{\boldmath$\scriptstyle#1$}}
\fi

\def\raisehook#1#2#3{\hbox{\setbox3=\hbox{#1$\scriptscriptstyle#3$}%
\dimen0\ss@axisheight
\dimen1\axisheight\advance\dimen1-\dimen0
\dimen2\ht3\advance\dimen2-\dimen0%
\advance\dimen2-0.021em\advance\dimen1 #2\dimen2%
\raise\dimen1\box3}}
\def\shifthook#1#2#3{\setbox1=\hbox{#1$\scriptscriptstyle#3$}\dimen0\wd1%
\divide\dimen0 12\CD@zH{\dimen0}
\dimen1\wd1\advance\dimen1-2\dimen0 \advance\dimen1-2\CD@oI\CD@zH{\dimen1}%
\kern#2\dimen1\box1}

\def\@cmex{\mathchar"03}

\def\make@pbk#1{\setbox\tw@\hbox to\z@{#1}\ht\tw@\z@\dp\tw@\z@\box\tw@}\def
\CD@fH#1{\overprint{\hbox to\z@{#1}}}\def\CD@qH{\kern0.11em}\def\CD@pH{\kern0%
.35em}

\def\dblvert{\def\CD@rH{\kern.5\PileSpacing}}\def\CD@rH{}

\def\SEpbk{\make@pbk{\CD@qH\CD@rH\vrule depth 2.87ex height -2.75ex width 0.%
95em \vrule height -0.66ex depth 2.87ex width 0.05em \hss}}

\def\SWpbk{\make@pbk{\hss\vrule height -0.66ex depth 2.87ex width 0.05em
\vrule depth 2.87ex height -2.75ex width 0.95em \CD@qH\CD@rH}}

\def\NEpbk{\make@pbk{\CD@qH\CD@rH\vrule depth -3.81ex height 4.00ex width 0.%
95em \vrule height 4.00ex depth -1.72ex width 0.05em \hss}}

\def\NWpbk{\make@pbk{\hss\vrule height 4.00ex depth -1.72ex width 0.05em
\vrule depth -3.81ex height 4.00ex width 0.95em \CD@qH\CD@rH}}

\def\puncture{{\setbox0\hbox{A}\vrule height.53\ht0 depth-.47\ht0 width.35\ht
0 \kern.12\ht0 \vrule height\ht0 depth-.65\ht0 width.06\ht0 \kern-.06\ht0
\vrule height.35\ht0 depth0pt width.06\ht0 \kern.12\ht0 \vrule height.53\ht0
depth-.47\ht0 width.35\ht0 }}

\def\NEclck{\overprint{\raise2.5ex\rlap{ \CD@rH$\scriptstyle\searrow$}}}
\def\NEanti{\overprint{\raise2.5ex\rlap{ \CD@rH$\scriptstyle\nwarrow$}}}
\def\NWclck{\overprint{\raise2.5ex\llap{$\scriptstyle\nearrow$ \CD@rH}}}
\def\NWanti{\overprint{\raise2.5ex\llap{$\scriptstyle\swarrow$ \CD@rH}}}
\def\SEclck{\overprint{\lower1ex\rlap{ \CD@rH$\scriptstyle\swarrow$}}}
\def\SEanti{\overprint{\lower1ex\rlap{ \CD@rH$\scriptstyle\nearrow$}}}
\def\SWclck{\overprint{\lower1ex\llap{$\scriptstyle\nwarrow$ \CD@rH}}}
\def\SWanti{\overprint{\lower1ex\llap{$\scriptstyle\searrow$ \CD@rH}}}

\def\rhvee{\mkern-10mu\greaterthan}
\def\lhvee{\lessthan\mkern-10mu}
\def\dhvee{\vboxtoz{\vss\hbox{$\vee$}\kern0pt}}
\def\uhvee{\vboxtoz{\hbox{$\wedge$}\vss}}
\newarrowhead{vee}\rhvee\lhvee\dhvee\uhvee

\def\dhlvee{\vboxtoz{\vss\hbox{$\scriptstyle\vee$}\kern0pt}}
\def\uhlvee{\vboxtoz{\hbox{$\scriptstyle\wedge$}\vss}}
\newarrowhead{littlevee}{\mkern1mu\scriptaxis\rhvee}{\scriptaxis\lhvee}%
\dhlvee\uhlvee\ifx\boldmath\CD@qK
\newarrowhead{boldlittlevee}{\mkern1mu\scriptaxis\rhvee}{\scriptaxis\lhvee}%
\dhlvee\uhlvee\else
\def\dhblvee{\vboxtoz{\vss\boldscript\vee\kern0pt}}
\def\uhblvee{\vboxtoz{\boldscript\wedge\vss}}
\newarrowhead{boldlittlevee}{\mkern1mu\boldscriptaxis\rhvee}{\boldscriptaxis
\lhvee}\dhblvee\uhblvee
\fi

\def\rhcvee{\mkern-10mu\succ}
\def\lhcvee{\prec\mkern-10mu}
\def\dhcvee{\vboxtoz{\vss\hbox{$\curlyvee$}\kern0pt}}
\def\uhcvee{\vboxtoz{\hbox{$\curlywedge$}\vss}}
\newarrowhead{curlyvee}\rhcvee\lhcvee\dhcvee\uhcvee

\def\rhvvee{\mkern-13mu\gg}
\def\lhvvee{\ll\mkern-13mu}
\def\dhvvee{\vboxtoz{\vss\hbox{$\vee$}\kern-.6ex\hbox{$\vee$}\kern0pt}}
\def\uhvvee{\vboxtoz{\hbox{$\wedge$}\kern-.6ex \hbox{$\wedge$}\vss}}
\newarrowhead{doublevee}\rhvvee\lhvvee\dhvvee\uhvvee

\def\rhtriangle{\triangleright\mkern1.2mu}
\def\lhtriangle{\triangleleft\mkern.8mu}
\def\uhtriangle{\vbox{\kern-.2ex \hbox{$\scriptscriptstyle\bigtriangleup$}%
\kern-.25ex}}
\def\dhtriangle{\vbox{\kern-.28ex \hbox{$\scriptscriptstyle\bigtriangledown$}%
\kern-.1ex}}
\def\dhblack{\vbox{\kern-.25ex\nointerlineskip\hbox{$\blacktriangledown$}}}%
\def\uhblack{\vbox{\kern-.25ex\nointerlineskip\hbox{$\blacktriangle$}}}%
\def\dhlblack{\vbox{\kern-.25ex\nointerlineskip\hbox{$\scriptstyle
\blacktriangledown$}}}
\def\uhlblack{\vbox{\kern-.25ex\nointerlineskip\hbox{$\scriptstyle
\blacktriangle$}}}
\newarrowhead{triangle}\rhtriangle\lhtriangle\dhtriangle\uhtriangle
\newarrowhead{blacktriangle}{\mkern-1mu\blacktriangleright\mkern.4mu}{%
\blacktriangleleft}\dhblack\uhblack\newarrowhead{littleblack}{\mkern-1mu%
\scriptaxis\blacktriangleright}{\scriptaxis\blacktriangleleft\mkern-2mu}%
\dhlblack\uhlblack

\def\rhla{\hbox{\setbox0=\lnchar55\dimen0=\wd0\kern-.6\dimen0\ht0\z@\raise
\axisheight\box0\kern.1\dimen0}}
\def\lhla{\hbox{\setbox0=\lnchar33\dimen0=\wd0\kern.05\dimen0\ht0\z@\raise
\axisheight\box0\kern-.5\dimen0}}
\def\dhla{\vboxtoz{\vss\rlap{\lnchar77}}}
\def\uhla{\vboxtoz{\setbox0=\lnchar66 \wd0\z@\kern-.15\ht0\box0\vss}}
\newarrowhead{LaTeX}\rhla\lhla\dhla\uhla

\def\lhlala{\lhla\kern.3em\lhla}
\def\rhlala{\rhla\kern.3em\rhla}
\def\uhlala{\hbox{\uhla\raise-.6ex\uhla}}
\def\dhlala{\hbox{\dhla\lower-.6ex\dhla}}
\newarrowhead{doubleLaTeX}\rhlala\lhlala\dhlala\uhlala

\def\hhO{\scriptaxis\bigcirc\mkern.4mu} \def\hho{{\circ}\mkern1.2mu}%
\newarrowhead{o}\hho\hho\circ\circ
\newarrowhead{O}\hhO\hhO{\scriptstyle\bigcirc}{\scriptstyle\bigcirc}

\def\rhtimes{\mkern-5mu{\times}\mkern-.8mu}\def\lhtimes{\mkern-.8mu{\times}%
\mkern-5mu}\def\uhtimes{\setbox0=\hbox{$\times$}\ht0\axisheight\dp0-\ht0%
\lower\ht0\box0 }\def\dhtimes{\setbox0=\hbox{$\times$}\ht0\axisheight\box0 }%
\newarrowhead{X}\rhtimes\lhtimes\dhtimes\uhtimes\newarrowhead+++++

\newarrowhead{Y}{\mkern-3mu\Yright}{\Yleft\mkern-3mu}\Ydown\Yup

\newarrowhead{->}\rightarrow\leftarrow\downarrow\uparrow

\newarrowhead{=>}\Rightarrow\Leftarrow{\@cmex7F}{\@cmex7E}

\newarrowhead{harpoon}\rightharpoonup\leftharpoonup\downharpoonleft
\upharpoonleft

\def\twoheaddownarrow{\rlap{$\downarrow$}\raise-.5ex\hbox{$\downarrow$}}
\def\twoheaduparrow{\rlap{$\uparrow$}\raise.5ex\hbox{$\uparrow$}}
\newarrowhead{->>}\twoheadrightarrow\twoheadleftarrow\twoheaddownarrow
\twoheaduparrow

\def\ltvee{\mkern-1mu{\lessthan}\mkern.4mu}
\newarrowtail{vee}\greaterthan\ltvee\vee\wedge

\newarrowtail{littlevee}{\scriptaxis\greaterthan}{\mkern-1mu\scriptaxis
\lessthan}{\scriptstyle\vee}{\scriptstyle\wedge}\ifx\boldmath\CD@qK
\newarrowtail{boldlittlevee}{\scriptaxis\greaterthan}{\mkern-1mu\scriptaxis
\lessthan}{\scriptstyle\vee}{\scriptstyle\wedge}\else\newarrowtail{%
boldlittlevee}{\boldscriptaxis\greaterthan}{\mkern-1mu\boldscriptaxis
\lessthan}{\boldscript\vee}{\boldscript\wedge}\fi

\newarrowtail{curlyvee}\succ{\mkern-1mu{\prec}\mkern.4mu}\curlyvee\curlywedge

\def\rttriangle{\mkern1.2mu\triangleright}
\newarrowtail{triangle}\rttriangle\lhtriangle\dhtriangle\uhtriangle
\newarrowtail{blacktriangle}\blacktriangleright{\mkern-1mu\blacktriangleleft
\mkern.4mu}\dhblack\uhblack\newarrowtail{littleblack}{\scriptaxis
\blacktriangleright\mkern-2mu}{\mkern-1mu\scriptaxis\blacktriangleleft}%
\dhlblack\uhlblack

\def\rtla{\hbox{\setbox0=\lnchar55\dimen0=\wd0\kern-.5\dimen0\ht0\z@\raise
\axisheight\box0\kern-.2\dimen0}}
\def\ltla{\hbox{\setbox0=\lnchar33\dimen0=\wd0\kern-.15\dimen0\ht0\z@\raise
\axisheight\box0\kern-.5\dimen0}}
\def\dtla{\vbox{\setbox0=\rlap{\lnchar77}\dimen0=\ht0\kern-.7\dimen0\box0%
\kern-.1\dimen0}}
\def\utla{\vbox{\setbox0=\rlap{\lnchar66}\dimen0=\ht0\kern-.1\dimen0\box0%
\kern-.6\dimen0}}
\newarrowtail{LaTeX}\rtla\ltla\dtla\utla

\def\rtvvee{\gg\mkern-3mu}
\def\ltvvee{\mkern-3mu\ll}
\def\dtvvee{\vbox{\hbox{$\vee$}\kern-.6ex \hbox{$\vee$}\vss}}
\def\utvvee{\vbox{\vss\hbox{$\wedge$}\kern-.6ex \hbox{$\wedge$}\kern\z@}}
\newarrowtail{doublevee}\rtvvee\ltvvee\dtvvee\utvvee

\def\ltlala{\ltla\kern.3em\ltla}
\def\rtlala{\rtla\kern.3em\rtla}
\def\utlala{\hbox{\utla\raise-.6ex\utla}}
\def\dtlala{\hbox{\dtla\lower-.6ex\dtla}}
\newarrowtail{doubleLaTeX}\rtlala\ltlala\dtlala\utlala

\def\utbar{\vrule height 0.093ex depth0pt width 0.4em}
\let\dtbar\utbar
\def\rtbar{\mkern1.5mu\vrule height 1.1ex depth.06ex width .04em\mkern1.5mu}%
\let\ltbar\rtbar
\newarrowtail{mapsto}\rtbar\ltbar\dtbar\utbar
\newarrowtail{|}\rtbar\ltbar\dtbar\utbar

\def\rthooka{\raisehook{}+\subset\mkern-1mu}
\def\lthooka{\mkern-1mu\raisehook{}+\supset}
\def\rthookb{\raisehook{}-\subset\mkern-2mu}
\def\lthookb{\mkern-1mu\raisehook{}-\supset}

\def\dthooka{\shifthook{}+\cap}
\def\dthookb{\shifthook{}-\cap}
\def\uthooka{\shifthook{}+\cup}
\def\uthookb{\shifthook{}-\cup}

\newarrowtail{hooka}\rthooka\lthooka\dthooka\uthooka\newarrowtail{hookb}%
\rthookb\lthookb\dthookb\uthookb

\ifx\boldmath\CD@qK\newarrowtail{boldhooka}\rthooka\lthooka\dthooka\uthooka
\newarrowtail{boldhookb}\rthookb\lthookb\dthookb\uthookb\newarrowtail{%
boldhook}\rthooka\lthooka\dthookb\uthooka\else\def\rtbhooka{\raisehook
\boldmath+\subset\mkern-1mu}
\def\ltbhooka{\mkern-1mu\raisehook\boldmath+\supset}
\def\rtbhookb{\raisehook\boldmath-\subset\mkern-2mu}
\def\ltbhookb{\mkern-1mu\raisehook\boldmath-\supset}
\def\dtbhooka{\shifthook\boldmath+\cap}
\def\dtbhookb{\shifthook\boldmath-\cap}
\def\utbhooka{\shifthook\boldmath+\cup}
\def\utbhookb{\shifthook\boldmath-\cup}
\newarrowtail{boldhooka}\rtbhooka\ltbhooka\dtbhooka\utbhooka\newarrowtail{%
boldhookb}\rtbhookb\ltbhookb\dtbhookb\utbhookb\newarrowtail{boldhook}%
\rtbhooka\ltbhooka\dtbhooka\utbhooka\fi

\def\dtsqhooka{\shifthook{}+\sqcap}
\def\ltsqhooka{\mkern-1mu\raisehook{}+\sqsupset}
\def\rtsqhooka{\raisehook{}+\sqsubset\mkern-1mu}
\def\utsqhooka{\shifthook{}+\sqcup}
\newarrowtail{sqhook}\rtsqhooka\ltsqhooka\dtsqhooka\utsqhooka

\newarrowtail{hook}\rthooka\lthookb\dthooka\uthooka\newarrowtail{C}\rthooka
\lthookb\dthooka\uthooka

\newarrowtail{o}\hho\hho\circ\circ
\newarrowtail{O}\hhO\hhO{\scriptstyle\bigcirc}{\scriptstyle\bigcirc}

\newarrowtail{X}\lhtimes\rhtimes\uhtimes\dhtimes\newarrowtail+++++

\newarrowtail{Y}\Yright\Yleft\Ydown\Yup

\newarrowtail{harpoon}\leftharpoondown\rightharpoondown\upharpoonright
\downharpoonright

\newarrowtail{<=}\Leftarrow\Rightarrow{\@cmex7E}{\@cmex7F}

\newarrowfiller{=}=={\@cmex77}{\@cmex77}
\def\vfthree{\mid\!\!\!\mid\!\!\!\mid}
\newarrowfiller{3}\equiv\equiv\vfthree\vfthree

\def\vfdashstrut{\vrule width0pt height1.3ex depth0.7ex}
\def\vfthedash{\vrule width\CD@LF height0.6ex depth 0pt}
\def\hfthedash{\CD@AJ\vrule\horizhtdp width 0.26em}
\def\hfdash{\mkern5.5mu\hfthedash\mkern5.5mu}
\def\vfdash{\vfdashstrut\vfthedash}
\newarrowfiller{dash}\hfdash\hfdash\vfdash\vfdash

\newarrowmiddle+++++

\iffalse
\newarrow{To}----{vee}
\newarrow{Arr}----{LaTeX}
\newarrow{Dotsto}....{vee}
\newarrow{Dotsarr}....{LaTeX}
\newarrow{Dashto}{}{dash}{}{dash}{vee}
\newarrow{Dasharr}{}{dash}{}{dash}{LaTeX}
\newarrow{Mapsto}{mapsto}---{vee}
\newarrow{Mapsarr}{mapsto}---{LaTeX}
\newarrow{IntoA}{hooka}---{vee}
\newarrow{IntoB}{hookb}---{vee}
\newarrow{Embed}{vee}---{vee}
\newarrow{Emarr}{LaTeX}---{LaTeX}
\newarrow{Onto}----{doublevee}
\newarrow{Dotsonarr}....{doubleLaTeX}
\newarrow{Dotsonto}....{doublevee}
\newarrow{Dotsonarr}....{doubleLaTeX}
\else
\newarrow{To}---->
\newarrow{Arr}---->
\newarrow{Dotsto}....>
\newarrow{Dotsarr}....>
\newarrow{Dashto}{}{dash}{}{dash}>
\newarrow{Dasharr}{}{dash}{}{dash}>
\newarrow{Mapsto}{mapsto}--->
\newarrow{Mapsarr}{mapsto}--->
\newarrow{IntoA}{hooka}--->
\newarrow{IntoB}{hookb}--->
\newarrow{Embed}>--->
\newarrow{Emarr}>--->
\newarrow{Onto}----{>>}
\newarrow{Dotsonarr}....{>>}
\newarrow{Dotsonto}....{>>}
\newarrow{Dotsonarr}....{>>}
\fi

\newarrow{Implies}===={=>}
\newarrow{Project}----{triangle}
\newarrow{Pto}----{harpoon}
\newarrow{Relto}{harpoon}---{harpoon}

\newarrow{Eq}=====
\newarrow{Line}-----
\newarrow{Dots}.....
\newarrow{Dashes}{}{dash}{}{dash}{}

\newarrow{SquareInto}{sqhook}--->

\newarrowhead{cmexbra}{\@cmex7B}{\@cmex7C}{\@cmex3B}{\@cmex38}
\newarrowtail{cmexbra}{\@cmex7A}{\@cmex7D}{\@cmex39}{\@cmex3A}
\newarrowmiddle{cmexbra}{\braceru\bracelu}{\bracerd\braceld}{\vcenter{%
\hbox@maths{\@cmex3D\mkern-2mu}}}
{\vcenter{\hbox@maths{\mkern2mu\@cmex3C}}}
\newarrow{@brace}{cmexbra}-{cmexbra}-{cmexbra}
\newarrow{@parenth}{cmexbra}---{cmexbra}
\def\rightBrace{\d@brace[thick,cmex]}
\def\leftBrace{\u@brace[thick,cmex]}
\def\upperBrace{\r@brace[thick,cmex]}
\def\lowerBrace{\l@brace[thick,cmex]}
\def\rightParenth{\d@parenth[thick,cmex]}
\def\leftParenth{\u@parenth[thick,cmex]}
\def\upperParenth{\r@parenth[thick,cmex]}
\def\lowerParenth{\l@parenth[thick,cmex]}

\let\hEq\rEq
\let\vEq\uEq

\def\labelstyle{
\ifincommdiag
\textstyle
\else
\scriptstyle
\fi}
\let\objectstyle\displaystyle

\newdiagramgrid{pentagon}{0.618034,0.618034,1,1,1,1,0.618034,0.618034}{1.%
17557,1.17557,1.902113,1.902113}

\newdiagramgrid{perspective}{0.75,0.75,1.1,1.1,0.9,0.9,0.95,0.95,0.75,0.75}{0%
.75,0.75,1.1,1.1,0.9,0.9}

\diagramstyle[
dpi=300,
vmiddle,nobalance,
loose,
thin,
pilespacing=10pt,%
shortfall=4pt,
]

\ifx\CD@qK\else\CD@PK\relax\CD@FF\CD@e\fi\fi

\CD@vE\CD@hK\else\fi\else\fi\cdrestoreat
\begin{document}
\address{Mathematics Research Unit, Luxembourg University, Grand Duchy of Luxembourg and Independent University of Moscow,
B. Vlasievsky per., 11,
Moscow 121002,
Russia}
\email{alexey.kalugin@uni.lu}
\author{Alexey Kalugin}
\title{A note on a quantization via the Ran space.}
\dedicatory{}
\begin{abstract}
In a present note we give a new proof of Etingof-Kazhdan quantization theorem.
\end{abstract}
\maketitle

\section{Introduction}

\subsection{Main result} A quantization of Lie bialgebras is a functorial way to associate with a conilpotent\footnote{For a bialgebra $A$ with a reduced comultiplication $\overline{\Delta}$ (resp. a Lie bialgebra $\mathfrak g$ with a cobracket $\delta$) a canonical filtration $A_k:=\{x\in A^+\,|\, \overline{\Delta}^k(x)=0\}$ (resp. $\mathfrak g_k:=\{x\in \mathfrak g\,|\, \delta(x)=0\}$ )is exhaustive.} Lie bialgebra $\mathfrak g$ a certain conilpotent bialgebra $Q(\mathfrak g),$ called a \textit{quantization} of  $\mathfrak g,$ in a such way that a quantization induces an equivalence between the corresponding categories (an inverse functor is called \textit{dequantization}). After the original construction by Etingof and Kazhdan \cite{EK} many others proofs appeared latter (for example \cite{Tam}). In a preset paper we are going to give an another "motivic" construction of Lie bialgebras quantization.

\subsection{Technique} There are three main ingredients in our construction of a quantization (resp. a dequantization) of Lie bialgebras (resp. bialgebras):
\par\medskip

\textbf{The first one} is the equivalence between category of \textit{factorization algebras} on $Ran(\mathbb A^1)$  in the sense of Beilinson and Drinfeld and conilpotent Lie bialgebras, which can be though as a geometric instance of \textit{Quillen duality:}
$$\Lambda\colon LieB_c \overset{\sim}{\longrightarrow} {FA}_{un}(Ran(\mathbb A)),\qquad \mathfrak g\longmapsto \Lambda(\mathfrak g):=\{\Lambda(\mathfrak g))_I\}_{I\in Fin}.$$
Here $\Lambda(\mathfrak g)_I$ is a unipotent $\EuScript D$-module on $\mathbb A^I$ and $Fin$ is category of finite set with surjections as morphisms. 
This equivalence is achieved by using a \textit{quiver description} of a category of unipotent $\EuScript D$-modules on $\mathbb A^I$ following \cite{Kho}. For example when $|I|=2$ out of Lie bialgebra $\mathfrak g$ we can define the following double quiver:
$$
A^-\colon (\mathfrak g_1\otimes \mathfrak g_1\oplus \mathfrak g_1\otimes \mathfrak g_1)^{-\Sigma_2}\longleftrightarrow \mathfrak g_2\colon A^+,
$$
where $-\Sigma_2$ indicates skew invariants with respect to a symmetric group and operators are defined from a bracket and a cobracket:
$$
A^+:=\delta\oplus \delta,\qquad A^-:=\frac{1}{2}[-,-]\oplus [-,-].
$$
The quasi-inverse functor to $\Lambda$ is given by computing \textit{vanishing cycles} of $\EuScript D$-modules.
\par\medskip
\textbf{The second one} is the equivalence between category of \textit{factorizable sheaves} on $Ran(\mathbb A^1)$  (topological incarnations of factorization algebras) and conilpotent bialgebras, which can be though as a topological instance of \textit{Koszul duality}:
$$\Omega\colon AlgB_c \overset{\sim}{\longrightarrow} {FS}_{un}(Ran(\mathbb A)),\qquad A\longmapsto \Omega(A):=\{\Omega(A))_I\}_{I\in Fin}.$$
A perverse sheaf $\Omega(A)_I$ on $\mathbb A^I$ is defined using \textit{Cousin resolutions}. Note that our construction is essentially parallel to the recent result of Kapranov and Schechtman \cite{Shuffle}, where instead of conilpotent bialgebras they have considered graded ones. As it was explained in \textit{ibid} Cousin resolution approach is essentially equivalent to an elementary quiver description of categories of perverse sheaves on $Sym^n(\mathbb A^1)$\footnote{A quotient of $\mathbb A^n$ with respect to a natural symmetric group action} for small values of $n.$

For an example when $|I|=2$ using the classical quiver description of perverse sheaves on $Sym^{2}(\mathbb A)$ \cite{B2} we have a perverse sheaf $\EuScript E_2$ associated with a quiver:
$$
\overline{\Delta}\colon A_2\longleftrightarrow A_1\otimes A_1\colon m,
$$
where $m$ is a multiplication in bialgebra $A.$ Then we define $\Omega (A)_I:=p^!\EuScript E_2(A),$ where $p\colon \mathbb A^I \longrightarrow Sym^{|I|}(\mathbb A)$ is a natural projection.

\par\medskip
Inverse functor to $\Omega $ is defined by computing cohomology of perverse sheaves with support in real dimensional strata i.e. \textit{hyperbolic stalks} in the sense of \cite{KaSh}.

\par\medskip
\textbf{The third one} is given by the \textit{Riemann-Hilbert correspondence} which relates factorization algebras to factorizable sheaves:
$$
RH\colon {FA}_{un}(Ran(\mathbb A))\longrightarrow {FS}_{un}(Ran(\mathbb A)).
$$
Then we define a \textit{quantization functor} $Q$ and a \textit{dequantization functor} $DQ:$
$$
Q\colon LieB_c\longleftrightarrow AlgB_c\colon DQ,
$$
by the rule:
$$
 Q:=\Omega^{-1}\circ RH\circ\Lambda,\qquad DQ:= \Lambda^{-1}\circ RH^{-1}\circ \Omega.
$$

\subsection{Content}

In the first section, we recall some basic facts about conilpotent Lie bialgebras and collect all necessary information about unipotent $\EuScript D$-modules which we will need.
\par\medskip

In the second section, we recall some basic facts about conilpotent bialgebras and collect all necessary facts about perverse sheaves on symmetric powers that we need. In our exposition about perverse sheaves, we will mostly follow the work \cite{Shuffle}. We omit proofs and refer to \textit{ibid} for details.

\par\medskip
In the third subsection, we define a quantization and a dequantization functors and give some insights about how one can see associators in our setting.

\subsection{Acknowledgments} We wish to thank Boris Feigin, Sergei Merkulov and Nikita Markarian for fruitful discussions.
\par\medskip
Special thanks are due to Sasha Beilinson for his interest which has encouraged the author to write this note and Sergei Merkulov for his constant help and support during the preparation of this note.
\par\medskip
This work was supported by the FNR project number PRIDE$15/10949314/$GSM and partially supported by the Simons Foundation

\subsection{Notation} For a natural number $n$ by $[n]$ we denote a closed interval of natural numbers from $1$ to $n.$ Category of all finite sets and surjective maps between them will be denoted by $Fin.$ By $Vect$ we denote category of $\mathbb C$-vector spaces with the corresponding subcategory of finite-dimensional vector spaces $Vect^f.$ Let $I\in Fin$ be a finite set and $\{V\}_{i\in I}$ be a collection of vector spaces indexed by $I$ with the corresponding tensor product $\otimes_I V_i.$ For a permutation $\pi\in \Sigma_{|I|}$ we denote by $\sigma_{\pi}$ a standard twisting operator
$\sigma_{\pi}\colon \otimes_I V_i \longrightarrow \otimes_I V_i.$

\section{Quillen duality}

\subsection{Lie bialgebras}

Recall that a \textit{Lie bialgebra} is a vector space $\mathfrak g$ with skew symmetric operations $[-,-]\colon \mathfrak g\otimes \mathfrak g\longrightarrow \mathfrak g$ (bracket) and $\delta\colon \mathfrak g\longrightarrow \mathfrak g\otimes \mathfrak g$ (cobracket) which satisfies Jacobi and coJacobi identity. Moreover a bracket and a cobracket are compatible in the following way: $\delta$ is a $1$-cocycle of $\mathfrak g$ with values in $\mathfrak g\otimes \mathfrak g$ \cite{Drin}. We say that a Lie bialgebra $\mathfrak g$ is \textit{conilpotent} if a canonical filtration $\mathfrak g_k:=\{x\in \mathfrak g\,|\, \delta^k(x)=0\}$ is exhaustive and for every $k$ a component $\mathfrak g_k$ is finite dimensional.  An important property of this filtration is that a bracket and a cobracket respect this filtration: $\delta(\mathfrak g_k)\subset \summm_{i+j=k}\mathfrak g_i\otimes \mathfrak g_j,$  $[\mathfrak g_i,\mathfrak g_j]\subset \mathfrak g_{i+j-1}.$ For an element $x\in \mathfrak g$ we will use notation $\delta(x)_{{ij}}$ for a component of $\delta(x),$ which lies in the a $ \mathfrak g_i\otimes \mathfrak g_j.$ The category of conilpotent Lie bialgebras will be denoted by $LieB_c.$

\subsection{$\EuScript D$-modules}

 For every finite set $I\in Fin$ we consider the corresponding complex affine space $\mathbb A^I,$ equipped with a \textit{diagonal stratification} $\mathcal S_{\emptyset}=\{\Delta_{ij}\},$ where $\Delta_{ij}=z_i-z_j.$ The unique minimal closed stratum of $\mathcal S_{\emptyset}$ will be denoted by $\Delta$ and the unique maximal open stratum will be denoted by $U^{I}.$ We say that a stratum $S$ is adjusted to a stratum $S'$ if $S\subset \overline{S'}.$ Strata of dimension $k$ can be identified with a set of equivalence classes of surjections $[\pi],$ where we say that two surjections $\pi\colon I\twoheadrightarrow [k]$ and $\pi'\colon I\twoheadrightarrow [k]$ are equivalent if $\pi'=\pi\circ\tau,$ where $\tau$ is a bijection. Set of equivalence classes of surjections will be denoted by $Q(I)$ and has a structure of a poset where we say that $[\pi]\colon I\twoheadrightarrow T$ is less or equal to $[\pi']\colon I\twoheadrightarrow T'$ if $T$ is a quotient of $T'.$  Set of equivalence classes of surjections of a cardinality $k$ will be denoted by $Q(I,k).$ For an equivalence class $[\pi]$ the corresponding stratum will be denoted by $S_{[\pi]}.$ We will freely switch between these two notations. For $[\pi]\colon I\twoheadrightarrow T\in Q(I)$ and $i\neq j$ we associate another surjection $[\pi_{ij}]\colon I\twoheadrightarrow T_{ij},$ where $T_{ij}=T-\{ij\}\sqcup *,$ by the following rule let $\gamma_{ij}\colon T\twoheadrightarrow T_{ij}$ be a map sending $i$ and $j$ to $*$ and every $k$ not equal to $i$ or $j$ to $k,$ then we define $[\pi_{ij}]=\gamma_{ij}\circ [\pi].$ In this way we get all codimension one adjunctions.

By $\mathcal M_{{un}}^{dR}(\mathbb A^{I},\mathcal S_{\emptyset}),$ we denote a category of $\EuScript D$-modules on $\mathbb A^I$ which are generated by delta functions with a support on diagonals. This is a subcategory of the category of holonomic $\EuScript D$-modules smooth along stratification $\mathcal  S_{\emptyset}.$ Objects of the category  $\mathcal M_{{un}}^{dR}(\mathbb A^{I},\mathcal S_{\emptyset})$ are called
\textit{unipotent $\EuScript D$-modules.}

For every finite set $I$ with an affine space $\mathbb A^I$ we can associate a quiver $\Gamma_I$ whose vertices correspond to elements of $\mathcal S_{\emptyset}$ and arrows correspond to an adjusted strata of codimension one:
$$[\alpha]\rightarrow [\beta] \Leftrightarrow {S_{[\beta]}}\subset \overline{S}_{[\alpha]},\quad\mathrm{codim}_{\overline{S}_{[\alpha]}} {S_{[\beta]}}=1.$$

We will use a notation $[\alpha_1]\circ \dots \circ [\alpha_n]$ for a minimal path in a quiver which passes through vertices $[\alpha_i],$ such that it starts at $[\alpha_1]$ and ends at $[\alpha_n].$

Following \cite{Kho} with a quiver $\Gamma_I$ we can associate a category $Rep_{un}(\Gamma_I)$ of \textit{double representations of quiver} $\Gamma_I$ with certain relations. An object $V$ of this category is by the definition consists of a vector space $ V_{[\beta]},$ where $[\beta]$ is a vertex of $\Gamma_I$ and a pair of linear operators $A^-_{[\lambda][\beta]}\colon V_{[\beta]}\longleftrightarrow  V_{[\lambda]}\colon A^+_{[\beta][\lambda]}$ for every arrow $[\beta]\rightarrow[\lambda].$ These linear operators must satisfy relations which we omit (See \textit{ibid}). The main result of \textit{ibid} (Theorem $2.1$) gives an \textit{explicit equivalence} between category $Rep_{un}(\Gamma_I)$ and the category $\mathcal M_{{un}}^{dR}(\mathbb A^{I},\mathcal S_{\emptyset}).$ For a path $[\tau]=[\alpha_1]\circ \dots \circ [\alpha_n]$ we will use notation $A^+_{\tau}$ for a composition $A_{[\alpha_1],[\alpha_2]}^+\circ \dots \circ A_{[\alpha_{n-1}],[\alpha_n]}^+$ of operators.

\begin{remark}\label{df} Consider a \textit{specialization functor:}
$$Sp_{\Delta}\colon \mathcal M_{un}^{dR}(\mathbb A^I,\mathcal S_{\emptyset})\longrightarrow \mathcal M_{un}^{dR}(T_{\Delta}\mathbb A^I,\mathcal S_{\emptyset})$$
along minimal diagonal $\Delta\subset \mathbb A^I.$ Then we define an exact functor from a category unipotent $\EuScript D$-modules to a category finite dimensional vector spaces by the rule:
$$
\omega^{dR}_I\colon \mathcal M_{un}^{dR}(\mathbb A^I,\mathcal S_{\emptyset})\longrightarrow Vect^f,\quad \omega^{dR}:=H^0_{dR}(T_{\Delta}\mathbb A^I,-)
$$
Under the equivalence from \textit{ibid} we have the identical isomorphism between a functor $\omega^{dR}$ and a forgetful functor from $Rep_{un}(\Gamma_I)$ to vector spaces. Thus a quiver description of category $\mathcal M_{un}^{dR}(\mathbb A^I,\mathcal S_{\emptyset})$ can be considered as an instance of \textit{Tannakian duality}. Note that by a construction $\omega^{dR}_I$ carries a natural $\mathbb C^{\times}$-action i.e. grading. We will denote the first grading component by $\Phi_I,$ which corresponds to the vector space associated with the smallest diagonal in the quiver description. It is easy to see that this construction extends to an exact functor:
$$
\Phi_I\colon \mathcal M_{un}^{dR}(\mathbb A^I,\mathcal S_{\emptyset})\longrightarrow Vect^f
$$

\end{remark}
For every surjection $\pi\colon J\twoheadrightarrow I$ of finite sets we have the corresponding diagonal embedding $\Delta^{(\pi)}\colon \mathbb A^I\hookrightarrow \mathbb A^J,$ which induces a pair of adjoint functors for $\EuScript D$-modules:
$$\Delta^{(\pi)}_*\colon \mathcal M^{dR}_{un}(\mathbb A^I,\mathcal S_{\emptyset})\longleftrightarrow \mathcal M^{dR}_{un}(\mathbb A^J,\mathcal S_{\emptyset})\colon \Delta^{(\pi)!},$$
where a pushforward functor is exact and a $!$-pullback functor is left exact. Let us recall following \cite{KV} a definition of a \textit{quiver pushforward functor} $\Delta^{(\pi)}_*:$

\par\medskip

For an element $M\in Rep_{un}(\Gamma_I)$ we define a new quiver representation $\Delta^{(\pi)}_*M\in Rep_{un}(\Gamma_J)$ by the following rule. For every stratum $[\gamma]\colon I\twoheadrightarrow T$ of $\mathbb A^I$ we define a stratum $[\gamma_{\pi}]$ of $\mathbb A^J$ by taking a composition with $\pi.$ Thus we set $\Delta^{(\pi)}_*M_{[\gamma_{\pi }]}:=M_{[\gamma]}.$ For all other strata we define $\Delta^{(\pi)}_*M$ to be zero. For a pair of non zero vector spaces we take corresponding operators from $M$ and for all others we set operators to be zero. It is easy to see that $\Delta^{(\pi)}_*M\in Rep_{un}(\Gamma_J)$ and the under equivalence between quiver representations and $\EuScript D$-modules quiver pushforward corresponds to $\EuScript D$-modules pushforward. Now we will define a \textit{quiver $!$-pullback functor}:

\begin{Def} For every $\pi\colon J\twoheadrightarrow I$ we define quiver $\Delta^{(\pi)!}M$ by the following rule. It $\pi$ is a bijection we set $(\Delta^{(\pi)!}M)_{[\alpha]}:=M_{[\alpha_{\pi}]}$ with induced operators. In the case when $\pi$ is not a bijection we set:
\par\medskip
 For $[\alpha]\in \Gamma_I$ we define a vector space:
$$
(\Delta^{(\pi)!}M)_{[\alpha]}:=\bigcap_{\tau}\mathrm {Ker}(A^+_{\tau})
$$
where intersection is taken over all \textit{admissible paths} ${\tau:=[\gamma]\circ [\alpha_j]}\circ \dots  \circ {[\alpha_1]\circ [\alpha_{\pi}]},$ that are paths from a stratum $[\gamma]$ of dimension $|J|-\dim [\alpha]$ to a stratum $[\alpha_{\pi}]$ in $\Gamma_J$ of minimal length such that $[\alpha_i]$ can not be factored into $J\twoheadrightarrow I \twoheadrightarrow T.$ Operators $A^+$ and $A^-$ are induced from operators in $M.$

\end{Def}

\begin{Prop}\label{pull} The construction above defines functor $\Delta^{(\pi)!}$ which is the right adjoint to a quiver pushforward
$$\Delta^{(\pi)}_*\colon Rep_{un}(\Gamma_I)\longleftrightarrow Rep_{un}(\Gamma_J)\colon \Delta^{(\pi)!},$$
and thus under the equivalence between double quiver representations and unipotent $\EuScript D$-modules a quiver $!$-pullback functor corresponds to a $!$-pullback functor for $\EuScript D$-modules.

\end{Prop}

\begin{proof} First we will show that operators in a quiver representation $M$ induce will defined operators in $\Delta^{(\pi)!}M.$ In the case when $\pi$ is a bijection there is nothing to prove. Let us consider a case when $\pi$ is not a bijection:

\par\medskip

Without loss of a generality we can consider a case when $|J|\geq 3$ (in the case when $|J|<3$ there is nothing to prove). Let $[\pi]\rightarrow [\pi']$ be an adjusted pair of strata in $\Gamma_I.$ Operator $A^+_{[\pi],[\pi']}\colon M_{[\pi']}\longrightarrow M_{[\pi]}$ descends to an operator $\Delta^{(\pi)!}M_{[\pi']}\longrightarrow\Delta^{(\pi)!}M_{[\pi]},$ because we have the condition one from \cite{Kho}. More precisely let $\tau=[\gamma]\circ \dots \circ [\alpha_1]\circ [\pi]$ be an arbitrary admissible path to the stratum $[\pi]$ then we have an admissible path $\tau_1$ (resp. $\tau_2$) to the stratum $[\pi']$ which passes through $[\alpha_1]$ and equal to $\tau $ everywhere except vertex $[\beta]$ (resp. $[\delta]$) (these strata have the same dimension as $[\alpha_1]).$ We have a picture like this:
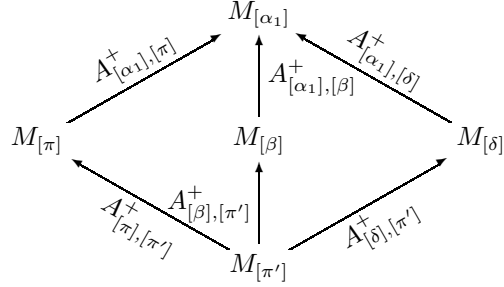
\begin{equation*}
\begin{diagram}[height=2.4em,width=4.2em]
 &  & M_{[\alpha_1]}  & & \\
  & \ruTo^{A^+_{[\alpha_1],[\pi]}} & \uTo_{A^+_{[\alpha_1],[\beta]}} & \luTo^{A^+_{[\alpha_1],[\delta]}} &\\
M_{[\pi]} &  & M_{[\beta]} &   &M_{[\delta]}\\
  & \luTo_{A^+_{[\pi],[\pi']}} & \uTo^{A^+_{[\beta],[\pi']}} & \ruTo_{A^+_{[\delta],[\pi']}} &\\
  &  & M_{[\pi']}  & & \\
\end{diagram}
\end{equation*}
Let $x\in (\Delta^{(\pi)!}M)_{[\pi']},$ then in particular $x\in \mathrm {Ker}A^+_{\tau_1}$ and $x\in \mathrm {Ker}A^+_{\tau_1}.$ Then by the first condition from \textit{ibid} we have: $$A^+_{[\alpha_1],[\pi]}\circ A^+_{[\pi],[\pi]'}=-A^+_{[\alpha_1],[\beta]}\circ A^+_{[\beta],[\pi]'} -A^+_{[\alpha_1],[\delta]}\circ A^+_{[\delta],[\pi]'}$$ and thus $A^+_{\tau}\circ A^+_{[\pi],[\pi']}(x)=-(A^+_{\tau_1}+A^+_{\tau_2})(x)=0.$

An operator $A^-_{[\pi'],[\pi]}\colon M_{[\pi]}\longrightarrow M_{[\pi']}$ descends to an operator $\Delta^{(\pi)!}M_{[\pi]}\longrightarrow\Delta^{(\pi)!}M_{[\pi']},$ because we have the last condition from \textit{ibid}. More precisely let $\tau=[\gamma]\circ \dots \circ [\alpha_1][\pi']$ be an arbitrary admissible path to the stratum $[\pi'].$ Then we have an adjusted strata $[\pi_1]\rightarrow [\pi].$ We have a picture like this:
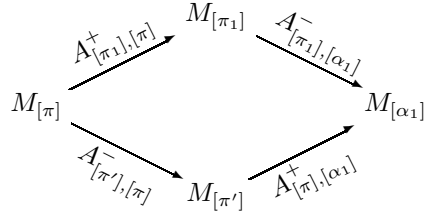
\begin{equation*}
\begin{diagram}[height=1.7em,width=3.4em]
 &  & M_{[\pi_1]}  & & \\
  & \ruTo^{A^+_{[\pi_1],[\pi]}} &  & \rdTo^{A^-_{[\pi_1],[\alpha_1]}} &\\
M_{[\pi]} &  &  &   & M_{[\alpha_1]}\\
  & \rdTo_{A^-_{[\pi'],[\pi]}} &  & \ruTo_{A^+_{[\pi],[\alpha_1]}} &\\
  &  & M_{[\pi']}  & & \\
\end{diagram}
\end{equation*}
Now due to the last condition from \textit{ibid} we have the following:
$$A^+_{[\alpha_1],[\pi']}\circ A^-_{[\pi'],[\pi]}=-A^-_{[\alpha_1],[\pi_1]}\circ A^+_{[\pi],[\pi_1]}.$$
Let $x\in (\delta^{(\pi)!}M)_{[\pi]}$ if $A^-_{[\alpha_1],[\pi_1]}\circ A^+_{[\pi],[\pi_1]}$ does not vanish we pick up an adjusted to $[\alpha_1]$ vertex $[\alpha_2]$ in a path $\tau$ and consider another quadruple of strata $[\alpha_1], [\pi_1], [\alpha_2]$ and $[\pi_2],$ where $[\alpha_2]\rightarrow [\pi_2]\rightarrow[\pi_2]\rightarrow [\alpha_1]$ and repeat the procedure. At some point we will reach the strata of dimension $|J|-\dim [\alpha]$ and thus obtain the admissible path $\tau'$ to a stratum $[\pi].$  Thus an operator $A_{\tau'}^+$ must vanish on $x$ and by the last condition from \textit{ibid} we have  $A^-_{[\pi'],[\pi]}(x)\in \Delta^{(\pi)!}M_{[\pi']}.$
All conditions from \textit{ibid} for $A^+$ operators and $A^-$ operators are satisfied because they descend from a double quiver representation $M.$ Moreover it is easy to see that $\Delta^{(\pi)!}$ extends to a functor which is the right adjoint to a quiver direct image functor $\Delta^{(\pi)}_*$ by the construction. Thus by the uniqueness of the right adjoint functor we obtain that under the equivalence with $\EuScript D$-modules a $!$-quiver pullback $\Delta^{(\pi)!}$ corresponds to a $!$-pullback functor for $\EuScript D$-modules.
\end{proof}

With every surjection $\pi\colon J\twoheadrightarrow I$ we can associate a space $U^{(\pi)}:=\{(x_j)\,|\,x_{i}\neq x_{i'}\, \mathrm{if}\, \pi(i)\neq \pi(i')\}$ with a corresponding open inclusion $j^{(\pi)}\colon U^{(\pi)}\hookrightarrow \mathbb A^J.$ A space $U^{(\pi)}$ inherits a diagonal stratification from $\mathbb A^J:$ strata of  $U^{(\pi)}$ are given by equivalences classes of surjections $[\gamma]\colon J\twoheadrightarrow T$
which are greater or equal to $[\pi]\colon J \rightarrow I.$ We have an exact $*$-pullback functor along map $j^{(\pi)}$ for $\EuScript D$-module:
$$j^{(\pi)*}\colon \mathcal M^{dR}_{un}(\mathbb A^J,\mathcal S_{\emptyset})\longrightarrow\mathcal M^{dR}_{un}(U^{(\pi)},\mathcal S_{\emptyset}),$$
Since $U^{(\pi)}$ is a stratified subspace of $\mathbb A^J$ methods of \cite{Kho} can be used to study category of unipotent $\EuScript D$-modules on it. Let us recall following \cite{KV} a definition of a \textit{quiver $*$-pullback functor} $j^{(\pi)*}:$

\par\medskip

For an element $M\in Rep_{un}(\Gamma_J)$ we define a new quiver representation  $j^{(\pi)*}M\in Rep_{un}(\Gamma_{(\pi)})$ by the following rule. For every stratum $[\gamma]\colon J\twoheadrightarrow T$ of $U^{(\pi)}$ we define $j^{(\pi)*}M_{[\gamma]}:=M_{[\gamma]}.$ Operators are induced from operators in $M.$
\par\medskip

Let $\pi \colon J \twoheadrightarrow I$ be a surjection. For unipotent $\EuScript D$-modules we have a standard operation of the exterior tensor product, which is a functor
$$\boxtimes_{I}\colon \times_I\mathcal M_{un}^{dR}(\mathbb A^{J_i},\mathcal S_{\emptyset})\longrightarrow \mathcal M_{un}^{dR}(\mathbb A^{J},\mathcal S_{\emptyset}),\quad J_i:=\pi^{-1}(i).$$
 Following \cite{fdmod} we define a \textit{quiver exterior tensor product}:
\par\medskip
Let $[\gamma]\colon J \twoheadrightarrow T$ be a stratum, then if $\pi$ can be factored as $\pi= \gamma\circ \beta,$ where $\beta\colon T\twoheadrightarrow I$ we define
$$
(\boxtimes_I M_i)_{[\gamma]}:=\otimes_I M_{[\gamma_i]},
$$
where $\gamma_i\colon J_i\twoheadrightarrow T_i,$ and $T_i=\beta^{-1}(i).$ Otherwise we define $(\boxtimes_I M_i)_{[\gamma]}$ to be zero. Operators in a quiver representation are defined by the following rule: if $[\gamma]$ and $[\gamma']$ is a pair of adjusted strata which satisfy a property above we define $A^+$ operators and $A^-$ operators as a tensor product of appropriate operators acting between vector spaces assigned to a strata $[\gamma_i]$ and $[\gamma_i'],$ otherwise operators are defined to be zero. It is easy to see that $\boxtimes_I M_i$ is an element of $Rep_{un}(\Gamma_J)$ and moreover $\boxtimes_I$ extends to a functor and under the equivalence between quiver representations and $\EuScript D$-modules corresponds to $\EuScript D$-module exterior tensor product.

\subsection{Factorization algebra $\Lambda(\mathfrak g)$} In this subsection out of conilpotent Lie bialgebra we built a certain factorization algebra.

\begin{Prop}\label{Gluing} With every conilpotent Lie bialgebra $\mathfrak g$ we associate unipotent $\EuScript D$-module $\Lambda(\mathfrak g)_I$ on $\mathbb A^I.$
\end{Prop}
\begin{proof} We will use quiver description of a category  $\mathcal M_{un}^{dR}(\mathbb A^I,\mathcal S_{\emptyset})$ from \cite{Kho}:
\par\medskip

Let $[\pi]\in Q(I,k)$ be a $k$-dimensional strata, then with an element of the equivalence class $\pi\colon I\twoheadrightarrow [k]\in [\pi]$ we can associate a vector space $\mathfrak g_{\pi}=\mathfrak g_{i_1}\otimes \dots \otimes \mathfrak g_{i_k},$ where  $i_j=|\pi^{-1}(j)|.$ Hence for a stratum $[\pi]$ we define the corresponding vector space by the rule
$$\mathfrak g_{[\pi]}:=\summm_{\pi\colon I\twoheadrightarrow [k] \in [\pi]}\mathfrak g_{\pi}^{-\Sigma_{k}}.$$ For an adjusted pair of strata $[\pi]\twoheadrightarrow [\pi_{pq}]$ we choose a representative $I\twoheadrightarrow [k]\in [\pi],$ such that $p$ will be the first coordinate and $q$ will be the second one. Then we can define a map
 $A^-_{\pi_{pq},\pi}\colon \mathfrak g_{\pi}\longrightarrow \mathfrak g_{\pi_{pq}}$ by the rule $([-,-]\otimes id \otimes \dots \otimes id).$ Thus we define a map $\mathfrak g_{\pi}\longrightarrow \mathfrak g_{[\pi_{pq}]}$ as:
$$
A^-_{[\pi_{pq}],\pi}=\summm_{\sigma\in \Sigma_{k-1}} \tau_{\sigma} \circ A^-_{\pi_{pq},\pi}.
$$
And finally we define a map $A^-_{[\pi_{pq},T]}\colon \mathfrak g_{[\pi]}\longrightarrow \mathfrak g_{[\pi_{pq}]}$ by the rule:
 $$A^-_{[\pi_{pq}],[\pi]}:=\frac{1}{k!}\summm_{\sigma\in \Sigma_{k}} A^-_{[\pi_{pq}],\pi}\circ \tau_{\sigma}.$$ Analogically we define a map $A^+_{[\pi],[\pi_{pq}]}\colon \mathfrak g_{[\pi_{pq}]}\longrightarrow \mathfrak g_{[\pi]}$ but instead of a bracket we use a cobracket. Thus we set $A^+_{\pi,\pi_{pq}}:=\delta_{i_1,i_2}\otimes id \otimes \dots \otimes id$ and then:
$$
A^+_{[\pi],\pi_{pq}}:=\summm_{\sigma\in \Sigma_{k}} \tau_{\sigma} \circ A^+_{\pi,\pi_{pq}},
$$
and thus:
$$A^+_{[\pi],[\pi_{pq}]}:=\frac{1}{(k-1)!}\summm_{\sigma\in \Sigma_{k-1}} A^+_{[\pi],\pi_{pq}}\circ \tau_{\sigma}.$$
\par\medskip

We need to check conditions from \textit{ibid}:
\par\medskip
The first two identities follow from Jacobi and coJacobi identities for a Lie bialgebra $\mathfrak g.$ Let us consider a case of a map $A^-_{[\pi_{ij},\pi]}.$ Choice a representative $\pi \colon I \twoheadrightarrow [k]$ such that $i,j,k\in [k]$ will be a first three coordinates. Then $A^-_{[\pi_{jk}],\pi}$ acts on $\pi$ component as $A^-_{[\pi_{ij}],\pi}\circ \tau_{(ijk)}$ and $A^-_{[\pi_{ki}],\pi}$ acts on $\pi$ component as $A^-_{[\pi_{ij}],\pi}\circ \tau_{(ijk)^2}.$ Hence from Jacobi identity we obtain that:
$$
A^-_{[\pi_{ijk}],[\pi_{jk}]}\circ A^-_{[\pi_{jk}],[\pi]}+A^-_{[\pi_{jki}],[\pi_{ki}]}\circ A^-_{[\pi_{ki}],[\pi]}+A^-_{[\pi_{ijk}],[\pi_{ij}]}\circ A^-_{[\pi_{ij}],[\pi]}=0.
$$
Case of $+$-map can be treated analogically and follows from coJacobi identity for a Lie cobracket.
\par\medskip
The third condition from \textit{ibid} follows from $1$-cocycle property for a Lie bialgebra. Choice a representative $\pi \colon I \twoheadrightarrow [k]$ such that $i,j,k\in [k]$ will be first three coordinates. Then $A^+_{\pi_{jk},\pi_{ijk}}\circ A^-_{\pi_{ijk},\pi_{ij}}$ acts between $\pi_{ij}$ and $\pi_{jk}$ component as $\tau_{(ij)}\circ A^+_{\pi_{ij},\pi_{ijk}}\circ A^-_{\pi_{ijk},\pi_{ij}}$ and $A^-_{\pi,\pi_{jk}}\circ A^+_{\pi,\pi_{ij}}$
acts as $A^-_{\pi,\pi_{ij}}\circ\tau_{(ijk)}\circ  A^+_{\pi,\pi_{ij}}.$ Thus by $1$-cocycle property of a Lie bialgebra we obtain that:
$$
A^+_{[\pi_{jk}],[\pi_{ijk}]}\circ A^-_{[\pi_{ijk}],[\pi_{ij}]}+A^-_{[\pi],[\pi_{jk}]}\circ A^+_{[\pi],[\pi_{ij}]}=0.
$$
In remains to show that monodromy operators are nilpotent \textit{ibid}. This obviously follows from a Lie property of a canonical filtration.
\end{proof}

\begin{Prop}\label{ran}
For every $\pi\colon J\twoheadrightarrow I$ conilpotent filtration induces a natural isomorphism:
$$\theta(\pi)\colon \Lambda(\mathfrak g)_I \overset{\sim}{\longrightarrow}  \Delta^{(\pi)!}\Lambda(\mathfrak g)_J,$$
compatible with compositions of surjections.

\end{Prop}

\begin{proof} We will use a description of quiver $!$-pullback from Proposition \ref{pull}.  Let us first consider a case when $\pi\colon I\overset{\sim}{\longrightarrow} I$ is a bijection. For every stratum $[\gamma]\colon I\twoheadrightarrow T$ we have a new stratum $[\gamma_{\pi}].$ Then we define an operator:
$$
\rho\colon \mathfrak g_{[\gamma]}\longrightarrow \mathfrak g_{[\gamma_{\pi}]},
$$
for a representative $\gamma\in [\gamma]$ by taking a signed twist morphism $\tau_{\pi}$ associated with a permutation $\pi:$
$$
\tau_{\pi}\colon \mathfrak g_{i_1}\otimes \dots \otimes \mathfrak g_{i_k}\longrightarrow \mathfrak g_{i_1}\otimes \dots \otimes \mathfrak g_{i_k}.
$$
Due to skew symmetry of a bracket and a cobracket isomorphisms $\rho$ extend to an isomorphism of quivers $\Lambda(\mathfrak g)_I \overset{\sim}{\longrightarrow}  \Delta^{(\pi)!}\Lambda(\mathfrak g)_I.$ It is easy to see that defined isomorphisms are compatible with compositions of bijections, i.e. $\EuScript D$-module $\Lambda(\mathfrak g)_I$ is equivariant.
\par\medskip
Now we need to consider a case of a non bijective surjection $\pi\colon J \twoheadrightarrow I.$ For a given stratum $[\alpha]$ in $\mathbb A^I$ and any stratum $[\gamma]$ of dimension $|J|-\dim [\alpha]$ in $\mathbb A^J$ choose a representatives $\alpha$ and $\gamma.$ Denote by $\tau$ a sequence of elementary surjections from $\gamma$ to $\alpha_{\pi}.$ Denote by $A^+_{\tau}$ a component of an operator $A^+_{[\tau]}$ which acts between $\mathfrak g_{\alpha_{\pi}}$ and $\mathfrak g_{\gamma}.$ Operator $A^+_{\tau}$ acts as a composition of appropriate components of $\delta$ and signed twisting maps. We have a canonical colimit map $can\colon \mathfrak g_{\alpha} \longrightarrow \mathfrak g_{[\alpha]}.$
Then it is easy to see that we have an isomorphism $can\colon \mathrm {Ker} A^+_{\tau}\overset{\sim}{\longrightarrow} \mathrm {Ker} A^+_{[\tau]}.$
Since we have proved that $\EuScript D$-module $\Lambda (\mathfrak g)_J$ is $\Sigma_{|J|}$-equivariant it is enough to consider equivalence classes of admissible pathes with respect to a natural symmetric group action on $\mathbb A^J.$ Thus by a skew symmetry of a cobracket and remarks above for a stratum $[\alpha]$ in $\mathbb A^I$ we obtain:
$$
\bigcap_{[\tau]/ \Sigma_{|J|}}\mathrm {Ker} A^+_{[\tau]}\overset{\sim}{\longrightarrow} \mathfrak g_{[\alpha]}.
$$
It is easy to see that constructed isomorphisms extends to an isomorphism of quivers which are compatible with compositions of surjections.

\end{proof}

\begin{remark} Following Beilinson and Drinfeld \cite{BD} it is natural to call $\Lambda(\mathfrak g):=\{\Lambda(\mathfrak g)_I,\theta(\pi)\}_{Fin^{\circ}}$ an \textit{unipotent left $\EuScript D$-module on Ran prestack} of $\mathbb A^1.$ These objects with obvious morphisms form a category which will be denoted by $\mathcal M_{un}^{dR}(Ran(\mathbb A^1)).$

\end{remark}

\begin{Prop}\label{fact} Unipotent $\EuScript D$-module $\Lambda(\mathfrak g)$ on $Ran(\mathbb A)$ has the following additional structure:

\begin{itemize}
\item For every surjection $\pi \colon J\twoheadrightarrow I,$ we have a natural isomorphism:
$$c_{\pi}\colon j^{(\pi)*}\Lambda(\mathfrak g)_J \overset{\sim}{\longrightarrow} j^{(\pi)*}\boxtimes_{I} \Lambda(\mathfrak g)_{J_i},\qquad J_i:=\pi^{-1}(i).$$
\item  Compatibilities between $c'$s and compositions of surjections of finite sets.
\item Compatibilities between maps $c$ and $\theta.$
\end{itemize}

\end{Prop}

\begin{proof}

Let $\pi\colon J\twoheadrightarrow I$ be a surjection and let $[\gamma]\colon J \twoheadrightarrow T$ be a stratum in $\mathbb A^J,$ then if $\pi$ can be factored as $\pi=\beta\circ \gamma,$ where $\beta\colon T\twoheadrightarrow I$ we have
$$
(\boxtimes_{I} \Lambda(\mathfrak g)_{J_i})_{[\gamma]}:=\otimes_I \mathfrak g_{[\gamma_i]}=\otimes_I (\summm_{J_i\twoheadrightarrow T_i}\otimes_{t\in T_i} \mathfrak g_{(J_i)_t})^{-\Sigma},
$$
where $\gamma_i\colon J_i\twoheadrightarrow T_i,$ and $T_i=\beta^{-1}(i).$ On another hand we have:
$$(\Lambda (\mathfrak g)_{J})_{[\gamma]}=( \summm_{\gamma \colon J\twoheadrightarrow T \in [\gamma]}\otimes_T  \mathfrak g_{J_t})^{-\Sigma_{|T|}}.$$ Note that a diagram which underlies $(\boxtimes_{I} \Lambda(\mathfrak g)_{J_i})_{[\gamma]}$ is naturally a subdiagram of a diagram which underlies  $(\Lambda (\mathfrak g)_{J})_{[\gamma]}$ thus we obtain a canonical map:
$$can\colon (\Lambda (\mathfrak g)_{J})_{[\gamma]}\overset{\sim}{\longrightarrow} (\boxtimes_{I} \Lambda(\mathfrak g)_{J_i})_{[\gamma]},$$
which is an isomorphism obviously. According to the definition of a quiver pullback $j^{(\pi)}$ it is easy to see that canonical maps define an isomorphism of quivers:
$$c_{\pi}\colon j^{(\pi)*}\Lambda(\mathfrak g)_J \overset{\sim}{\longrightarrow} j^{(\pi)*}\boxtimes_{I} \Lambda(\mathfrak g)_{J_i},$$
which satisfies all properties listed above.

\end{proof}

\begin{remark}

Following Beilinson and Drinfeld \cite{BD} we say that an object $\EuScript N$ in $\mathcal M_{un}^{dR}(Ran(\mathbb A),\mathcal S_{\emptyset})$ is an \textit{unipotent factorization algebra} on $\mathbb A^1$ if conditions above are satisfied. Factorization algebras naturally organize into a category denoted by $FA_{un}(Ran(\mathbb A^1)).$ It is obvious that the construction defined above is functorial i.e. we have a functor:
$$\Lambda\colon LieB_c \longrightarrow {FA}_{un}(Ran(\mathbb A)),\qquad \Lambda\colon \mathfrak g\longmapsto \Lambda(\mathfrak g).$$

\end{remark}

\subsection{Vanishing cycles} In this subsection we will prove that functor $\Lambda$ induces equivalence between conilpotent Lie bialgebras and unipotent factorization algebras on $\mathbb A:$

\begin{Prop} We have the following functor:
\begin{equation}
\Phi_{Ran}\colon  FA_{un}(Ran(\mathbb A))\longrightarrow  LieB_c,\qquad \Phi_{Ran}:=\lim_{Fin} \Phi_{I}.
\end{equation}
Such that $\Phi_{Ran}$ is a quasi inverse a functor to $\Lambda.$
\end{Prop}

\begin{proof} Let $\EuScript K\in FA_{un}(Ran(\mathbb A))$ then for every $I\in Fin$ we define a vector space $\Phi_{Ran}(\EuScript K)_{|I|}:=\Phi_I(\EuScript K_I)^{\Sigma_{|I|}}.$  Note that for every $\pi \colon J\twoheadrightarrow I$ we have an inclusion $\Delta_*^{(\pi)}\EuScript K_I \hookrightarrow \EuScript K_J,$ which comes as an adjoint map to an isomorphism from Proposition \ref{ran}. These maps descends to an injection:
$\Phi_{Ran}(\EuScript K)_{|I|}\hookrightarrow \Phi_{Ran}(\EuScript K)_{|J|}.$ Thus we have:
$$\Phi_{Ran}(\EuScript K)_{1}\subset\dots  \subset\Phi_{Ran}(\EuScript K)_{n}\subset\dots \subset \Phi_{Ran}(\EuScript K).$$
Note that $A^+$ operators together with factorization isomorphisms from Proposition \ref{fact} define the following map:
$$
\delta_{|I'|,|I''|}\colon \Phi_{Ran}(\EuScript K)_{|I|}\longrightarrow \Phi_{Ran}(\EuScript K)_{|I'|}\otimes \Phi_{Ran}(\EuScript K)_{|I''|},\qquad {|I'|+|I''|=|I|}.
$$
From the quiver description of a category of unipotent $\EuScript D$-modules, we obtain that $\delta$ satisfies coJacobi identity. Also from the Koszul sign rule for the exterior tensor product $\boxtimes$ of $\EuScript D$-modules we obtain that $\delta$ is skew-symmetric.
\par\medskip
Analogically using $A^-$operators and factorization isomorphisms we define the following map:
$$
[-,-]_{|I'|,|I''|}\colon  \Phi_{Ran}(\EuScript K)_{|I'|}\otimes \Phi_{Ran}(\EuScript K)_{|I''|}\longrightarrow \Phi_{Ran}(\EuScript K)_{|I|},\quad |I'|+|I''|=|I|.
$$
Which is skew-symmetric and satisfies Jacobi identity. Operators $\delta$ and $[-,-]$ are compatible in the sense of Drinfeld ($\delta$ is $1$-cocycle for the bracket $[-,-]$) because of the last condition of quiver description of category of unipotent $\EuScript D$-modules.

Thus we have defined a structure of conilpotent Lie bialgebra on $\Phi_{Ran}(\EuScript K).$ It is obvious that $\Phi_{Ran}$ is a quasi-inverse functor to $\Lambda.$
\end{proof}

\begin{remark} Analogically to Remark \eqref{df} we can define functor $\omega^{dR}$ on a category of $\EuScript D$-modules on a Ran prestack:
\begin{equation}\label{cp}
\omega_{Ran}^{dR} \colon \mathcal M_{un}^{dR}(Ran(\mathbb A),\mathcal S_{\emptyset}) \longrightarrow  Vect,\qquad \omega^{dR}_{Ran}:=\lim_{Fin} \omega^{dR}_{I}.
\end{equation}
Denote by $AlgCB_{c}$ is a category of \textit{conilpotent coPoisson bialgebras}. We have a primitive elements functor $Prim\colon AlgCB_{c}\longrightarrow LieB_c,$ defined by taking the underlying Lie algebra i.e. $Prim(A):=\{x\in A^+\,|\, \overline{\Delta}(x)=0\}.$ One can show that a functor \eqref{cp} makes the following diagram into the commutative:
\begin{equation*}
\begin{diagram}[height=2.4em,width=6em]
FA_{un}(Ran(\mathbb A)) &   \rTo^{\omega_{Ran}^{dR}}   &   AlgCB_{c} &  \\
 & \rdTo_{\sim}^{\Phi_{Ran}}  &  \dTo_ {\sim}^{Prim}  && \\
 &      &  LieB_c    \\
\end{diagram}
\end{equation*}
Note that a primitive elements functor $Prim$ defines an equivalences of categories. This is corollary from the Milnor-Moore equivalence \cite{Drin}.

\end{remark}

\section{Koszul duality}

\subsection{Hopf algebras}

For a bialgebra $A$ we denote by $\Delta\colon A\longrightarrow A\otimes A$ the corresponding comultiplication with the corresponding reduced comultiplication $\overline{\Delta}(x)=\Delta(x)-x\otimes 1-1\otimes x$ and by $m\colon A\otimes A\longrightarrow A$ we denote the corresponding multiplication and $\varepsilon\colon A\longrightarrow \Bbbk$ the counit. By $A^+:=\mathrm{Ker}(\varepsilon)$ we denote an augmentation ideal. An operator $\overline{\Delta}$ defines a coalgebra structure on $A^+.$
Let $A$ be a \textit{conilpotent bialgebra}, that is a bialgebra such that a canonical filtration $A_k:=\{x\in A^+\,|\,\overline{\Delta}^k(x)=0\}$ is exhaustive and $\dim A_k<\infty.$ Note that $m$ and $\overline{\Delta}$ respect a canonical filtration i.e. $m(A_i,A_j)\subset A_{i+j}$ and $\overline{\Delta}(A_k)\subset \summm_{i+j=k} A_i\otimes A_j$ and by $\overline{\Delta}_{ij}$ we denote a component which lies in $A_i\otimes A_j.$ Conilpotent bialgebras form a category $AlgB_c$ in the obvious way which is equivalent to the category of conilpotent Hopf algebras $AlgH_c.$ Further we will not distinguish between conilpotent bialgebras and conilpotent Hopf algebras.

\subsection{Perverse sheaves} With every $I\in Fin$ we associate complex affine space $\mathbb A^I$ (in analytic topology) equipped with a diagonal stratification $\mathcal S_{\emptyset}.$ Let $\mathcal M_{un}^B(\mathbb A^I,\mathcal S_{\emptyset})$ be a category of unipotent perverse sheaves on $\mathbb A^I$ smooth with respect to a diagonal stratification.
\par\medskip
For every $I$ we will also consider $\mathbb A^I_{\mathbb R}$ with a real diagonal stratification denoted by $\mathcal S_{\emptyset,\mathbb R}.$ Real strata of dimension $|T|$ correspond to a surjections $\rho\colon I \twoheadrightarrow T.$ We denote by $i\colon \mathbb A^I_{\mathbb R}\longrightarrow \mathbb A^I$ embedding of $\mathbb A^I_{\mathbb R}$ as $x$-axes. A category of constructible sheaves smooth with respect to $\mathcal S_{\emptyset,\mathbb R}$ will be denoted by $\mathcal Const(\mathbb A_{\mathbb R}^I,\mathcal S_{\emptyset,\mathbb R}).$ In \cite{KaSh} it was proved that a functor $i^!\colon \mathcal M_{un}^B(\mathbb A^I,\mathcal S_{\emptyset})\longrightarrow \mathcal Const(\mathbb A_{\mathbb R}^I,\mathcal S_{\emptyset,\mathbb R})$ is an exact functor. Moreover using this functor \textit{explicit description} of a category  $\mathcal M_{un}^B(\mathbb A^I,\mathcal S_{\emptyset})$ was given in \textit{ibid}. A category of perverse sheaves was identified with a category of double representations with relations of a certain quiver $\Gamma_{I,\mathbb R}$ (which is defined as a quiver $\Gamma_I$\textit{mutatis mutandis} with respect to a real diagonal stratification of $\mathbb A^I_{\mathbb R}.$).

\begin{remark} Denote by $\Delta_{\mathbb R}$ the unique closed minimal stratum of $\mathcal S_{\emptyset, \mathbb R}.$ For every $I$ we define the following functor:
$$
\omega^B_I\colon \mathcal M_{un}^B(\mathbb A^I,\mathcal S_{\emptyset})\longrightarrow Vect^f,\qquad \omega^B_I:=\Gamma(\Delta_{\mathbb R},i^!-).
$$
The equivalence from \textit{ibid} can be though as an instance of \textit{Tannakian duality} with respect to a functor $\omega^B_I.$

\end{remark}

By $Sym^{|I|}(\mathbb A)$ we denote a geometric quotient of $\mathbb A^I$ for the natural symmetric group action. Geometric quotient $Sym^{|I|}(\mathbb A)$ can be naturally identified with a space $Div^{|I|}(\mathbb A)$ of \textit{effective Cartier divisors} on $\mathbb A$ of degree $|I|.$ The diagonal stratification on $\mathbb A^I$ induces a stratification on $Sym^{|I|}(\mathbb A),$ which by the abuse of notation will be denoted by the same symbol $\mathcal S_{\emptyset}.$ From point of view of Cartier divisors elements in $\mathcal S_{\emptyset}$ can be identified with \textit{partitions}. By $\mathcal M^B(Sym^{|I|}(\mathbb A),\mathcal S_{\emptyset})$ we denote a category of perverse sheaves on $Sym^{|I|}(\mathbb A)$ smooth with respect to a diagonal stratification.

\par\medskip

We also consider a space $Sym^n(\mathbb A_{\mathbb R}):=\mathbb A_{\mathbb R}^n/ \Sigma_n.$ This space also has a real diagonal stratification denoted by $\mathcal S_{\emptyset,\mathbb R}.$ Elements of this stratification can be identified with \textit{ordered partitions}. Following \cite{Shuffle} a poset of ordered partitions, equipped with a partial order of reverse inclusion of subsets will be denoted by $\mathbf 2^{n-1}.$  The corresponding category of combinatorial (constructible) sheaves on $Sym^n(\mathbb A_{\mathbb R})$ will be denoted by $\mathcal Const(Sym^n(\mathbb A_{\mathbb R}),\mathcal S_{\emptyset,\mathbb R}).$ By $Rep(\mathbf 2^{n-1})$ we will denote a category of functors from $\mathbf 2^{n-1}$ to the category $Vect^f.$ Note \cite{KaSh} that a category $\mathcal Const(Sym^n(\mathbb A_{\mathbb R}))$ is equivalent to a category $Rep(\mathbf 2^{n-1}).$

\par\medskip
For every ordered partition $(n_1,\dots,n_k)$ of $n$ we have the corresponding embedding of strata: $\Delta^{(n_1,\dots,n_k)}\colon Sym^{k}(\mathbb A_{\mathbb R})\rightarrow Sym^{n}(\mathbb A_{\mathbb R})$ and the corresponding pair of adjoint functors:
\begin{equation}\label{ppq}
\Delta^{(n_1,\dots,n_k)}_*\colon \mathcal Const(Sym^{k}(\mathbb A_{\mathbb R}),\mathcal S_{\emptyset,\mathbb R})\longleftrightarrow \mathcal Const(Sym^{n}(\mathbb A_{\mathbb R}),\mathcal S_{\emptyset,\mathbb R})\colon \Delta^{(n_1,\dots,n_k)!}
\end{equation}
between constructible sheaves. We have the corresponding functors between combinatorial data:
\begin{equation}\label{q2}
\Delta^{(n_1,\dots,n_k)}_*\colon Rep(\mathbf 2^{k-1})\longleftrightarrow Rep(\mathbf 2^{n-1})\colon \Delta^{(n_1,\dots,n_k)!}
\end{equation}
Note that we have the following commutative square:
\begin{equation}\label{d1}
\begin{diagram}[height=2.2em,width=4.5em]
\mathbb A^I_{\mathbb R} &   \rTo^{i}   &   \mathbb A^I  &  \\
\dTo_{p_{\mathbb R,I}} & &  \dTo_ {p_I}  && \\
Sym^{|I|}(\mathbb A_{\mathbb R}) &   \rTo^{i_{\sigma}}   &  Sym^{|I|}(\mathbb A)    \\
\end{diagram}
\end{equation}
Where $p_I$ and $p_{\mathbb R,I}$ are projection maps and $i_{\sigma}$ is the induced map between quotient. Note that map $p_I$ is finite flat and surjective.
\par\medskip
Let $\pi\colon J\twoheadrightarrow I$ be a surjection with a corresponding ordered partition $|\pi|:=(|\pi^{-1}(i)|)_I.$ Then we have the following commutative square:
\begin{equation}\label{d2}
\begin{diagram}[height=2.4em,width=5em]
\mathbb A^I_{\mathbb R} &   \rTo^{\Delta^{(\pi)}}   &   \mathbb A^I_{\mathbb R}  &  \\
\dTo_{p_I} & &  \dTo_ {p_J}  && \\
Sym^{|I|}(\mathbb A_{\mathbb R}) &   \rTo^{\Delta^{|\pi|}}   &  Sym^{|J|}(\mathbb A_{\mathbb R})    \\
\end{diagram}
\end{equation}
It is easy to see that this square is fibered.

\subsection{Factorizable sheaf $\Omega(A)$} In this subsection out of a conilpotent bialgebra we built a certain factorizable sheaf.

\begin{Prop} With every conilpotent bialgebra $A$ and for every $n$ we associate perverse sheaf $\EuScript E(A)_n$ on $Sym^n(\mathbb A.)$

\end{Prop}

\begin{proof} We essentially follow a construction from \cite{Shuffle} which goes \textit{mutatis mutandis} in a conilpotent case. This construction consists of the following steps:

\par\medskip

We have a natural map $Sym^n(\mathbb A)\longrightarrow Sym^n(\mathbb A_{\mathbb R}),$ which sends a complex tuple to its imaginary part $\mathfrak I\colon Sym^n(\mathbb A)\longrightarrow Sym^n(\mathbb A_{\mathbb R}).$ Following \textit{ibid} we will call a preimage of a real diagonal stratification an \textit{imaginary stratification}. We denote by $S^{\mathfrak J}_{\alpha}:=\mathfrak I^{-1}(K_{\alpha}),$ and $j_{\alpha}\colon S^{\mathfrak J}_{\alpha}\hookrightarrow Sym^n(\mathbb A).$ Then for every $n$ and $\alpha\in \mathbf 2^{n-1}$ we have a sheaf $\widetilde{\EuScript E}^{\alpha}(A)$ on $S_{\alpha}^{\mathfrak I}.$ Stalk of this sheaf at the point $x=\Sigma \lambda_i x_i$ is identified with $A_{\lambda_1}\otimes \dots \otimes A_{\lambda_k}$ and generalization maps are given by the appropriate component of $\overline{\Delta}.$ We omit proof and refer reader to \textit{ibid}. We call $\EuScript E^{\alpha}(A):=Rj_{\alpha*}\widetilde{\EuScript E}^{\alpha}(A)$ the $\alpha$th \textit{Cousin sheaf} of $A.$

\par\medskip

For $\alpha,\beta\in \mathbf 2^{n-1}$ we define the morphism of sheaves on $S_{\beta}^{\mathfrak I}:$
$$\delta_{\alpha,\beta}'\colon j^*_{\beta}\EuScript E^{\alpha}(A)=j^*_{\beta}Rj_{\alpha*}\widetilde{\EuScript E}^{\alpha}(A)\longrightarrow \widetilde{\EuScript E}^{\beta}(A),$$
using multiplication in a bialgebra $A,$ we again refer to \textit{ibid} for details. By an adjunction this map gives a morphism of sheaves on $Sym^n(\mathbb A):$
$$\delta_{\alpha,\beta}\colon \EuScript E^{\alpha}(A)\longrightarrow \EuScript E^{\beta}(A).$$ Thus have the following complex of sheaves:
$$
\EuScript E_n^{\hdot}(A)=\{\EuScript E^{1^n}(A)\rightarrow \summm_{l(\alpha)=n-1} \EuScript E^{\alpha}(A)\rightarrow \dots \rightarrow \EuScript E^{n}(A)\},
$$
where $l(\alpha)$ is a length of an ordered partition.
\par\medskip
The last step is to show that  $\EuScript E_n^{\hdot}(A)$ is a perverse sheaf on $Sym^n(\mathbb A),$ constructible with respect to a diagonal stratification $\mathcal S_{diag}.$ This is Proposition $4.2.16$ from \textit{ibid}.

\end{proof}

For every $I$ we put $\Omega(A)_I:=p_I^!\EuScript E(A)_{|I|}.$ Since $p$ is a finite flat and surjective map a sheaf $\Omega(A)_I$ will be perverse. Note that because $p^!$ is a submersion we have $p^!\cong p^*.$

\begin{Prop}
For every $\pi\colon J\twoheadrightarrow I$ conilpotent filtration induces natural isomorphism:
$$\varphi(\pi)\colon \Omega(A)_I \overset{\sim}{\longrightarrow}  \Delta^{(\pi)!}\Omega(A)_J,$$
compatible with compositions of surjections.

\end{Prop}

\begin{proof} By the definition of a quiver $!$-pullback \eqref{q2} and the definition of a conilpotent filtration for every surjection $\pi \colon J \twoheadrightarrow I$ we have an isomorphism:
$$
i^!_{\sigma}\EuScript E(A)_{|I|}\overset{\sim}{\longrightarrow} \Delta^{|\pi|!}i^!_{\sigma}\EuScript E(A)_{|J|}
$$
Moreover this map lifts to an isomorphism:
\begin{equation}\label{i1}
\phi(\pi)\colon i^!\Omega(A)_{|I|}\overset{\sim}{\longrightarrow} \Delta^{(\pi)!}i^!\Omega(A)_{|J|},
\end{equation}
compatible with compositions of surjections. This follows from the standard diagram chasing:
\begin{align}\label{123}
  \Delta^{(\pi)!}i^!\Omega(A)_{|J|}:=\Delta^{(\pi)!}i^!p_J^!\EuScript E(A)_{|J|}&\cong  \Delta^{(\pi)!} p_J^! i^!_{\sigma}\EuScript E(A),\quad \mbox{(square \eqref{d1})}, \\
  \Delta^{(\pi)!} p_J^! i^!_{\sigma}\EuScript E(A)&\cong p^!_I\Delta^{|\pi|!} i^!_{\sigma}\EuScript E(A),\quad \mbox{(square \eqref{d2}).}
\end{align}
For every $\pi \colon J\twoheadrightarrow I$ we also have the following morphism:
\begin{equation}\label{phi}
\Delta^{|\pi|}_*i^*_{\sigma}\EuScript E(A)_{|I|}\longrightarrow i^*_{\sigma}\EuScript E(A)_{|J|}.
\end{equation}
This morphism follows from the fact that a multiplication respects the filtration on $A.$ More precisely a functor $i^*_{\sigma}$ sends perverse sheaf $\EuScript E(A)$ to a complex of sheaves (in a $t$-structure dual to the standard one). The later complex is equivalent to a datum of cosheaf on $Sym^{n}(\mathbb A_{\mathbb R})$ associated with an algebra underlying $A$ (See \cite{Shuffle}). Then map \eqref{phi} lifts to a morphism:
$$
\Delta^{(\pi)}_*i^*\Omega(A)_{|I|}\longrightarrow i^*\Omega(A)_{|J|}
$$
Once again this follows from the standard diagram chasing:
\begin{align*}\label{123}
\Delta^{(\pi)}_*i^*\Omega(A)_{|I|}:=\Delta^{(\pi)}_*i^*p^!_I\EuScript E(A)&\cong \Delta^{(\pi)}_*i^*p_I^*\EuScript E(A),\quad \mbox{($p_I$ is submersion)},\\
\Delta^{(\pi)}_*i^*p_I^*\EuScript E(A)&\cong \Delta^{(\pi)}_* p^*_{\mathbb R,I}i^*_{\sigma}\EuScript E(A),\quad \mbox{(square \eqref{d1})}\\
\Delta^{(\pi)}_* p^*_{\mathbb R,I}i^*_{\sigma}\EuScript E(A)&\cong p^*_{\mathbb R,J}\Delta^{|\pi|}_* i^*_{\sigma}\EuScript E(A),\quad \mbox{(base change \eqref{d2})}.
\end{align*}
Hence from the quiver description of a category $\mathcal M^B(\mathbb A^I,\mathcal S_{\emptyset})$ from \cite{KaSh} isomorphism \eqref{i1} extends to an isomorphism of perverse sheaves.

\end{proof}

Note that for every $I\in Fin$ sheaf $\Omega(A)_I$ belongs to the subcategory of unipotent perverse sheaves.

\begin{remark} Just like in the case of $\EuScript D$-modules it is natural to call $\Omega(A):=\{\Omega(A)_I,\phi(\pi)\}_{Fin^{\circ}}$ an \textit{unipotent perverse on Ran space} of $\mathbb A^1.$ These objects with obvious morphisms form a category denoted $\mathcal M_{un}^{B}(Ran(\mathbb A^1)).$

\end{remark}

\begin{Prop} Unipotent perverse sheaf $\Omega(A)$ over $Ran(\mathbb A)$ has the following additional structure:

\begin{itemize}
  \item For every surjection $\pi \colon J\twoheadrightarrow I,$ bialgebra structure induces an isomorphism:
$$v_{\pi}\colon j^{(\pi)*}\Omega(A)_J \overset{\sim}{\longrightarrow} j^{(\pi)*}\boxtimes_{i\in I} \Omega(A)_{J_i}.$$
  \item Compatibilities between $v$'s and compositions of surjections.
  \item Compatibilities between maps $\phi$ and $v.$
\end{itemize}

\end{Prop}

\begin{proof} For every $n$ and an ordered partition $(n_1,\dots,n_k)$ of $n$ we have an \textit{addition map}: $$a_{(n_1,\dots,n_k)}\colon Sym^{n_1}(\mathbb A)\times \dots \times Sym^{n_k}(\mathbb A) \longrightarrow Sym^{n}(\mathbb A).$$ Denote by $Sym^{(n_1,\dots,n_k)}(\mathbb A)_{disj} \hookrightarrow Sym^{n_1}(\mathbb A)\times \dots \times Sym^{n_k}(\mathbb A)$ a subspace which consists $k$-tuples of divisors with a  disjoint support. Then following \cite{Shuffle} one can show that we have natural isomorphisms:
$$
v_{(n_1,\dots,n_k)}\colon j^{(n_1,\dots,n_k)*}a^*_{(n_1,\dots,n_k)}\EuScript E(A)_n \overset{\sim}{\longrightarrow}j^{(n_1,\dots,n_k)*}\EuScript E(A)_{n_1}\boxtimes \dots \boxtimes \EuScript E(A)_{n_1}
$$
which is associative in a natural case.
\par\medskip
For very $\pi\colon J \twoheadrightarrow I$ we have the following commutative diagram:
\begin{equation}
\begin{diagram}[height=2.4em,width=5em]
U^{(\pi)} &   \rTo^{j^{(\pi)}}   &   \times_{i\in I}\mathbb A^{J_i}   &  \rTo^{id}  & \mathbb A^J   \\
\dTo_{p_{disj}} & &  \dTo_ {\times_I(p_{J_i})}  &&  \dTo_ {p_J} &&\\
Sym^{|\pi|}(\mathbb A)_{disj} &   \rTo^{j^{|\pi|}}   &  \times_I Sym^{|\pi^{-1}(i)|}(\mathbb A)  &  \rTo^{a_{|\pi|}}  & Sym^{|J|}(\mathbb A)\\
\end{diagram}
\end{equation}
Then a standard diagram chasing gives us:

\begin{align}
j^{(\pi)*}\Omega(A)_J\cong j^{(\pi)*}p^*_J \EuScript E(A)_{|J|}&\cong p_{disj}^*j^{|\pi|*}a_{|\pi|}^*\EuScript E(A)_{|J|},\\
 p_{disj}^*j^{|\pi|*}a_{|\pi|}^*\EuScript E(A)_{|J|}&\cong p_{disj}^*j^{|\pi|}\boxtimes_I \EuScript E(A)_{|\pi^{-1}(i)|}\quad \mbox{(morphism $c_{|\pi|}$)},\\
p_{disj}^*j^{|\pi|*}\boxtimes_I \EuScript E(A)_{|\pi^{-1}(i)|}&\cong p_{disj}^*j^{(\pi)*}\boxtimes_I p^*_{J_i}\EuScript E(A)_{|\pi^{-1}(i)|},\\
 &\cong j^{(\pi)*}\boxtimes_{i\in I} \Omega(A)_{J_i}.
\end{align}

Compatibilities between $v_{\pi}$ and $\phi(\pi)$ are obvious ones.
\end{proof}

\begin{remark}

We say that $\EuScript K$ in $\mathcal M_{un}^{B}(Ran(\mathbb A),\mathcal S_{\emptyset})$ is an \textit{unipotent factorizable sheaf} on $\mathbb A^1$ if conditions above are satisfied. The corresponding category will be denoted by $FS_{un}(Ran(\mathbb A^1)).$ It obvious that we have a functor:
$$\Omega\colon AlgB_c \longrightarrow {FS}_{un}(Ran(\mathbb A)),\quad \Omega\colon A\longmapsto \Omega(A).$$

\end{remark}

\subsection{Hyperbolic stalk} In this subsection we will prove that the functor $\Omega$ is an equivalence.

\begin{Prop} We have a functor: \footnote{Strictly speaking functor $\omega_{Ran}^B$ takes values in the category of bialgebras without a unit, but since the latter category is equivalent to $AlgB_c$ we will not distinguish between them.}
\begin{equation}
\omega_{Ran}^B \colon FS_{un}(Ran(\mathbb A))\longrightarrow  AlgB_c,\qquad \omega_{Ran}^B:=\lim_{Fin} \omega_I^B.
\end{equation}
which is the quasi inverse to the functor $\Omega.$
\end{Prop}

\begin{proof}

Let $\EuScript K\in FS_{un}(Ran(\mathbb A))$ then for every $I\in Fin$ we define a space $\omega_{Ran}^B(\EuScript K)_{|I|}:=\omega^B_I(\EuScript K_I)^{\Sigma_{|I|}}.$ Note that for every $\pi \colon J\twoheadrightarrow I$ we have an inclusion $\Delta_*^{(\pi)}\EuScript K_I \hookrightarrow \EuScript K_J$ which descends to an injection:
$\omega_{Ran}^B(\EuScript K)_{|I|}\hookrightarrow \omega_{Ran}^B(\EuScript K)_{|J|}$ and thus we have:
$$\omega_{Ran}^B(\EuScript K)_{1}\subset\dots  \subset\omega_{Ran}^B(\EuScript K)_{n}\subset \dots \subset \omega_{Ran}^B (\EuScript K).$$ Note that we have:
$i_{\sigma}^!\EuScript E(A)=(p_{I*}i^!p^*_I \EuScript E(A))^{\Sigma},$ hence analogically to \cite{Shuffle} generalization maps define a comultiplication:
$$
\overline{\Delta}_{|I'|,|I''|}\colon \omega_{Ran}^B(\EuScript K)_{|I|}\longrightarrow \omega_{Ran}^B(\EuScript K)_{|I'|}\otimes \omega_{Ran}^B(\EuScript K)_{|I''|},\quad |I|=|I'|+|I''|,
$$
which is coassociative.
\par\medskip
Using \textit{ibid} we also define a map:
$$
m_{|I'|,|I''|}\colon \omega_{Ran}^B(\EuScript K)_{|I'|}\otimes \omega_{Ran}^B(\EuScript K)_{|I''|} \longrightarrow \omega_{Ran}^B(\EuScript K)_{|I|},,\quad |I|=|I'|+|I''|,
$$
which is associative. Once again by \textit{ibid} operators $m_{|I'|,|I''|}$ and $\overline{\Delta}_{|I'|,|I''|}$ are compatible, therefore $\omega_{Ran}^B(\EuScript K)$ is a conilpotent bialgebra with a unit.

\end{proof}

\section{Quantization and complements}

\subsection{Quantization} Recall that we have a de Rham functor which induces \textit{Riemann-Hilbert equivalence} $RH\colon \mathcal M^{dR}_{un}(\mathbb A^I,\mathcal S_{\emptyset})\overset{\sim}{\longrightarrow} \mathcal M^{B}_{un}(\mathbb A^I,\mathcal S_{\emptyset}).$ Important property of this equivalence that $RH$ commutes with a six functor formalism and therefore induces an equivalence between unipotent $\EuScript D$-modules and unipotent perverse sheaves on $Ran(\mathbb A^1)$ and also an equivalence between categories of unipotent factorization algebras and unipotent factorizable sheaves. Now we are ready to give a construction of Lie bialgebras (resp. Hopf algebras) quantization (resp. dequantization):

\begin{Def}
We define a \textbf{quantization functor} $Q$ over $\mathbb C$ as a functor which makes the following diagram into commutative:
\begin{equation*}
\begin{diagram}[height=2.5em,width=6em]
LieB_c &   \rTo_{\sim}^{Q}   &   AlgB_c  &  \\
\dTo_{\sim}^{\Lambda} & &  \dTo_ {\sim}^{\Omega}  && \\
FA_{un}(Ran(\mathbb A^1)) &   \rTo_{\sim}^{RH}   &  FS_{un}(Ran(\mathbb A^1))    \\
\end{diagram}
\end{equation*}
\end{Def}

\begin{remark}[Associators] Recall that all proofs of Etingof-Kazhdan quantization (for example \cite{EK} \cite{Tam}) require a choice of \textit{associator}. Here we will briefly sketch how associators appear in our construction. To define a connection with Drinfeld associators let us note that categories $\mathcal M^{B}_{un}(Ran(\mathbb A^1))$ and $\mathcal M^{dR}_{un}(Ran(\mathbb A^1))$ can be equipped with certain \textit{natural tensor structures}. These tensor structures come from \textit{fusion tensor structure} on category of perverse sheaves (resp. $\EuScript D$-modules) on $\mathbb A^I.$ The definition of the fusion tensor product in the case of perverse sheaves is the following:
\par\medskip

Denote by $\Delta$ a flat in a product stratification of $\mathbb A_{\mathbb R}^I\times \mathbb A_{\mathbb R}^I,$ which corresponds to a $\mathbb A^I_{\mathbb R}.$ The inclusion of a complex diagonal $\Delta_{\mathbb C}$ can be factored as:
\begin{equation}
k\colon \Delta_{\mathbb C}=\Delta+i\Delta\hookrightarrow (\mathbb A^{I}_{\mathbb R}\times \mathbb A^{I}_{\mathbb R})+i\Delta,\quad j_{\Delta}\colon (\mathbb A^{I}_{\mathbb R}\times \mathbb A^{I}_{\mathbb R})+i\Delta\hookrightarrow \mathbb A^I\times \mathbb A^I.
\end{equation}
Then for two perverse sheaves $\EuScript K_1$ and $\EuScript K_2$ on $\mathbb A^I$ the corresponding exterior tensor product $\EuScript K_1\boxtimes \EuScript K_2$ is naturally a perverse sheaf on $\mathbb A^I\times \mathbb A^I$ smooth with respect to a product stratification. We define a \textit{new perverse sheaf} $\EuScript K_1\otimes^{*!} \EuScript K_2$ by the rule:
\begin{equation}
\EuScript K_1\otimes^{*!} \EuScript K_2:=k^*j_{\Delta}^!\EuScript K_1\boxtimes \EuScript K_2[|I|],
\end{equation}
This operation naturally extends to the symmetric unital product on $\mathcal M_{un}^{B}(\mathbb A^I,\mathcal S_{\emptyset}).$

Categories $\mathcal M^{B}_{un}(Ran(\mathbb A^1))$ and $\mathcal M^{dR}_{un}(Ran(\mathbb A^1))$ are equipped with ''fiber'' functors $\omega_{Ran}^B$ and $\omega_{Ran}^{dR}$ to vector spaces. Composing the functor $\omega^B_{Ran}$ with the Riemann-Hilbert functor we obtain two ''fiber'' functors from category of $\EuScript D$-modules on $Ran(\mathbb A).$ Then we have the following:
\begin{Conj} Denote by $\mathrm {Isom}^{\otimes}(\omega_{Ran}^{dR},\omega_{Ran}^B RH)$ scheme of tensor isomorphisms between ''fiber'' functors, then we have an isomorphism of schemes:
$$
\mathrm {Isom}^{\otimes}(\omega_{Ran}^{dR},\omega_{Ran}^B RH)\cong Ass(\mathbb C)
$$
\end{Conj}
\end{remark}
Note that equivalences $\Lambda $ and $\Omega $ can be easily defined over $\mathbb Q.$ It seems that together with methods from \cite{Drew} this allows us to extend the above conjecture to the case of rational associators.

\newpage
\bibliographystyle{alphanum}
\bibliography{tt}

\end{document}